\documentclass[aos]{imsart}

\RequirePackage{amsthm,amsmath,amsfonts,amssymb}
\RequirePackage[numbers]{natbib}
\RequirePackage[colorlinks,citecolor=blue,urlcolor=blue]{hyperref}
\RequirePackage{graphicx}

\usepackage{tikz}
\usepackage{pgfplots}
\usepgfplotslibrary{groupplots}

\RequirePackage{fix-cm}

\startlocaldefs
\newcommand{\commented}[1]{}
\newcommand{\essinf}{\mathop{\rm ess inf}}
\newcommand{\E}{\mathbb{E}}
\mathchardef\given="626A
\newcommand{\eps}{\varepsilon}
\def\e{\epsilon}
\newcommand\sumin{\sum_{i=1}^n}
\newcommand\sumjn{\sum_{j=1}^n}
\newcommand{\Sob}{H}
\newcommand{\Hil}{G}
\newcommand{\Gil}{L}
\renewcommand{\div}{\mathop{\rm div}\nolimits}
\def\dWW{\dot{\mathbb{W}}}
\def\L{\mathcal{L}}
\def\d{\delta}
\def\O{\mathcal{O}}
\def\a{\alpha}
\def\g{\gamma}
\def\b{\beta}
\def\h{\eta}
\def\t{\tau}
\def\r{\rho}
\def\k{\kappa}
\def\l{\lambda}

\def\RR{\mathbb{R}}
\def\NN{\mathbb{N}}
\def\GG{\mathbb{G}}
\def\HH{\mathbb{H}}
\def\LL{\mathbb{L}}
\def\VV{\mathbb{V}}
\def\WW{\mathbb{W}}
\def\ra{\rightarrow}
\newcommand{\m}{\mu}
\def\diam{\text{diam}}

\newcommand\lle{\lesssim}
\def\F{{\cal F}}
\def\tr{\mathop{\rm tr}\nolimits}
\def\cov{\mathop{\rm cov}\nolimits}
\def\var{\mathop{\rm var}\nolimits}
\def\generalweak#1{\ {\mathchoice{\buildrel #1\over \rightsquigarrow}%
{\raise-2pt\hbox{$\buildrel #1\over \rightsquigarrow$}}{}{}}\ }

\def\generalprob#1{\ {\mathchoice{\raise-1.5pt\hbox{$\buildrel#1\over\ra$}}
{\raise-2pt\hbox{$\buildrel#1\over\ra$}}{}{}}\ }
\def\prob{\generalprob{\scriptstyle P}}

\def\Ccor{c}
\def\tp{\intercal}
\newcommand{\ind}{\, \raise-2pt\hbox{$\stackrel{\makebox{\scriptsize ind}}{\sim}$}\, }
\newcommand{\iid}{\, \raise-2pt\hbox{$\stackrel{\makebox{\scriptsize iid}}{\sim}$}\,}
\def\argmin{\mathop{\rm argmin}}
\def\thalf{{\textstyle{\frac12}}}

\newenvironment{caseAlign}
{
	\left\lbrace
	\array{@{}l@{\quad} l@{}l@{}l@{}l@{}l@{} }
}
{\endarray \right.}

\numberwithin{equation}{section}
\theoremstyle{plain}

\newtheorem{theorem}{Theorem}[section]
\newtheorem{proposition}[theorem]{Proposition}

\newtheorem{lemma}[theorem]{Lemma}
\theoremstyle{remark}
\newtheorem{remark}[theorem]{Remark}

\newtheorem*{example}{Example}

\endlocaldefs

\begin{document}
	
	\begin{frontmatter}
		\title{Linear methods for  non-linear inverse problems}
		\runtitle{Linear methods for non-linear inverse problems}
		
		\begin{aug}
			\author[A]{\fnms{Geerten} \snm{Koers}\thanksref{t1}},
			\author[B]{\fnms{Botond} \snm{Szab\'o}\thanksref{t2}\ead[label=e2]{botond.szabo@unibocconi.it}}
			\and
			\author[A]{\fnms{Aad} \snm{van der Vaart}\thanksref{t1}\ead[label=e3]{a.w.vandervaart@tudelft.nl}}
			\address[A]{Delft University of Technology, DIAM, \printead{e3}}
			\address[B]{Bocconi University, Department of Data Sciences and BIDSA, \printead{e2}}

\thankstext{t1}{The research leading to these results  is partly financed by a Spinoza prize 
by the Netherlands Organisation for Scientific Research (NWO).}
\thankstext{t2}{Co-funded by the European Union (ERC, BigBayesUQ, project number:
101041064). Views and opinions expressed are however those of the author(s) only and do not
necessarily reflect those of the European Union or the European Research Council. Neither the
European Union nor the granting authority can be held responsible for them.}
		\end{aug}
		
		\begin{abstract}
We consider the recovery of an unknown function $f$ from a noisy observation of the solution
$u_f$ to a partial differential equation that can be written in the form $\L u_f=c(f,u_f)$,
for a differential operator $\L$ that is rich enough to recover $f$ from $\L u_f$. 
Examples include the time-independent Schr\"odinger equation $\Delta u_f = 2u_ff$,
the heat equation with absorption term $(\partial_t  -\Delta_x/2) u_f=fu_f$, and the Darcy problem
$\nabla\cdot (f \nabla u_f) = h$.
We transform this problem into the linear inverse problem of recovering $\L u_f$ under the Dirichlet boundary condition, and show that Bayesian methods 
with priors placed either on $u_f$ or $\L u_f$ for this problem yield optimal recovery rates not only 
for $u_f$, but also for $f$. We also derive frequentist coverage guarantees for the corresponding Bayesian credible sets. Adaptive
priors are shown to yield adaptive contraction rates for $f$, thus eliminating the need to know the smoothness of this function. 
The results are illustrated by numerical experiments on synthetic data sets.
		\end{abstract}
		
		
		
	\end{frontmatter}
	
\section{Introduction}\label{sec: introduction}
Consider the recovery of a function $f: \O\to\RR$ on a known domain $\O\subset\RR^d$ 
from a noisy observation of the function $u_f$ satisfying  a partial differential equation (PDE) of the form
\begin{equation}
\label{eq: general PDE}
    \left\{\begin{aligned}
		 \L u_f &= c(f,u_f), \qquad &&\text{ on } \O, \\
	         u_f &= g, \qquad &&\text{ on } \Gamma\subseteq\partial \O.
	\end{aligned}\right.
\end{equation}
Here  $\L$ is a linear differential operator and $g: \Gamma\to\RR$ is a known function
on  a subset $\Gamma$ of the boundary $\partial \O$ of $\O$. 
The map $c$ transforms the pair of functions $(f, u_f)$ into a function on $\O$.
In most of our examples this map takes the form of a pointwise transformation, such as
$c(f,u_f)(x)=\bar c\bigl(f(x),u_f(x)\bigr)$, for a map $\bar c: \RR^2\to\RR$, but $c$ may also act on $(f,u_f)$ as elements
of given function spaces (and hence also involve derivatives of $f$ or $u_f$). 
The idea of \eqref{eq: general PDE} is to isolate possible non-linearities in the PDE in the term $c(f,u_f)$,
and to choose the linear transformation $\L u_f$ rich enough to be able to recover the function $f$ from $\L u_f$. 
We make the latter precise by assuming that there is a (sufficiently regular) solution map $e$ such that 
\begin{equation}
f=e(\L u_f).
\label{EqSolutionOperator}
\end{equation}
In several examples the map $e$ takes a simple, pointwise form, which makes our
general method very easy to implement. More generally, the map $e$ may act on $\L u_f$ 
as an element of a function space, and practical implementation  may involve a numerical
implementation of the recovery of $f$ from $\L u_f$. 

The general idea of the paper is to solve the linear inverse problem of recovering 
$\L u_f$ from the observations by the Bayesian method, and next recover $f$ by an implementation of
\eqref{EqSolutionOperator}. We prove that the method can attain optimal results, 
and confirm the feasibility of the approach also by numerical experiments. The two-step approach 
has both computational and theoretical advantages.

The observational model consists of observing the function $u_f$ plus Gaussian white noise, which we write informally
in the form
\begin{align}
\label{eq: observation}
	Y_n = u_f + \frac{1}{\sqrt n} \dWW.
\end{align}
We consider two precise versions of this model. In the white noise model the observation is 
the Gaussian process $\bigl(Y_n(h): h \in L_2(\O)\bigr)$, given by 
$$Y_n(h)=\langle h, u_f \rangle_{L_2}+\frac1{\sqrt n}\dWW(h),$$ 
for $\dWW=\bigl(\dWW(h): h \in L_2(\O)\bigr)$ the mean zero Gaussian process with covariance function
$\cov\bigl(\dWW(h_1), \dWW(h_2)\bigr) = \langle h_1, h_2 \rangle_{L_2}$.
Next to this continuous interpretation of \eqref{eq: observation}, 
we also consider the practically relevant discrete observational setting, 
where $Y_n$ is observed at $n$ given design points $x_{n,1},\ldots, x_{n,n}\in \O$. In this case the observation $Y_n$
is a vector in $\RR^n$ with coordinates $Y_n(i)=u_f(x_{n,i})+W_i$, where $W_1,\ldots, W_n$ are
i.i.d.\ standard normal variables. In our general discussion we shall use \eqref{eq: observation} to refer to both models.

Examples of the structure \eqref{eq: general PDE}
include the time-independent Schr\"odinger equation, 
considered in \cite{Nickl18,Monardetal2021, nickl2020convergence} with $\L=\Delta$ the Laplacian,
the heat equation with absorption term, considered in \cite{kekkonen2022consistency,koers2023misspecified},
with $\L=-\Delta_x/2+\partial_t$, 
the one-dimensional Darcy problem with $\L$ the inverse Volterra operator,
and the general Darcy problem, considered in \cite{GiordanoNickl2020}, with $\L$ equal to the Laplacian. 
Details for these examples and additional examples are given in Sections~\ref{sec: Schrodinger}-\ref{SectionExponentialVolterra}. 
It is not clear that the approach works in general, but the list of examples
is encouraging.

A non-homogeneous boundary condition ($u_f=g$ on $\Gamma$ with $g$ nonzero)
still renders the problem of recovering 
$\L u_f$ non-linear. To remove also this non-linearity,
define $K$ as the solution operator $u \mapsto Ku$ of the 
homogeneous partial differential equation
\begin{equation}\label{eq: definition K gen}
	\left\{\begin{aligned}
		\L Ku &= u, \qquad &&\text{ on } \O,\\
		Ku&= 0, \qquad &&\text{ on } \Gamma\subseteq\partial\O.
	\end{aligned}\right.
\end{equation}
Thus the linear operator $K$ is the inverse of $\L$ under the Dirichlet boundary condition.
(For our purposes it suffices that $K$ is defined and satisfies \eqref{eq: definition K gen}  on
the smaller set of functions $\{\L u: u=0 \text{ on }\Gamma\}$.)
Furthermore,  for the given boundary function $g: \Gamma\to\RR$ in \eqref{eq: general PDE}, 
let $\tilde{g}: \O \to \RR$ be the function that satisfies 	
\begin{equation}\label{eq: definition g tilde gen}
	\left\{\begin{aligned}
		\L \tilde{g} &= 0, \qquad &&\text{ on } \O,\\
		\tilde{g} &= g, \qquad &&\text{ on } \Gamma\subseteq\partial \O.
	\end{aligned}\right.
\end{equation}
Solutions to the equations \eqref{eq: definition K gen}-\eqref{eq: definition g tilde gen}
exist and are unique under mild regularity conditions on the  operator $\L$, 
domain $\O$, boundary set $\Gamma$ and function $g$.
In particular, suppose that the solution to the problem to \eqref{eq: definition g tilde gen} with zero boundary condition
is unique
(i.e.\ the function $v=0$ is the only solution to the problem $\L v=0$ on $\mathcal {O}$ and $v=0$ on $\Gamma$).
Because the definitions \eqref{eq: definition K gen} and \eqref{eq: definition g tilde gen}
imply that $\L(K\L u_f+\tilde g)=\L K(\L u_f)+0=\L u_f$ on $\O$,
and $K\L u_f + \tilde{g}=0+g$ on $\Gamma$,
 it then follows  that 
\begin{equation}\label{Equ=KLu+tildeg}
u_f = K\L u_f + \tilde{g}.
\end{equation}
Since the function $\tilde g$ is known, this and \eqref{eq: observation} leads to the adapted observational model 
\begin{align}\label{def:model:linear}
\tilde{Y}_n := Y_n - \tilde{g} = K(\L u_f) + \frac{1}{\sqrt{n}}\dWW.
\end{align}
Thus the problem of estimating $\L u_f$ has been reduced to solving the linear inverse problem
$\tilde Y_n=Kv+n^{-1/2}\dWW$, defined by the operator $K$. The solution $v$ to this problem 
is next substituted as $v=\L u_f$ in the inversion map \eqref{EqSolutionOperator} to find $f$. In the Bayesian approach,
the induced distribution of $f$ when this map is applied to a posterior distribution of $v$ is the usual posterior distribution
of $f$ for the given prior.  (We give a  precise statement in Section~\ref{sec: general method}.)

It is natural to put a prior distribution directly on the function $f$, which is
the basic parameter. Within the context of non-linear inverse problems this approach is taken 
in \cite{Nickl18,GiordanoNickl2020,Monardetal2021,MR4230066,kekkonen2022consistency,Nickl23}, 
who put a Gaussian process prior on $f$ or $\log f$ or independent priors 
on the coefficients of an orthonormal wavelet basis of the domain.
However, in our two-step approach, which uses the Bayesian method to recover $\L u_f$ in the observational model \eqref{def:model:linear}, 
it is easier to put a prior on $\L u_f$ or $u_f$. Because  $f$ is determined
uniquely by $\L u_f$  in view of \eqref{EqSolutionOperator}, this induces a prior (and posterior) distribution on $f$. 
Methods for computing the posterior distribution of $\L u_f$ given a prior on this function, 
as well as theoretical results on this posterior distribution
are  available (e.g.\ \cite{Knapik2011,KnapikHeat,Ray2013,agapiou:2013,Knapik2016,Yan2020,Yan2020thesis,Yan2024}). 
A main purpose of the present paper is to show 
that natural priors on $\L u_f$ can lead to accurate recovery also of the function $f$. Nothing
appears to be lost at the level of contraction rates of the posterior distribution, and
our numerical experiments also show satisfactory performance. 

The posterior distribution is the conditional distribution given the data $Y_n$, as usual,
where the distribution of the data given the parameter is determined by \eqref{eq: observation}.
We denote all posterior distributions by $\Pi_n(\cdot \given Y_n)$, where the argument
$\cdot$ can be specialised to a set concerning $u_f$, $\L u_f$  or $f$.

We obtain posterior contraction rates, as usual (\cite{GGvdV2000}), in the non-Bayesian setup
where the observation is assumed to be generated according to the model
\eqref{eq: general PDE} and \eqref{eq: observation} with a fixed parameter $f_0$. For instance,
a contraction rate for $f$ relative to a given norm $\|\cdot\|$ is a sequence $\e_n\rightarrow 0$
such that $\Pi_n\bigl(f: \|f-f_0\|\ge M_n\e_n\given Y_n\bigr)\rightarrow 0$ in probability, for every
$M_n\ra\infty$.
This rate can be compared to an optimal contraction rate, which is determined, 
in a  minimax sense, by the smoothness of the
function $f_0$. The optimal contraction rate is typically obtainable by using a prior that matches the smoothness of the true parameters. 
By mixing over priors of different smoothness levels (hierarchical Bayes) or using
a prior of data-determined smoothness level (empirical Bayes),
optimal contraction rates for a range of different smoothness levels are obtainable by a single prior.
An advantage of our approach is that such ``adaptive'' contraction rates for the linear inverse problem of estimating
$\L u_f$ carry over into adaptive rates for the recovery of $f$. 
Adaptive priors for the linear problem have been 
constructed by using a scale of priors of fixed smoothness simultaneously through empirical or hierarchical Bayes 
methods (e.g.\ \cite{Knapik2016,Yan2020,Szabo2013,rousseau2017asymptotic,agapiou:savva:2024}). 

Besides for estimating the parameter at a (nearly) minimax rate, a posterior distribution can be 
used for uncertainty quantification. This can be justified, or not, from a non-Bayesian point of view by 
the coverage of credible sets, i.e.\ the probability that a (central) set of given posterior probability 
covers the true parameter $f_0$. 
For infinite-dimensional credible sets that involve a bias-variance trade-off, such as balls and bands, this was 
studied for linear inverse problems in \cite{Knapik2016,SzabovdVvZ2015,SniekersvdV2020,rousseau2020asymptotic}. Coverage is different
for priors of fixed smoothness and adaptive priors, and depends on properties of the true parameter
relative to the prior. Since our method is based on transforming the non-linear problem at hand into a linear one, 
the positive results showing the existence of reasonable credible sets in the linear problem imply such
results for the non-linear problem of interest (see Section~\ref{sec: general method}). 

The coverage of credible intervals for smooth functionals of the parameter
or of credible balls in a weak norm follow from a suitable version of
the Bernstein-von Mises theorem (\cite{Castillo2012,CastilloNickl2014,Ray2017,Nickl18,Monardetal2021}).
It is an open question whether such results are valid for the main construction in the present paper.

The approach of this paper is partly motivated by computational efficiency.  The recent paper \cite{nickl:wang:2020} derived
a Langevin type algorithm to compute the posterior mean, which provenly takes only polynomial time.
The algorithm is based on only the forward map $f\mapsto u_f$, which is an advantage for many problems.
However, repeatedly evaluating the forward map
may still be heavy for some practical applications in case of medium or large data sets. 
A gain of our approach is that the linear problem, to compute $\L u_f$,
may be substantially faster to solve than the original non-linear one. If the inverse map \eqref{EqSolutionOperator} is easy
(as it is in the case of most of our examples), then this advantage is propagated to the
computation of $f$.  A further gain is that our approach  may combine well with distributed computation 
based on partitioning the data in space. Since the
observation is a noisy version of the function $u_f: \O\to\RR$, a (possibly overlapping) partition of
the domain $\O$, followed by separate reconstruction of $\L u_f$ on the partitioning sets, is feasible. 
Preliminary work shows that such a distributional approach may give theoretically optimal and practically competitive 
results compared to other computational shortcuts. In particular, spatial partitioning allows priors and posteriors
to adapt to the smoothness of the true function  (for the regression problem without inversion, see \cite{szabo2023adaptation}).
In comparison, an approach based on partitioning the domain
of the function of $f$ seems impossible, due to the structure of the partial differential equation.  Another  approach to speed up the computations is 
the extended version of the inducing variable variational Bayes approach \cite{titsias2009variational}. 
By optimally choosing the number of inducing variables, this procedure recovers the functional parameter of interest 
with the optimal minimax rate \cite{randrianarisoa2023variational}  and (in case of the direct regression model) 
provides reliable uncertainty quantification \cite{nieman2023uncertainty,travis2024pointwise}.

The paper is organised as follows. 
In Section~\ref{sec: general method} we formalise the relation between the non-linear inverse problem of
interest and the auxiliary linear inverse problem. Next in Section~\ref{SectionLinearProblem} we review and extend Bayesian methods for
the linear inverse problem, describing prior choices, with both fixed and data driven choices of the tuning parameter,
and resulting results on contraction rates and frequentist coverage of credible sets.
In Sections~\ref{sec: Schrodinger}--\ref{SectionExponentialVolterra}
we apply our approach to five examples of PDEs.  In Section~\ref{sec: simulations} we present numerical simulations
that illustrate our approach in various settings. 
We conclude with a discussion of the approach and open questions in Section~\ref{sec:discussion}.
Sections~\ref{SectionHilbertAndSobolev}--\ref{SectionComplementsDarcy} provided
technical complements. 
Sections~\ref{sec:contract:smoothness:GWN} and~\ref{sec:contract:smoothness:regression} 
presents posterior contraction rates in smoothness norms for the linear inverse problem in the Gaussian white noise, 
and the discrete observation settings, respectively.

\subsection{Notation}
For two sequences $a_n,b_n$ we write $a_n\lesssim b_n$ if there exists a constant $C>0$ such that $a_n/b_n\leq C$, for every $n$. 
We write $a_n\asymp b_n$ if $a_n\lesssim b_n$ and $b_n\lesssim a_n$ hold simultaneously.  
The letters $c$ and $C$ denote constants not depending on $n$; their values might change from line to line. 

A \emph{domain} is an open set $\O\subset\RR^d$, which  throughout the paper we shall take to be bounded and connected
with a Lipschitz boundary.

For $\b\in\RR$ the notation $\Sob^\b(\O)$ is the Sobolev space $W^{\b,2}(\O)$ of functions (or distributions)  $f: \O\to \RR$ of smoothness $\b$. 
For $\b=0$ this is $L_2(\O)$, while for $\b \in \NN$ this consists of the functions  $f: \O\to\RR$
with weak derivatives $D^\a f$ up to order $|\a|\le \b$, normed by
\begin{align}
	\|f\|_{\Sob^\b(\O)} := \sqrt{\sum_{0\le |\a| \leq \b} \|D^{\a} f\|_{L_2}^2} < \infty.\label{def:Sobolev}
\end{align}
Sobolev spaces of non-integer smoothness $\b \notin \NN$ can be constructed in various equivalent ways
(\cite{Evans10,Triebel2008EMS,Hitchhiker}). Sobolev spaces with negative smoothness are by definition
the dual spaces of the corresponding spaces of positive smoothness. All spaces refer to a given
domain $\O$, which may be dropped from the notation.

\section{General method}\label{sec: general method}
In this section we present the skeleton of our approach, specialised to rates of contraction and uncertainty quantification.
In the next sections, we verify the generic assumptions of the present section for combinations of priors and  PDE models.

We assume given a prior distribution on the function $v=\L u_f$, and consider the two posterior distributions
corresponding to the models \eqref{def:model:linear} and \eqref{eq: observation}:
\begin{itemize}
\item[(L)] the posterior $\tilde\Pi_n(v\in \cdot\given \tilde Y_n)$ of $v$ in the observational model $\tilde Y_n=Kv+n^{-1/2}\,\dWW$;
\item[(N)] the posterior $\Pi_n(f\in \cdot\given Y_n)$ of $f$ in the observational model $Y_n=u_f+n^{-1/2}\,\dWW$.
\end{itemize}
The two models are linked through the identity $v=\L u_f$ and equations  \eqref{EqSolutionOperator} and \eqref{Equ=KLu+tildeg}.
Our approach is to construct the posterior distribution
of $v$ from model (L) and data $\tilde Y_n$, and next  form the posterior distribution $\tilde\Pi_n\bigl(e(v)\in \cdot\given \tilde Y_n\bigr)$ 
of $e(v)$ induced by the solution operator $e$ given in \eqref{EqSolutionOperator}. The latter equation shows that this
is a posterior distribution of $f$ in model (N) under the identification $\tilde Y_n=Y_n-\tilde g$.

To be more precise, a prior distribution on $v$ that gives full probability to the set $\{\L u_f: f\in\F\}$, for some set $\F$, induces a prior on
$f\in \F$ through \eqref{EqSolutionOperator}. The induced posterior distribution $\tilde\Pi_n\bigl(e(v)\in \cdot\given \tilde Y_n\bigr)$ of $e(v)$ is the
ordinary posterior distribution $\Pi_n(f\in \cdot\given Y_n)$ of $f$ under this prior
(see the proof of Proposition~\ref{prop:stability} for a precise statement and derivation of this correspondence).
In particular, and trivially,  this applies to  priors on $v=\L u_f$ induced from a prior on $f$. 

For flexibility we also allow prior distributions on $v$ that are not restricted to the range of the map $f\mapsto \L u_f$, 
requiring only that the inverse map $e$:
\begin{itemize}
\item[(a)] is defined on a set of full probability under the prior of $v$ and
\item[(b)] recovers $f$ as in \eqref{EqSolutionOperator} when applied to a function $\L u_f$. 
\end{itemize}
The posterior distribution $\tilde\Pi_n\bigl(e(v)\in \cdot\given \tilde Y_n\bigr)$ 
may then involve functions $f=e(v)$ for which existence of the forward solution $u_f$ 
to the partial differential equation \eqref{eq: general PDE} may be hard to verify, or does not  exist.
In particular, PDE theory may ensure the existence of a solution $u_f$ of \eqref{eq: general PDE} only under certain conditions on $f$.  

For our purposes it is not important to verify such conditions; instead we can extend the definition of the forward map $f\mapsto u_f$ as follows,
keeping the basic identities.
For a given function $v$ in model (L), the function $u=Kv+\tilde g$ is well defined  and satisfies $v=\L u$ on $\O$ and $u=g$ on $\partial \O$,
by \eqref{eq: definition K gen}-\eqref{eq: definition g tilde gen}. If $v=\L u_f$ for a function $f$ and solution $u_f$ of \eqref{eq: general PDE},
then $u=u_f$ by \eqref{Equ=KLu+tildeg} and $f=e(\L u_f)$ by  the definition \eqref{EqSolutionOperator} of the solution operator. In the other case, if $v$ does not belong to the
range of the map $f\mapsto \L u_f$, then we \emph{define} $f=e(v)$ and $u_f=u$ and 
again \eqref{EqSolutionOperator} and \eqref{Equ=KLu+tildeg} hold, even if  \eqref{eq: general PDE} may not.
(It does if $c\bigl(e(v),Kv+\tilde g\bigr)=v$ on $\O$.)
An alternative to extending the definition of $u_f$ in this way is to consider the posterior distribution
$\tilde\Pi_n\bigl(e(v)\in \cdot\given \tilde Y_n\bigr)$  as a \emph{projection-posterior} for $f$ in the sense of \cite{moumita},
but this disregards that it is an ordinary posterior distribution in most cases.

The observations $\tilde Y_n$ and $Y_n$ may be interpreted in terms of the white noise model, or be understood as
vectors in $\RR^n$, for given $x_{n,1},\ldots, x_{n,n}\in\O$, equal to the sum of the 
vectors $\bigl(Kv(x_{n,i}): i=1,\ldots, n\bigr)$ and $\bigl(u_f(x_{n,i}): i=1,\ldots, n\bigr)$  and a 
standard normal vector, in the linear and non-linear model, respectively.

\subsection{Posterior contraction rate}
We assume that the posterior distribution of $v$ is a Borel law on a normed space $(V,\|\cdot\|)$ and
possesses a rate of contraction $\e_n$ to $v_0$ in the sense that $\tilde\Pi_n(v: \|v-v_0\|\ge \e_nM_n\given \tilde Y_n)\prob 0$,
for any $M_n\ra\infty$, under the assumption that $v_0$ gives the true distribution of $\tilde Y_n=Kv_0+n^{-1/2}\dWW$.

\begin{proposition}\label{prop:stability}
Suppose that the posterior distribution $\tilde \Pi_n(v\in\cdot\given \tilde Y_n)$ of $v$ in model (L) contracts under $v_0$
to $v_0=\L u_{f_0}$ at rate $\e_n$ in $(V,\|\cdot\|)$ 
and satisfies $\Pi_n(v\in V_n\given \tilde Y_n)\prob 1$ for given sets $V_n\subset V$. 
If  \eqref{EqSolutionOperator} holds for a map $e$ such that $e: V_n\to L_2$ is Lipschitz at $v_0$, 
then the posterior distribution of
$f$ in model (N) attains a rate of contraction $\e_n$ under $f_0$ relative to the $L_2$-norm. 
\end{proposition}

\begin{proof}
By assumption the prior distributions of $\L u_f$ and $v$ are the same, and in view of 
\eqref{def:model:linear} and \eqref{Equ=KLu+tildeg}, the conditional distributions of 
$Y_n-\tilde g=K\L u_f+n^{-1/2}\dWW$ given $\L u_f$ and $\tilde Y_n=Kv+n^{-1/2}\dWW$ given $v$ are the same as well,
under the identification $v=\L u_f$.
It follows that the conditional distributions  of $\L u_f$ given $Y_n-\tilde g$ and of $v$ given $\tilde Y_n$
are the same. More precisely, if $(y,B)\mapsto L(y,B)$ is a Markov kernel such that
$v\given \tilde Y_n\sim L(\tilde Y_n,\cdot)$, then $\L u_f \given Y_n\sim L(Y_n-\tilde g,\cdot)$.

By \eqref{EqSolutionOperator}, we have that $f-f_0=e(\L u_f)-e(\L u_{f_0})$ and hence
$\Pi_n\bigl(f: \|f-f_0\|_{L_2}<\e, \L u_f\in V_n\given Y_n\bigr)
=\Pi_n\bigl(f: \|e(\L u_f)-e(\L u_{f_0})\|_{L_2}<\e, \L u_f\in V_n\given Y_n\bigr)$. 
By the preceding paragraph this is the same as 
$L(Y_n-\tilde g, B)$, for $B=\{v\in V_n: \|e(v)-e(v_0)\|_{L_2}<\e\}$. Since $Y_n-\tilde g$ is distributed
as $\tilde Y_n$ under $v_0$, we obtain that $L(Y_n-\tilde g, B)\sim \tilde\Pi_n(B\given \tilde Y_n)$.
The assumptions that $e$ is Lipschitz on $V_n$ and that $\tilde \Pi_n(V_n\given \tilde Y_n)\prob 1$ show
that there exists a constant  $C$ such that 
$\tilde\Pi_n(B\given \tilde Y_n)\le  \tilde\Pi_n(v: \|v-v_0\|\le C \e\given \tilde Y_n)$ on a set
of probability tending to 1. The last probability tends to zero in probability for $\e=M_n\e_n$ and 
every $M_n\ra\infty$.
\end{proof}

The proposition uses the $L_2$-norm on the functions $f$ for definiteness, and allows a general norm $\|\cdot\|$
on the functions $\L u_f$ for flexibility. It connects the posteriors for the two problems
\eqref{eq: general PDE} and  \eqref{def:model:linear} in an abstract way, and needs to be made precise
for particular problems. The input is a $\|\cdot\|$-contraction rate for $v$ in the problem 
$\tilde Y_n=Kv+n^{-1/2}\,\dWW$ at $v_0=\L u_{f_0}$  and the guarantee that
the posterior in this problem concentrates on sets $V_n$ on which the solution map \eqref{EqSolutionOperator} is 
Lipschitz. 

In our examples the Lipschitz assumption, relative to a natural norm, is usually mild. For instance, it can be verified
by ascertaining that the posterior distribution  $\tilde\Pi_n(Kv\in \cdot\given \tilde Y_n)$ of $Kv$ concentrates
on functions that are bounded away from zero, or functions whose derivative is well behaved.
We show this by establishing consistency of this posterior distribution for a relevant norm
combined with the relevant assumption on the true function $Kv_0$. Because $K$ is smoothing,
uniform consistency for $Kv$ is typically automatic for the usual priors, and often sufficient.

\subsection{Uncertainty quantification}
Credible sets for the linear inverse problem (L) map naturally into credible sets for the non-linear
problem (N) of interest. Assume given  a credible set $\tilde C_{n}(\tilde Y_n)$ for $v$ based on 
the observation $\tilde Y_n=Kv+n^{-1/2}\dWW$: a set of prescribed probability under the posterior
distribution $\tilde\Pi_n(\cdot\given \tilde Y_n)$. We transform this in the set of functions $f$ given by 
\begin{equation}
\label{EqGeneralCredibleSet}
C_{n}(Y_n):= \bigl\{f: \L u_f\in \tilde C_{n}(Y_n-\tilde g)\bigr\}=\bigl\{e(v): v\in  \tilde C_{n}(Y_n-\tilde g)\bigr\},
\end{equation}
where $e$ is the solution map given in \eqref{EqSolutionOperator}. 

The \emph{credible level} of the credible set $\tilde C_{n}(\tilde Y_n)$ for $v$ in the model $\tilde Y_n=Kv+n^{-1/2}\dWW$  
is by definition the posterior probability $\tilde\Pi_n\bigl(\tilde C_{n}(\tilde Y_n)\given \tilde Y_n\bigr)$
and its \emph{coverage} at $v_0$ is by definition the probability $\Pr_{v_0}\bigl(\tilde C_{n}(\tilde Y_n)\ni v_0\bigr)$
if $\tilde Y_n=Kv_0+n^{-1/2}\dWW$. For the model $Y_n=u_f+n^{-1/2}\dWW$  the same quantities 
are defined analogously relative to the parameter $f$ with true value $f_0$. 

In the following proposition we identify $\tilde Y_n$ and $Y_n-\tilde g$, so that
the assertions can be understood in an almost sure sense.

\begin{proposition}
\label{proposition_credible_general}
The credible levels of the credible sets $\tilde C_{n}(\tilde Y_n)$ for $v$ in model (L) and $C_{n}(Y_n)$ for $f$ in model (N), given in \eqref{EqGeneralCredibleSet}, are equal,
and so are the coverage levels of these sets at $v_0=\L u_{f_0}$ and $f_0$, respectively. 
Furthermore, if  the map  $e: V_n\to L_2$  in \eqref{EqSolutionOperator} is uniformly Lipschitz at points $\bar v_n\in V_n$ 
and $\tilde C_{n}(\tilde Y_n)\subset V_n$, then on the event $\bar v_n\in \tilde C_n(\tilde Y_n)$ 
the $L_2$-diameters of the sets $C_{n}(Y_n)$ are of the same order in probability under $f_0$ as 
the $\|\cdot\|$-diameters of the sets $\tilde C_{n}(\tilde Y_n)$ under $v_0$.
\end{proposition}

\begin{proof}
As noted in the proof of Proposition~\ref{prop:stability},
if the Markov kernel $(y,B)\mapsto L(y,\cdot)$ gives the posterior distribution of $v$ given $\tilde Y_n=y$, 
then $L(Y_n-\tilde g,\cdot)$ gives the posterior distribution of $\L u_f$ given $Y_n$. 
By the definition of $C_n(Y_n)$ it follows that
$\Pi_n\bigl(f\in C_n(Y_n)\given Y_n\bigr)=\Pi_n\bigl(\L u_f\in \tilde C_n(Y_n-\tilde g)\given Y_n\bigr)
=L\bigl(Y_n-\tilde g, \tilde C_n(Y_n-\tilde g)\bigr)$, which is the same as 
$\tilde\Pi_n\bigl(\tilde C_{n}(\tilde Y_n)\given \tilde Y_n\bigr)$, since $\tilde Y_n=Y_n-\tilde g$.
Similarly $\Pr_{f_0}\bigl(f_0\in C_n(Y_n)\bigr)=\Pr_{f_0}\bigl(\L u_{f_0}\in \tilde C_n(Y_n-\tilde g)\bigr)$
is the same as $\Pr_{v_0}\bigl(v_0\in \tilde C_n(\tilde Y_n)\bigr)$.

The $L_2$-diameter  $\sup\{\|f-g\|_{L_2}: f,g\in C_n(Y_n)\}$ of $C_n(Y_n)$ is equal to 
$\sup\{\|e(v)-e(w)\|_{L_2}: v,w \in\tilde C_n(\tilde Y_n)\}\le 2\sup\{\|e(v)-e(\bar v_n)\|_{L_2}: v\in\tilde C_n(\tilde Y_n)\} $, 
which is bounded by a constant times $\sup\{\|v-\bar v_n\|: v \in\tilde C_n(\tilde Y_n)\}$, because
$\tilde C_n(\tilde Y_n)\subset V_n$ and $e$ is uniformly Lipschitz at $\bar v_n\in V_n$. 
On the event $\bar v_n\in \tilde C_n(\tilde Y_n)$ the right side is bounded above by the
diameter of $\tilde C_n(\tilde Y_n)$.
\end{proof}

Thus credible and coverage levels translate from the linear to the non-linear inverse problem
for arbitrary credible sets. To retain also the size of the sets it may be necessary to restrict the
starting sets $\tilde C_{n}(\tilde Y_n)$ to sets $V_n$ on which the solution map
\eqref{EqSolutionOperator} is locally Lipschitz. 

In our examples the latter can be achieved without affecting credible or coverage levels.
Indeed, in general as soon as $\E_{v_0}\tilde \Pi_n( V_n\given \tilde Y_n)\ra 1$ 
and $\Pr_{v_0}( V_n\ni v_0)\ra 1$, the credible and coverage levels of the
sets $\tilde C_{n}(\tilde Y_n)$ and $\tilde C_{n}(\tilde Y_n)\cap V_n$ are asymptotically the same.
In our examples we choose the sets $V_n$ for instance a uniform ball $V_n=\{v: \|Kv-K\hat v_n\|_\infty\le c_0\}$
of small radius around estimators $K\hat v_n$ of $Kv_0$. Then consistency of $K\hat v_n$ for $Kv_0$ for the uniform norm
gives  $\Pr_{v_0}(V_n\ni v_0)\ra 1$,  and consistency of the posterior distribution of $Kv$ for $Kv_0$
relative to the uniform norm gives $\E_{v_0}\tilde \Pi_n( V_n\given \tilde Y_n)\ra 1$.

\section{Bayesian analysis in the linear problem}\label{SectionLinearProblem}
In this section we review and extend results on the problem of recovering a function $v$ from the
noisy observation $\tilde Y_n=Kv+n^{-1/2}\dWW$, where $K$ is the operator given in
\eqref{eq: definition K gen}. When applied to solving the non-linear problem \eqref{eq: observation},
the function $v$ will be taken equal to $v=\L u_f$. With a view towards this setting, we extend results from the literature by
considering contraction relative to the uniform norm and smoothness norms.
In Sections~\ref{SmoothnessScale}--\ref{SectionSobolevSpaces}
we consider the white noise model, whereas in Section~\ref{sec:discrete:obs} we consider
the discrete observational model. We restrict to mildly ill-posed problems
in the sense of \cite{cavalier2008nonparametric}. 

We consider Gaussian priors on $v$, or mixtures thereof, possibly with tuning parameters set by 
the empirical Bayes method. (Non-Gaussian priors were considered in \cite{agapiou:2013,Ray2013} and \cite{Yan2020,Yan2020thesis}.)
Without loss of generality such a prior can be constructed by equipping the coefficients of $v$
in a basis expansion with (conditionally) independent univariate normal distributions. For  
a given orthonormal basis $(h_i)$ of $L_2(\O)$ and $v=\sum_{i=1}^\infty v_ih_i$,  we set
\begin{align}\label{eq: prior}
	(v_1,v_2,\ldots)\given \t,\a\sim \bigotimes_{i=1}^{\infty} N(0, \t^2 i^{- 1 - 2\a/d}).
\end{align}
The hyper-parameters $\t$ or $\a$ determine the smoothness of the prior, and 
may be chosen fixed, given a hyper prior, or set by the empirical Bayes method.

\subsection{Smoothness scales}
\label{SmoothnessScale}
The basis $(h_i)$ give rises to a scale of function classes $\Hil^s$, for $s\ge0$, consisting
of all functions $v=\sum_{i=1}^\infty v_ih_i\in L_2(\O)$ with $\|v\|_{\Hil^s}<\infty$, for
\begin{equation}\label{EqDefHilbertSpaceNorm}
\|v\|_{\Hil^s}^2=\sum_{i=1}^\infty v_i^2i^{2s/d}.
\end{equation}
For $s<0$ we define the same norm on the sequences $(v_i)$ and define $\Hil^s$ as the corresponding
normed space; this can also be identified with the dual space of $\Hil^{-s}$. 
Assume that the operator $K:\Hil^0\to L_2(\O) $ is \emph{smoothing of order} $p$ in the scale $\Hil^s$ in the sense that
\begin{equation}
	\|K v\|_{L_2}\asymp \|v\|_{\Hil^{-p}},\qquad v\in \Hil^0.\label{eq:K:smoothing}
\end{equation}

\begin{theorem}\label{theorem: general uniform convergence}
Let $K: \Hil^0\to L_2$ be a linear operator satisfying \eqref{eq:K:smoothing}. Consider the white noise model
with the prior \eqref{eq: prior}, for fixed $\t$ and $\a$.
\begin{itemize}
\item[(i).] If $v_0\in \Hil^\b$ and $-p\le\d<\a\wedge \b$, then the $\Hil^\d$-contraction rate to $v_0$
of the posterior distribution of $v$ is of the order $n^{-(\a\wedge \b-\d)/(2\a+2p+d)}$.
\item[(ii).] If either $\sup_ii^{u/d}\|Kh_i\|_\infty<\infty$, for some $u$ with $d/2<\a \wedge\b+u$, 
or $K: \Hil^{\g-p}\to \Hil^\g$ is continuous and $\|\cdot\|_{\Sob^\g(\O)}\lesssim \|\cdot\|_{\Hil^\g}$ 
for some $\g$ with $d/2<\g<\a\wedge \b+p$, then the posterior distribution of $Kv$ is consistent for $Kv_0$
relative to the uniform norm.
\end{itemize}
\end{theorem}

\begin{proof}
The contraction rate (i) is an extension of Theorem~6.1 in \cite{Yan2020} to the $\|\cdot\|_{\Hil^\d}$-norm,
proved in Theorem~\ref{thm: contr:diff:norm}.

For the uniform consistency (ii) of the posterior distribution of $Kv$ under the first condition, we 
argue that $Kv=\sum_i v_iKh_i$, and hence $\|Kv\|_\infty^2\le \|v\|_{\Hil^\d}^2\sum_i\|Kh_i\|_\infty^2i^{-2\d/d}$, by
the Cauchy-Schwarz inequality. Under the condition on the uniform norms of the functions $Kh_i$, this
is bounded above by a multiple of $\|v\|_{\Hil^\d}^2$, for $(2\d+2u)/d>1$, equivalently $\d+u>d/2$. 
Thus the uniform consistency follows from the contraction in $\Hil^\d$, which takes place for $\d<\a\wedge\b$.
Under the condition $\a\wedge\b+u>d/2$, a value of $\d$ that satisfies both requirements exists.
Under the second condition we use Sobolev embedding to see that $\|K v\|_\infty\lesssim \|Kv\|_{\Sob^\g(\O)}$,
for $\g>d/2$, which  by assumption is bounded above by $\|Kv\|_{\Hil^\g}\lesssim \|v\|_{\Hil^{\g-p}}$. Thus the uniform
consistency follows from contraction in $\Hil^\d$ for $\d=\g-p$, which takes place if $\d<\a\wedge\b$.
\end{proof}

The preceding theorem assumes fixed hyper-parameters. Theorem~7.2 of \cite{Yan2020} shows that the mixture prior obtained by
choosing a parameter $\tau$ from an inverse Gamma distribution and a fixed $\a$ gives an $L_2$-contraction
rate $n^{-\b/(d+2\b+2p)}$, for any $\b\in (0,\a]$, but has not been extended to rates in $\Hil^\d$.

\subsection{Eigenbasis}
If $K: L_2(\O)\to L_2(\O)$ is compact, then the self-adjoint operator $K^\tp K: L_2(\O)\to L_2(\O)$ possesses
an orthonormal basis  $(h_i)$ of eigenfunctions.
If  $\k_i^2$ are the corresponding eigenvalues of $K^\tp K$, then $\|Kv\|_{L_2}^2=\langle K^\tp Kv, v\rangle_{L_2}=\sum_{i=1}^\infty v_i^2\k_i^2$.
Hence relative to the eigenbasis $K$ is smoothing of order $p$ in the sense of \eqref{eq:K:smoothing} if 
\begin{align}\label{def:kappa:mild}
	\k_i\asymp i^{-p/d}.
\end{align}
If $K$ is self-adjoint, then $K^sv=\k_i^sv$, and $K$ is also
smoothing in the extended sense that $\|Kv\|_{\Hil^s}\asymp\|v\|_{\Hil^{s-p}}$, for every $s\in\RR$.

For the prior constructed relative to the eigenbasis,  the white noise problem is equivalent to observing a 
sequence $(\tilde Y_{n,1},\tilde Y_{n,2},\ldots)$ of independent normal variables with $\tilde Y_{n,i}\given v\sim N(\k_i v_i ,1/n)$,
equipped with conditionally independent priors $v_i\given \a,\t\sim N(0,\t^2 i^{-1-2\a/d})$. Conditionally on
the hyper-parameters the problem is conjugate and the posterior can be obtained explicitly.
The contraction rates for fixed hyper-parameters follow from
Theorem~\ref{theorem: general uniform convergence},
but  contraction rates for priors with hyper-parameters set by the empirical 
and hierarchical Bayes methods were studied in \cite{Szabo2013,Knapik2016},
and credible balls and adaptive  credible balls in \cite{Knapik2011,SzabovdVvZ2015}. 
(The parameter $\a$ in \cite{Knapik2011,Szabo2013,SzabovdVvZ2015,Knapik2016} is denoted by $\a/d$ in the present
paper.) 

The hierarchical Bayes method can be implemented with fixed $\t$ and
a prior on $\a$ with density $\l$ on $(0,\infty)$ chosen such that, for every $c_1>0$ there exist
$c_2\geq 0$, $c_3\in\RR$ and $c_4>1$, with $c_3>1$ if $c_2=0$, such that for $\a\geq c_1$,
\begin{align}
c_4^{-1}\a^{-c_3}e^{-c_2\a}\leq\l(\a)\leq c_4\a^{-c_3}e^{-c_2\a}.\label{def:hyperprior}
\end{align}
The empirical Bayes method chooses $\a$ equal to the maximum likelihood estimator, restricted to $[0,\log n]$,
based on the Bayesian marginal distribution of the observation $\tilde Y_n$. This can be seen to be 
\begin{align}\label{eq: log likelihood Y given a}
	\hat\a_n =\argmin_\a  \sum_{i=1}^{\infty}\Bigl[\log\Bigl(1 + \frac{n}{i^{1+2\a/d} \k_i^{-2}}\Bigr)
-\frac{n^2}{i^{1+2\a/d}\k_i^{-2} + n} \tilde Y_{n,i}^2\Bigr].
\end{align}
Alternatively, the parameter $\a$ may be fixed and the parameter $\t$ chosen by empirical Bayes or 
equipped with a prior, see \cite{Szabo2013}.

The papers \cite{Knapik2011,SzabovdVvZ2015} studied credible balls of the type,
for $\hat{v}_n$ the posterior mean and $r_n$ set to 
the smallest value such that $\bigl\{v: \|v-\hat v_n\|_{L_2}\le r_n\bigr\}$ possesses a prescribed posterior probability in $(0,1)$,
\begin{equation}
\label{EqCredibleBall}
\tilde C_n(\tilde Y_n)=\bigl\{v: \|v-\hat v_n\|_{L_2}\le c r_n\bigr\}.
\end{equation}
In \cite{Knapik2011} priors with fixed hyper-parameters $(\t, \a)$ were shown to
give frequentist coverage provided $\a$ undershoots the true smoothness. 
As it is impossible to construct confidence sets of uniform coverage and rate-adaptive size, see e.g. \cite{cai2006adaptive,
robins2006adaptive}, neither the empirical nor the hierarchical Bayes methods can provide credible 
sets with prescribed uniform frequentist coverage over all possible parameters. 
However, in \cite{SzabovdVvZ2015,SniekersvdV2015,rousseau2020asymptotic} the credible sets $\tilde C_n(\tilde Y_n)$
are shown to be confidence sets for true parameters whose coefficients decrease not too
irregularly, so-called \emph{polished tail} sequences. 

We record these results, together with extensions of contraction in $\Hil^\d$ and uniform consistency,
in the following theorem.

\begin{theorem}\label{theorem: general uniform convergence_adaptive}
Consider the white noise model with the prior based on the eigenbasis with coefficients \eqref{eq: prior} with fixed $\t$ and $\a$ 
determined by the hierarchical or empirical Bayes method \eqref{def:hyperprior} or \eqref{eq: log likelihood Y given a}.
\begin{itemize}
\item[(i).] If $v_0\in \Hil^\b$ and $-p\le\d<\b$, 
then the $\Hil^\d$-contraction rate of the posterior distribution of $v$ is $l_nn^{-(\b-\d)/(2\b+2p+d)}$, for $l_n=(\log n)^{3/2}(\log\log n)^{1/2}$.
\item[(ii).] If either $\sup_ii^{u/d}\|Kh_i\|_\infty<\infty$, for some $u$ with $d/2<\b+u$, 
or $\|\cdot\|_{\Sob^\g(\O)}\lesssim \|\cdot\|_{\Hil^\g}$ for  some $\g$ with $d/2<\g<\b+p$, then
the posterior distribution of $Kv$ is consistent for the uniform norm.
\item[(iii).] If the prior is constructed with fixed hyper-parameters $\t$ and $\a$, then the coverage
tends to 1 provided $\a<\b$ and the $L_2$-diameter is of the order $O_P(n^{-(\a\wedge \b)/(2\a+2p+d)})$.
\item[(iv).] If $v_0$ satisfies the polished tail condition, then for sufficiently large $c$ the coverage of the 
sets \eqref{EqCredibleBall} tends to 1 and the $L_2$-diameter is of the order 
$O_P(l_nn^{-\b/(2\b+2p+d)})$.
\end{itemize}
\end{theorem}

\begin{proof}
Assertions (iii) and (iv) follow immediately from the results in \cite{Knapik2011} and \cite{SzabovdVvZ2015}.
Assertion (ii) follows from assertion (i) by the arguments in the proof of Theorem~\ref{theorem: general uniform convergence}.

For the proof of (i), we extend the results of \cite{Knapik2016}, which are based on
the explicit representation of the posterior distribution: for $\l_i(\a) = i^{-1-2\a/d}$ and $\k_i^2 \asymp i^{-2p/d}$,
the posterior distribution satisfies $v_i-v_{i,0}\given (\tilde Y_n,\a) \sim N(m_i(\a),\sigma_i^2(\a))$,  for
\begin{align*}
	m_i(\a) &:= \frac{n \l_i(\a) \k_i \tilde Y_{n,i}}{1 + n\l_i(\a) \k_i^2} - v_{0,i},
\qquad	\sigma_i^2(\a) := \frac{\l_i(\a)}{1 + n\l_i(\a) \k_i^2}.
\end{align*}
Theorem~1 in \cite{Knapik2016} shows that the empirical Bayes estimator \eqref{eq: log likelihood Y given a} is 
with  probability tending to 1 under $v_0$ bounded below and above by deterministic sequences $\underline\a_n$ and $\overline\a_n$,
characterised as crossing points of a certain deterministic criterion function $h_n$, derived from the likelihood.
In the proof of Theorem~3 in the same reference, it is shown that 
 in the hierarchical Bayes case the posterior probability that $\a$ falls in the interval $[\underline\a_n,\overline\a_n]$ tends in
probabiity to one. The lower bound satisfies $\underline{\a}_n \geq \b - {c_0}/{\log n}$, for some constant $c_0$, which ensures
that in both cases the rate of contraction does not fall essentially below the optimal rate for a $\b$-smooth parameter.

It follows that, for the proof of (i) for both methods it suffices to show, for any $M_n\ra\infty$,
\begin{align}
\sup_{\underline\a_n\le\a\le\overline\a_n} \Pi_n\bigl(v:\, \|v- v_0\|_{\Hil^\d}>M_n l_ nn^{-\frac{\b-\d}{d+2\b+2p}}\given \tilde{Y}_n,\a\bigr)
\stackrel{P_{v_0}}{\rightarrow}0.\label{eq:help:unifconv}
\end{align}
By Markov's inequality, the posterior probability can be bounded above by
$(M_n\e_n)^{-2}\E\bigl(\|v-v_0\|^2_{\Hil^\d}\given \tilde Y_n,\a\bigr)$, for $\e_n$ the rate of contraction.
The second moment can be computed using the explicit form of the posterior distribution. It follows that the preceding display
is valid if  the following two relations hold:
\begin{align*}
\sup_{\underline\a_n\le\a\le\overline\a_n}\sum_{i=1}^{\infty} m_i^2(\a) i^{2\d/d}&=O_{P_{v_0}}(l_n^2n^{-\frac{2(\b-\d)}{d+2\b+2p}}) ,\\
 \sup_{\underline\a_n\le\a\le\overline\a_n}\sum_{i=1}^{\infty} \sigma_i^2(\a) i^{2\d/d}&=O(l_n^2n^{-\frac{2(\b-\d)}{d+2\b+2p}}).
\end{align*}
The left side of the second equation is deterministic and monotonely decreasing in $\a$, and hence it is sufficient to consider the
series at $\a=\underline{\a}_n$. By elementary calculus, this can
be seen to be of order $n^{-2(\underline{\a}_n-\d)/(d+2\underline{\a}_n+2p)}\lesssim n^{-2(\b-\d)/(d+2\b+2p)}$.

The analysis of the first relation is considerably harder, as it is random and non-monotone.
However, the relation can be verified by the same arguments as in \cite{Knapik2016}. For completeness
we include the complete proof in Section~\ref{sec:EB:delta} of the supplement, 
see Theorems \ref{thm: ConvergenceEB} and \ref{thm: ConvergenceHB}.
\end{proof}

In the preceding results the proof of uniform consistency is essentially based on Sobolev embedding.
The following lemma shows that this approach is generally possible for
parabolic or elliptic differential operators.

\begin{lemma}\label{cor:elliptic_parabolic}
Let $(\Hil^s)$ be the eigenscale of the inverse operator  $K$  (as in \eqref{eq: definition K gen})
of a parabolic or elliptic differential operator $\L$ under the Dirichlet boundary condition.
Then $K: \Hil^\g\to L_\infty$ is continuous for every $\g>d$.
\end{lemma}

\begin{proof}
See Theorem~B.2 in \cite{Kinderlehrer00} for elliptic $\L$, and Theorem~4.5 in \cite{troltzsch2010optimal} 
(adapted to Dirichlet boundary condition) for parabolic $\L$.
\end{proof}

\subsection{Sobolev spaces}\label{SectionSobolevSpaces}
The scale of Sobolev spaces $\Sob^s(\O)$ is of special interest in connection to differential operators.
There are several bases that generate this scale,
popular ones being the wavelet bases. Although a wavelet basis is most 
conveniently indexed by a double index in the form $(h_{j,k})$, 
the base functions can be ordered in a sequence and the resulting scale takes the form as in 
Section~\ref{SmoothnessScale} (see Example~\ref{ExampleWaveletOrdering}).
The Gaussian prior on this sequence constructed as in Section~\ref{SmoothnessScale} takes the form
$\sum_j\sum_k v_{j,k} h_{j,k}$ in the wavelet domain, for independent variables $v_{j,k}\sim N(0, \t^22^{-2j(\a+d/2)})$.

Thus we obtain the results of Theorem~\ref{theorem: general uniform convergence} for a wavelet prior,
provided the operator $K$ is smoothing \eqref{SmoothnessScale} in the Sobolev scale. Unfortunately, this
appears to be not often true for the full Sobolev scale, due to the boundary conditions connected
to the differential operator. For instance, the inverse of the Laplacian is smoothing (of order 2) relative
to the scale $\Sob_0^1\cap \Sob^s(\O)$, where $\Sob_0^1$ is the space of weak solutions of \eqref{eq: definition K gen},
the subscript 0 arising from the Dirichlet boundary condition (see Proposition~\ref{PropositionSmoothingEllipticH0}).
In such a case the contraction rates provided by Theorem~\ref{theorem: general uniform convergence} 
are still valid for  priors constructed from wavelet bases that observe the boundary conditions.
We conjecture that an unrestricted wavelet basis would give the same results.

More abstractly, the operator $K$ will be smoothing relative to the scale generated
by the differential operator $\L$, and we may construct a Gaussian prior with covariance operator equal to
the inverse of $\L$ (see Proposition~8.19 in \cite{Engletal}). The equivalence of Hilbert scales 
resulting from linear differential operators $\L$ and Sobolev spaces has been investigated extensively, 
see for instance Theorem~10.1 in \cite{Seeley65} for elliptic differential operators $\L$.

\subsection{Discrete observations}\label{sec:discrete:obs}
The preceding results for the white noise model extend to the discrete observational scheme, in which
for given ``design points'' $x_{n,1},\ldots, x_{n,n}\in\O$, we observe  $\tilde Y_n(1),\ldots, \tilde Y_n(n)$ given by,
for  i.i.d.\ standard normal variables $Z_i$,
\begin{equation*}
	\tilde Y_n(i) =  Kv(x_{n,i}) + Z_i, \qquad i = 1, \ldots, n.
\end{equation*}
The key is an interpolation technique,
mapping discrete signals to the continuous domain, as employed in Chapter~8 of \cite{Yan2020thesis} and \cite{Yan2024}.

Denote the empirical inner product by $\langle v, w \rangle_{\LL_n} := n^{-1}\sum_{i=1}^n v(x_{n,i}) w(x_{n,i})$
and let $\|v\|_{\LL_n} := \sqrt{\langle v, v \rangle_{\LL_n}}$ be the induced semi-norm. 
Assume that for every $n \in \NN$, there exists an $n$-dimensional subspace $\LL_n\subset L_2(\O)\cap C(\O)$
such that, for constants $0<C_1<C_2<\infty$, independent of $n$, 
\begin{align}
\label{eq:regression_W_n_isomorphism}
C_1 \|v\|_{L_2} \leq \|v\|_{\LL_n} \leq C_2 \|v\|_{L_2},\qquad v\in \LL_n.
\end{align}
Let $\mathcal{I}_n: L_2(\O)\cap C(\O)\to \LL_n$ be the map such that $\mathcal{I}_nv$ is
the unique element of $\LL_n$ that interpolates $v$ at the design points, i.e.\ $\mathcal{I}_nv(x_{n,i})=v(x_{n,i})$, for every $i=1,\ldots, n$.
Then assume that for every $s$ in some interval $(s_d,S_d)$, and some sequence $\d_{n,s} \ra 0$,
\begin{align}
\label{eq:regression_stable_reconstruction}
\|\mathcal{I}_n v - v \|_{L_2}\lesssim \d_{n,s} \|v\|_{\Hil^{s}}.
\end{align}

\begin{theorem}\label{thm:regression:linear}
If \eqref{eq:regression_W_n_isomorphism}-\eqref {eq:regression_stable_reconstruction}
hold with $n\d_{n,\b}^2\ra 0$ and $s_d<\b+p<S_d$ and $K:G^\b\to G^{\b+p}$ continuous, then under
the conditions of Theorem~\ref{theorem: general uniform convergence} the assertions (i)-(ii) of the
theorem are valid also in the discrete observational model.
\end{theorem}

\begin{proof}
The assertion on the contraction rate is an extension of results of \cite{Yan2020thesis,Yan2024} to smoothness norms,
which is proved in Section~\ref{sec:contract:smoothness:regression}.

The remainder of the proof is the same as the proof of Theorem~\ref{theorem: general uniform convergence}. 
\end{proof}

\section{Schr\"odinger equation}\label{sec: Schrodinger}
For a bounded domain $\O\subset \RR^d$, a 
given non-negative  function $f \in L_2(\O)$ and  a bounded, measurable function $g: \partial\O \to \RR$,
 let $u_f$ be the solution to the time-independent \emph{Schr\"odinger equation}
\begin{equation}\label{eq: EqSchr\"odinger}
	\left\{\begin{aligned}
		\textstyle\frac12\Delta u_f &= fu_f, \qquad &&\text{ on } \O,\\
		 u_f &= g, \qquad &&\text{ on } \partial \O.
    \end{aligned}\right.
\end{equation}
The goal is to recover the unknown ``potential'' $f$ using a noisy observation of $u_f$.  

This model can serve as a prototypical, benchmark example for elliptic PDEs, with substantial interest in its own right. 
Previous results in the literature placed a Gaussian prior distribution of fixed regularity on the (logarithm) of the functional parameter $f$.
In \cite{nickl2020convergence} minimax convergence rates were derived for
the MAP estimator corresponding to an optimally rescaled Gaussian process prior. 
Posterior contraction rates were derived for a uniform prior on the coefficients in a series expansion in \cite{Nickl18} and for a rescaled Gaussian
process prior in \cite{Monardetal2021}. Frequentist coverage guarantees for credible balls in weak Sobolev spaces
were deduced from uniform non-parametric Bernstein-von Mises results in \cite{Nickl18,Monardetal2021}.


The solution $u_f$ to the Schr\"odinger equation has
a probabilistic expression through the Feynman-Kac formula (see Theorem~4.7 in \cite{Zhao2012}):
\begin{align}\label{eq: Feynman-Kac}
	u_f(x) = \E^x\left[ g(X_\tau) e^{- \int_0^{\t} f(X_s)\, ds}\right],
\qquad x \in \O.
\end{align}
Here $(X_s)_{s \geq 0}$ denotes a $d$-dimensional Brownian motion with starting value $X_0 = x \in \O$,
and $\t$ is its exit time from $\O$.  It is known that $\sup_x\E^x \t$ is finite (c.f. Proposition~1.16 in \cite{Zhao2012}).
If $f$ is bounded and $g$ is bounded away from zero on $\partial\O$,
then Jensen's inequality applied to \eqref{eq: Feynman-Kac} shows that
\begin{align*}
	\inf_{x \in \O} u_f(x) \geq \inf_{y \in \partial \O} g(y) \inf_{z \in \O}e^{-\|f\|_\infty  \E^z \tau} > 0.
\end{align*}
This justifies an inversion formula of the form \eqref{EqSolutionOperator}:
\begin{align}\label{eq: inversion formula}
	f = \frac{\Delta u_f}{2 u_f}.
\end{align}
We conclude that to estimate $f$, it suffices to estimate  $\Delta u_f$ and $u_f$. In fact,
an estimator for $\Delta u_f$ suffices, since $u_f$ can be recovered without error from $\Delta u_f$ 
and the known boundary condition $u_f=g$ on $\partial \O$. 

This motivates to take $\L$ as in 
\eqref{eq: general PDE} equal to the Laplacian $\L=\Delta$. Its inverse $K$ as in \eqref{eq: definition K gen}
is then the inverse Laplacian under the Dirichlet boundary condition, and 
the function $\tilde g$ in \eqref{eq: definition K gen} is the solution to the {Poisson equation} with
boundary function $g$. We record the existence of this operator and function in the following lemma.
We assume that the boundary points of $\O$ are \emph{regular} in the sense of \cite{Gilbarg2015}, page~25; a sufficient
condition is that $\O$ satisfies the exterior sphere condition: for every $\xi\in\partial\O$ there exists a closed
ball $B\subset\RR^d$ such that $B\cap \bar\O=\{\xi\}$. This is true in particular for the domain $\O=(0,1)^d$.

\begin{lemma}\label{lemma: existence g tilde}
Let $\L=\Delta$ be the Laplacian.
There exists a compact, self-adjoint  linear operator $K: L_2(\O)\to L_2(\O)$ such that \eqref{eq: definition K gen} holds.
If $\O$ is regular, and $g: \partial \O \to \RR$ is continuous, then there exists a function $\tilde{g}: \O \to \RR$ satisfying
\eqref{eq: definition g tilde gen}.
If $\partial \O$ is $C^{m+2}$, for $m\in\NN\cup\{0\}$, then $K: \Sob^m(\O)\to \Sob^{m+2}(\O)$ is continuous.
\end{lemma}

\begin{proof}
Theorem~3 in Section~6.2 of \cite{Evans10} with $L$ minus the Laplacian
together with the example following the theorem show that 
there exists a unique (weak) solution $Ku\in H_0^1$ to equations \eqref{eq: definition K gen} with $\L$ the Laplacian, 
for every $u \in L_2(\O)$.
The linearity of the map $u\mapsto Ku$ is clear, while 
compactness and self-adjointness is noted in the proof of Theorem~1 on page 335 of \cite{Evans10}.
The continuity $K: \Sob^m(\O)\to \Sob^{m+2}(\O)$ follows from Theorem~5 on page~323 of \cite{Evans10}. 
	
By Theorem~2.14 in \cite{Gilbarg2015}, the Dirichlet problem on a bounded domain  has a solution 
for every continuous boundary function $g$
if and only if the boundary points of the domain are regular. By the remark made below the stated theorem, a sufficient condition for this is that
the domain satisfies the exterior sphere condition.
\end{proof}

Thus we can proceed by the general approach outlined in the introduction with $\L=\Delta$ the Laplacian. 
Because $u_f=K\Delta u_f+ \tilde{g}$, equation \eqref{eq: inversion formula}
shows that the solution map \eqref{EqSolutionOperator} can be written as $f=e(\Delta u_f)$ for
\begin{align*}
e(v)&= 
	\begin{cases}
		\frac{v}{2 (Kv+\tilde g)}, \quad &\text{if} \ \essinf Kv+\tilde g > 0,\\
		0, \quad &\text{ otherwise}.
	\end{cases}
\end{align*}

\begin{lemma}
\label{LemmaSolutionOperatorSchrodinger}
If $u_{f_0}=Kv_0+\tilde g$ satisfies $\inf_{x\in\O} u_{f_0}(x)\ge c_0$, for some $c_0>0$,
then the map $e: V\subset  L_2(\O)\to L_2(\O)$, for $V=\bigl\{v: \|Kv-Kv_0\|_\infty\le c_0/2\bigr\}$,
is Lipschitz at every uniformly bounded $\bar v\in V$ with Lipschitz constant bounded by a multiple of $1\vee\|\bar v\|_\infty$.
\end{lemma}

\begin{proof}
If $Kv_0+\tilde g\ge c_0$ and $\|Kv-Kv_0\|_\infty\le c_0/2$, then $Kv+\tilde g\ge c_0/2$. 
By the triangle inequality, uniformly in $v\in V$,
$$\bigl|e(v)-e(\bar v)\bigr|=\Bigl|\frac{v/2}{Kv+\tilde g}-\frac{\bar v/2}{K\bar v+\tilde g}\Bigr|\le\frac{|v-\bar v|}{c_0}+\frac{|Kv-K\bar v|\,\|\bar v\|_\infty}{c_0^2}.$$
Thus the $L_2$-norm of left-side is bounded above by a multiple of 
$\|v-\bar v\|_{L_2}+\|Kv-K\bar v\|_{L_2}\|\bar v\|_\infty\lesssim \|v-\bar v\|_{L_2}(1\vee \|\bar v\|_\infty)$,
by the continuity of $K$.
\end{proof}

\begin{theorem}
\label{CorollarySchrodinger}
Suppose that $u_{f_0}=Kv_0+\tilde g$ is bounded away from zero and $\|v_0\|_\infty<\infty$.
If the posterior distribution of $Kv$  based on the observation $\tilde Y_n= Kv+ n^{-1/2} \dWW$ 
is consistent for the uniform norm, then the posterior distribution of $f$ given $Y_n$ contracts under $f_0$
at the same rate as the posterior distribution of $v$ given $\tilde Y_n$ under $v_0$.
Furthermore, the credible sets $C_n(Y_n)$ given in \eqref{EqGeneralCredibleSet} have diameter of the same order in probability under $f_0$ as
the credible sets $\tilde C_n(\tilde Y_n)$ under $v_0$ provided the latter sets are contained
in $\{v: \|Kv-K v_0\|_\infty< c_0/2\}$, for $c_0=\inf_{x\in\O} u_{f_0}(x)$, and  contain $\bar v_n$ with $\|\bar v_n\|_\infty=O_P(1)$.
\end{theorem}

\begin{proof}
We combine  Lemma~\ref{LemmaSolutionOperatorSchrodinger}
with Proposition~\ref{prop:stability} to obtain the contraction rate, and
with Proposition~\ref{proposition_credible_general} to obtain the coverage of credible sets.
\end{proof}

Thus we focus on priors for $v$ in the problem $\tilde Y_n= Kv+ n^{-1/2} \dWW$, where $K$ is the
inverse Laplacian under the Dirichlet boundary condition, that attain fast $L_2$-contraction rates
and at the same time give uniformly consistent posteriors of $Kv$. 
Theorems~\ref{theorem: general uniform convergence},~\ref{theorem: general uniform convergence_adaptive} and~\ref{thm:regression:linear}
give many possibilities. For a concrete example consider the domain $\O=(0,1)^d$.

It can be verified that the eigenfunctions of the inverse Laplacian on $\O=(0,1)^d$ with Dirichlet boundary condition are given by
\begin{align}\label{eq: eigenfunctions}
	h_{i_1,\ldots, i_j}(x_1, \ldots, x_d) = 2^{d/2}\prod_{j=1}^d \sin(i_j \pi x_j), \qquad  (i_1, \ldots, i_d) \in \NN^d,
\end{align}
with corresponding eigenvalues $-\k_{i_1,\ldots, i_d}$, for
\begin{align*}
	\k_{i_1, \ldots, i_d} = \frac{1}{(\sum_{j=1}^d i_j^2) \pi^2}.
\end{align*}
The indexing by the $d$-tuples $(i_1, \ldots, i_d)$ precludes an immediate application of
contraction results for sequence priors. However the sequence $k_1\geq k_2\ge\cdots$
resulting from ordering the array of eigenvalues $(\k_i: i\in \NN^d)$ in decreasing magnitude has polynomial decrease
$k_\ell\asymp \ell^{-2/d}$, as $\ell\ra\infty$ (see Lemma~\ref{lemma: sorted eigenvalues} in Section~\ref{SectionProofsSchrodinger}). 
Furthermore, if $v_1,v_2,\ldots$ are the array of coefficients $(\nu_i: i\in\NN^d)$ 
in corresponding order of a function $v=\sum_{i\in\NN^d}\nu_ih_i$,
then the square norm \eqref{EqDefHilbertSpaceNorm} of the smoothness scale  $(\Hil^s)_{s\in\RR}$ 
generated by the basis $(h_i)$ ordered in a sequence satisfies,  for $s\ge 0$,
$$\|v\|_{\Hil^s}^2=\sum_{l=1}^\infty v_l^2l^{2s/d}
\asymp \sum_{i\in\NN^d}\nu_i^2\Bigl(\sum_{j=1}^d i_j^2\Bigr)^s.$$
Although the eigenfunctions are not always explicit, 
similar results are valid for the Laplacian on more general domains; see equation~(3.1) in \cite{Ivrii2016}.

The next lemma verifies the interpolation assumptions of Theorem~\ref{thm:regression:linear}
(see Section~\ref{SectionProofsSchrodinger} for a proof).

\begin{lemma}\label{lemma: interpolation}
Let  $(\Hil^s)$ be the scale generated by the eigenfunctions \eqref{eq: eigenfunctions} and 
let $\LL_n := \bigl\{\sum_{i\in\NN^d, \|i\|_\infty \leq n^{1/d}} c_i h_i: c_i \in \RR\big\}$.
If $\b > d/2$, then for every $v\in\Hil^\b$,
there exists an element $\mathcal{I}_n v \in \LL_n$ 
such that $\mathcal{I}_n v(x) = v(x)$ for every $x\in \{(\frac{2i}{2m+1})_{i=1,...,m}\}^d$, and
		\begin{equation}\label{eq: interpolation assumption}
			\| \mathcal{I}_n v - v\|_{L_2} \lesssim  n^{-\b/d}\|v\|_{\Hil^\b}.
		\end{equation}
Furthermore, there exist constants $0<C_1<C_2<\infty$ such that
\begin{align*}
 C_1 \|v\|_{L_2} \leq \|v\|_{\LL_n} \leq C_2 \|v\|_{L_2},\qquad \forall v\in \LL_n.
\end{align*}
\end{lemma}

Given the eigenfunctions $(h_i)$, for $i\in \NN^d$, equip $v=\Delta u_{f}$  
with the prior $\sum_{i=1}^\infty \nu_ih_i$, for $\nu_i\ind N(0, \k_i^{d/2+\a})$,
where the value of $\a$ is  chosen either fixed or determined by the empirical Bayes or hierarchical Bayes approaches.
For the array $(h_i)$ ordered in by decreasing eigenvalues, this corresponds to a prior
of the form  $v=\sum_{\ell=1}^\infty v_\ell h_\ell$, with $v_\ell\ind N(0, \ell^{-1-2\a/d})$, as
in Section~\ref{SectionLinearProblem}.
To quantify the uncertainty of the procedure, consider credible sets of the form 
\begin{equation}
\label{EqCredibleSchrodinger}
C_n(Y_n)
=\Bigl\{\frac{v}{2(Kv+\tilde g)}: \|v-\hat{v}_n\|_{L_2}\leq r_n, \|Kv-K\hat{v}_n\|_\infty\leq \d_0\Bigr\},
\end{equation}
for $\hat v_n$ the posterior mean of $v$, the constant $r_n$ determined as the
$1-\g$-quantile of the posterior distribution of $\|v-\hat v_n\|_{L_2}$ and a fixed constant $\d_0>0$. 

\begin{theorem}\label{thm:Schr\"odinger}
Assume that $\Delta u_{f_0} \in \Hil^\b$, for $\b > d/2$ and $\inf_x u_{f_0}(x)>2\d_0>0$. 
(i). The $L_2$-contraction rate of the posterior distribution of $f$ to $f_0$ in the white noise model is 
$n^{-(\a\wedge \b)/(d+2\a+4)}$ if the regularity of the prior is chosen 
deterministic and equal to $\a$ with $\a\wedge\b+2>d/2$.
(ii). This contraction rate is $l_n n^{-\b/(d + 2\b + 4)}$ for a slowly varying sequence $l_n$,
if $\a$ is chosen by the empirical Bayes or hierarchical Bayes methods and $\b+2>d/2$. 
(iii). For the prior with deterministic $\a$ set to $\a=\b-c/\log n$ and sufficiently large  $c$,
 the credible set $C_n(Y_n)$ has frequentist coverage tending to one and $L_2$-diameter
of the order $O_P(n^{-\b/(d+2\b+4)})$. (iv). For the prior with fixed $\a>d/2$, the contraction rate as in (i) is also attained 
in the discrete observational scheme.
\end{theorem}

\begin{proof}
The contraction rates (i) and (ii) in the Gaussian white noise model follow from combining 
Theorem~\ref{CorollarySchrodinger}
with Theorem~\ref{theorem: general uniform convergence} or Theorem~\ref{theorem: general uniform convergence_adaptive}(i)-(ii). 
Parts (i) of these theorems give a contraction rate for $\Delta u_f$, which translates into a contraction rate
for $f$ provided the uniform norms $\|u_f\|_\infty$ are bounded away from zero. The latter is
guaranteed through the uniform posterior consistency obtained in parts (ii) of the two theorems,
which apply with $u=2$ under the conditions $\a\wedge \b+2>d/2$
and $\b+2>d/2$, respectively, as $\|K h_\ell\|_\infty\lesssim \ell^{-2/d}$.

The frequentist coverage guarantee (iii) follows in the same way
from combining Theorem~\ref{theorem: general uniform convergence_adaptive}(iii) with
Proposition~\ref{proposition_credible_general}. 
The contraction rate (iv) in the discrete observational model follows from combining
Lemma~\ref{lemma: interpolation} and Theorem~\ref{thm:regression:linear}.
\end{proof}

As discussed in Section~\ref{SectionSobolevSpaces}, similar results are valid for other priors than those resulting
from the eigenexpansion, for instance priors based on certain wavelet expansions.

A sufficient condition for the assumption that the function $\Delta u_{f_0}$ belongs to the Hilbert scale $\Hil^\b$ is
that it is compactly supported in $\O$ and belongs to the Sobolev space $\Sob^\b(\O)$.
See Proposition~\ref{PropositionSmoothingEllipticH0} and
Chapter 5, Appendix~A (page 471), in particular equation (A.18) in \cite{Taylor2011}.)

\section{Heat equation with absorption} 
For  given (smooth) functions $f\in L_2\bigl((0,1)^d\times(0,1)\bigr)$,  $u_0 \in L_2((0,1)^d)$ and $g \in L_2\bigl(\partial (0,1)^d\times [0,1]\bigr)$,
consider the solution $u_f:  [0,1]^d\times[0,1]\to\RR$ to the parabolic partial differential equation
\begin{equation}\label{eq: heat:equation}
	\left\{\begin{aligned}
		\frac{\partial}{\partial t}u - \frac{1}{2}\Delta u &= fu, \qquad && \text{on }   (0,1)^d\times (0,1];\\
		u &= g, \qquad && \text{on }  \partial (0,1)^d\times [0,1];\\
		u &= u_0, \qquad && \text{on } (0,1)^d\times  \{0\}.
	\end{aligned}\right.
\end{equation}
The white noise version of this problem was investigated in \cite{kekkonen2022consistency}, 
where minimax contraction rates and frequentist guarantees for credible sets in a weak metric 
space were derived under the assumption that the regularity $\b$ of the function $f\in \Sob^\b(\O)$ is known. 

The equation is of the form \eqref{eq: general PDE} for the 
parabolic differential operator $\L = \frac{\partial}{\partial t} - \frac{1}{2}\Delta$,
with $\Gamma=\bigl(\partial (0,1)^d\times [0,1]\bigr) \cup \bigl((0,1)^d\times \{0\}\bigr)$. The function
$f$ can be recovered from the solution $u_f$ to the equation as 
$$f=\frac{\L u_f}{u_f}.$$
The existence of the inverse operator $K \in L_2\bigl([0,1)^d\times [0,1]\bigr)$ satisfying \eqref{eq: definition K gen} and 
the function $\tilde g$ satisfying \eqref{eq: definition g tilde gen}
is proved in Theorems~3 and~4 of Section~7.1 in \cite{Evans10} and in Section~3 of \cite{Doob}.
This gives the solution map $f=e(\L u_f)$ as in \eqref{EqSolutionOperator} given by 
\begin{align*}
e(v)&= 
	\begin{cases}
		\frac{v}{Kv+\tilde g}, \quad &\text{if} \ \essinf Kv+\tilde g > 0,\\
		0, \quad &\text{ otherwise}.
	\end{cases}
\end{align*}
This is similar to the Schr\"odinger equation in Section~\ref{sec: Schrodinger}, except for the
different definition of $K$.  Lemma~\ref{sec: Schrodinger} is true in exactly the same form for the
present operator $K$ and solution map.

In view of Proposition~\ref{prop:stability}, an $L_2$-contraction rate of the
posterior distribution of $v$ to $v_0$ based on the observation $\tilde Y_n= Kv+ n^{-1/2} \dWW$ 
implies the same $L_2$-contraction rate for the posterior distribution of $f$, 
provided the posterior distribution
of $Kv$ is consistent for $Kv_0$ relative to the uniform norm and $u_{f_0}=Kv_0+\tilde g$ is bounded away from zero. Proposition~\ref{proposition_credible_general} 
shows that the coverage and diameter of credible
sets translate similarly. 
This is exactly as in Theorem~\ref{CorollarySchrodinger}, and applicable to
many priors.

For a concrete example, consider the prior on the eigenfunctions of 
the operator $K^\tp K$, which can be derived in closed form
 (see Lemma~\ref{lemma: bounded eigenfunctions parabolic}).
We may endow $v=\L u_f$ with a Gaussian prior 
$\sum_{i\in\NN}v_ih_i$, for $h_1,h_2,\ldots$ the eigenfunctions ordered by decreasing eigenvalues
and $v_i\ind N(0, i^{-1-\a/d})$, with the hyper-parameter $\a$ either fixed to some value 
or selected with the empirical or hierarchical Bayes methods, 
as described in Section~\ref{SectionLinearProblem}. 
For uncertainty quantification we consider the credible set described in \eqref{EqCredibleSchrodinger}.

Let $(\Hil^s)$ be the smoothness scale corresponding to the sequence of eigenfunctions,
 ordered by corresponding decreasing eigenvalues and viewed as functions on a $(d+1)$-dimensional domain.
Thus a function  belongs to $\Hil^\b$ if it has a representation $v=\sum_{i\in\NN}v_ih_i$ for coefficients $v_1,v_2,\ldots$
with $\sum_{i\in\NN}i^{2\b/(d+1)}v_i^2<\infty$. 

\begin{theorem}\label{thm:parabolic}
Assume that  $\L u_{f_0} \in \Hil^\b$, for $\b > d+1$, and $\inf_x u_{f_0}(x)>2\d_0>0$.
(i). In the white noise model the $L_2$-contraction rate of the posterior distribution for $f$ to $f_0$ is
equal to $n^{-(\a\wedge \b)/((d+1)(d+6)/(d+2)+2\a)}$ if the regularity of the prior is chosen deterministic and equal to $\a>d+1$,
and equal to $l_n n^{-\b/((d+1)(d+6)/(d+2) + 2\b)}$ for a slowly varying sequence $l_n$
if $\a$ is chosen by the  hierarchical or empirical Bayes methods.  
(ii). For the prior with deterministic $\a$ set to $\a=\b-c/\log n$ and  sufficiently large $c$, 
the credible set \eqref{EqCredibleSchrodinger} has frequentist coverage tending to one and 
$L_2$-diameter of the order $n^{-\b/((d+1) (d+6)/(d+2)+ 2\b )}$.
\end{theorem}

\begin{proof}
In view of Lemma~\ref{lemma: sorted eigenvalues} and Lemma~\ref{lemma: bounded eigenfunctions parabolic},
the eigenvalues $k_\ell^2$ of $K^\tp K$ ordered by decreasing magnitude satisfy $k_\ell\asymp \ell^{-2/(d+2)}$.
Thus the  operator $K$ is smoothing as in \eqref{eq:K:smoothing}  relative to the eigenscale $(\Hil^s)$ 
on a $(d+1)$-dimensional domain of order $p$ such that $2/(d+2)=p/(d+1)$, i.e.\ $p=2(d+1)/(d+2)$.

Consequently, Theorem~\ref{theorem: general uniform convergence_adaptive} 
gives  $L_2$-contraction at the rates $n^{-(\a\wedge\b)/(d+1+2\a+2p)}$, which is the rate given in (i), and
$\|\cdot\|_{\Hil^\g}$-contraction at some rate, for every $\g<\b$.

By Proposition~\ref{prop:stability}
the $L_2$-contraction rate of the posterior of $v$  is transferred to the same rate for $f$ provided
that the posterior distribution of $Kv$ is consistent for the uniform norm. 
By Lemma~\ref{cor:elliptic_parabolic}, the latter
 follows from consistency in the $\|\cdot\|_{\Hil^\g}$-norm for some $\g> d+1$.

The frequentist coverage (ii) of the credible sets follows similarly, with the help of
Proposition~\ref{proposition_credible_general}.
\end{proof}

\begin{remark}
The space and time variables $x$ and $t$  in this problem do not play symmetric roles and it would be
reasonable to treat them differently when estimating the function $(x,t)\mapsto f(x,t)$, in practice when
devising a prior distribution and in theory when obtaining a contraction rate. In the preceding we lumped the
two variables together in order to apply the contraction results on series priors given in Section~\ref{SectionLinearProblem}. 
As shown in Lemma~\ref{lemma: bounded eigenfunctions parabolic}, the singular values of $\L$ are of the order $\|i\|^2+k$
(for the eigenfunctions $(x,t)\mapsto\prod_{i=1}^d\sin(i_j\pi x) \bigl(c_{i,k}\sin(\nu_{i,k} t)+\cos(\nu_{i,k} t)\bigr)$),
where $\nu_{i,k}\asymp k$. This suggests that the problem is ill-posed of order $2$ in the space variable and
of order $1$ in the time variable, as is also intuitively clear from the definition of $\L$. In the series setup of
Section~\ref{SectionLinearProblem}  this is accommodated by a single parameter
of ill-posedness equal to $p=2(d+1)/(d+2)$ (see the preceding proof), which is between 1 and 2. Using instead two parameters of ill-posedness
would allow to obtain results for more general priors and more general model classes, more appropriate for the
anisotropic nature of the problem. (A further refinement would be to model also the coordinates
of the space variable anisotropically.) We leave this for further work.
\end{remark}

\begin{remark}
In \cite{kekkonen2022consistency} the function $f$ was considered static over time, i.e. $f:\mathcal{O}\rightarrow \mathbb{R}$, while in our case $f$ was allowed to change over time. We also note that our contraction rate is different than $n^{-\beta/(d+2\beta+4)}$ obtained in  \cite{kekkonen2022consistency}. A partial explanation of the mismatch is the different smoothness condition and choice of the prior (in our case $\L u_{f_0} \in \Hil^\b$ and $\L u_{f}$ is endowed with a Gaussian process prior, while \cite{kekkonen2022consistency} assumes that $f_0\in H^\beta(\mathcal{O})$ and endows $f$ with a rescaled Mat\'ern process prior). 
\end{remark}

\section{One-dimensional Darcy equation}
\label{SectionOneDimensionalDarcy}
For given functions $h: (0,1)\to\RR$ and $g: \{0,1\}\to\RR$, consider the one-dimensional differential equation
\begin{equation}\label{EqDarcyOneDimensional}
	\left\{\begin{aligned}
		f'u_f'+fu_f'' &= h, \qquad && \text{on } (0,1),\\
		u_f &= g, \qquad && \text{on } \{0, 1\}.
	\end{aligned}\right.
\end{equation}
Because the solutions $f$ to the first order linear ordinary differential equation $f'u'+fu''=h$, for given $u$ (and hence $u'$ and $u''$),
form a one-dimensional affine space, and we can at best retrieve $u_f$ from the data, the function $f$ can
be recovered from the data at best up to a constant. We shall assume that the value $f(0)$ is given. 

The equation can be written in the form $(fu_f')'=h$. Integration of this equation gives that
$fu_f'-f(0)u_f'(0)= H$, for $H(x)=\int_0^x h(s)\,ds$.
Under the assumption that $u_f'$ is positive, we can solve $f$ as $f=e(u_f')$ for
\begin{equation}\label{eq: inversion Darcy}
	e(v) (x)= \frac{H(x) + f(0)v(0)}{v(x)}.
\end{equation}
We conclude that we can write equation \eqref{EqDarcyOneDimensional}
 in the form \eqref{eq: general PDE} with $\L u=u'$ (and $c(f,u_f)=(f(0)u_f'(0)+H)/f$)
and solution map \eqref{EqSolutionOperator} given in the preceding display. 
Let $\tilde g$ be the function $\tilde g(x)=g(0)+\bigl(g(1)-g(0)\bigr)x$, and define an operator $K$ by
$$Kv(x)=\int_0^x v(s)\,ds-x\int_0^1 v(s)\,ds.$$
Then $Kv(0)=Kv(1)=0$, for every $v$, and hence $Ku'+\tilde g=g$ on $\partial \O$. Since also $(Ku'+\tilde g)''=u''$
on $\O=(0,1)$, it follows that $u_f=K\L u_f+\tilde g$, verifying the identity \eqref{Equ=KLu+tildeg}.

The adjoint of $K$ is given by
$K^\tp u =-\int_0^\cdot u(s)\,ds+\int_0^1\int_0^tu(s)\,ds\,dt$, and hence $(K^\tp K v)''=-(Kv)'=-v+\int_0^1 v(s)\,ds$,
with (Neumann) boundary conditions $(K^\tp K v)'(x)=-Kv(x)=0$, for $x\in\{0,1\}$.
It follows that the eigenfunctions and corresponding eigenvalues of $K^\tp K$  are given by 
\begin{align*}
h_0&=1, \qquad \k_0^2=0, \\
h_i(x)&=\sqrt 2 \cos(\pi ix),\qquad \k_i^2=\frac1{(\pi i)^2}, \qquad i\in\NN.
\end{align*}
The boundary conditions $(u_f-\tilde g)(0)=(u_f-\tilde g)(1)=0$ imply that $\int_0^1 (u_f-\tilde g)'(s)\,ds=0$,
whence the functions $u_f'-\tilde g'$ are orthogonal to $h_0$. This motivates to
consider the prior on $v=u_f'$ equal to the distribution of $\tilde g'+\sum_{i\in\NN} v_ih_i$, for independent $v_i\sim N(0, i^{-1-2\a})$. 
Define $(\Hil^s)_{s\in\RR}$ as the scale generated by the eigenexpansion of $K^\tp K$.

\begin{theorem}\label{cor:darcy}
Suppose $\|H\|_\infty<\infty$,  $u_{f_0}' \in \Hil^\b$, for $\b>1/2$, and $\inf_{0<x<1} u_{f_0}'(x)>0$.
(i). For any $\a\wedge\b>\d>1/2$ the contraction rate of the posterior distribution of $f$
to $f(0)$ relative to the $\|\cdot\|_\infty$-norm is (at least) $n^{-(\a \wedge \b-\d)/(1+2\a + 2)}$. 
(ii). If $f(0)=0$ and $\a\wedge \b>1/2$, then the contraction rate of the posterior distribution of $f$
 to $f(0)$ relative to the $L_2$-norm is $n^{-(\a \wedge \b)/(1+2\a + 2)}$.
\end{theorem}

\begin{proof}
The solution map \eqref{EqSolutionOperator}, relative to $\L u=u'$, is given by \eqref{eq: inversion Darcy}.
On the set of functions $V_n=\{v: v\ge c\}$, for fixed $c>0$, it satisfies
$$\bigl|e(v)-e(v_0)\bigr|\le \frac{f(0)|v(0)-v_0(0)|}{c}+\bigl(\|H\|_\infty+f(0)v_0(0)\bigr)\frac{|v_0-v|}{c^2}
\lesssim \|v-v_0\|_\infty.$$
Since $v_0=u_{f_0}'$ is bounded away from zero by assumption, the posterior distribution
of $v$ will concentrate on $V_n$ as soon as it is consistent
for the uniform norm. Thus in view of Proposition~\ref{prop:stability},
for assertion (i) it suffices to prove that the posterior distribution of $v$ contracts to $v_0$ relative to the uniform norm
at the rate $n^{-(\a \wedge \b-\d)/(1+2\a + 2)}$.

By Theorem~\ref{theorem: general uniform convergence_adaptive} 
the posterior contraction rate for $v$ relative to the $\|\cdot\|_{\Hil^\d}$-norm is $n^{-(\a \wedge \b-\d)/(1+2\a + 2)}$.
By construction $\langle u_{f_0}',1\rangle_{L_2}=\langle v,1\rangle_{L_2} =\tilde g(1)-\tilde g(0)=0$, for almost every function $v$ in a set of prior (and hence
posterior) probability one. It follows that on this set
$$\|v-v_0\|_\infty=\sup_{0<x<1}\Bigl|\sum_{j\in \NN}(v_j-v_{0,j})h_j(x)\Bigr|\le \|v-v_0\|_{\Hil^\d}\sup_{0<x<1}\sqrt{\sum_{j\in\NN}\frac{h_j(x)^2}{j^{2\d}}}.$$
Since the functions $h_j$ are uniformly bounded, the right side
is bounded by a universal multiple of $\|v-v_0\|_{\Hil^\d}$, for any fixed $\d>1/2$.
Thus the posterior contraction relative to the $\|v-v_0\|_{\Hil^\d}$-norm implies a posterior contraction relative to the uniform norm at the same rate.

To prove assertion (ii), we first note that if $f(0)=0$, then  the map $e$ is Lipschitz at $v_0$ also relative to the $L_2$-norm on $V_n$.
By Theorem~\ref{theorem: general uniform convergence_adaptive} the posterior contraction rate for $v$ 
relative to the $L_2$-norm is $n^{-(\a \wedge \b)/(1+2\a + 2)}$.
By Proposition~\ref{prop:stability} this carries over to a posterior contraction rate for $f$ provided that the posterior 
probability of the sets $V_n$ tends to one. Under the assumption that $v_0=u_{f_0}'$ is bounded away from zero,
the latter follows from posterior consistency of $v$ for the uniform norm. By the argument of the preceding paragraph
this follows from consistency relative to the $\|\cdot\|_{\Hil^\d}$-norm, for some $\d>1/2$,  and this is true for $\a\wedge \b>1/2$.
\end{proof}

\begin{remark}
The assumption $\inf_{0<x<1}u_{f_0}'(x)>0$ can be relaxed to the assumption that $\inf_{0<x<1}\bigl(|u_{f_0}'(x)|+u_{f_0}''(x)\bigr)>0$. See
Appendix~\ref{AppendixOneDimensionalDarcy}.
\end{remark}

\begin{remark}
The \emph{two} boundary conditions of the non-homogeneous problem \eqref{EqDarcyOneDimensional}
cannot be removed by the general scheme  \eqref{eq: definition K gen} --\eqref{eq: definition g tilde gen} explained in the
introduction when using the first order differential operator $\L u=u'$. Posing the problem instead
in terms of the second order operator $\L_0 u=-u''$ would remedy this, and yield the same function $\tilde g$
and  inverse operator $-KK^\tp$ as in the preceding discussion. This would give $u_f=-KK^\tp u_f''+\tilde g= Ku_f'+\tilde g$, 
and hence lead to the same inverse equation as introduced.
\end{remark}

\begin{remark}
Other boundary conditions than the Dirichlet boundary condition in \eqref{EqDarcyOneDimensional}
may motivate different priors. For instance, for the mixed boundary condition $u_f''(0)=u_f'(1)=0$,
we may use the same differential operator $\L u=u'$, but take the standard Volterra
operator $Kv=\int_0^\cdot v(s)\,ds$ as its inverse. 
The eigenfunctions of $K^\tp K$ 
are the functions $h_i(x) = \sqrt{2}\cos\bigl((i+1/2) \pi x\bigr)$ with corresponding eigenvalues 
$\k_i^2 = ((i + 1/2)\pi)^{-2}$, for $i\in \NN\cup\{0\}$. The functions in the eigenscale $\Hil^s$ generated by the
operator $K^\tp K$ satisfy the boundary condition on $u_f'$, and the preceding corollary
remains valid provided $u_{f_0}' \in \Hil^\b$.
\end{remark}

\section{Darcy equation}
\label{SectionDarcy}
For a bounded domain $\O\subset\RR^d$, an unknown function $f:\O\to\RR$ and given functions $h: \O\to\RR$ and $g: \partial \O\to\RR$, consider 
the solution $u_f: \bar \O\to\RR$ of the partial differential equation
\begin{equation}\label{EqDarcy}
	\left\{\begin{aligned}
		\div (f \nabla u_f) &= h, \qquad && \text{on } \O,\\
		u_f &= g, \qquad && \text{on } \partial \O.
	\end{aligned}\right.
\end{equation}
Here $\div v=\nabla\cdot v$ denotes the divergence $\div v=\sum_{i=1}^d\frac{\partial}{\partial x_i} v_i$ of a vector field $v: \O\to\RR^d$
and $\nabla u$ is the gradient of a function $u: \O\to \RR$, so that
$\div (f\nabla u)=\nabla f\cdot\nabla u +f\Delta u$, with the dot on the right side the ordinary
inner product of $\RR^d$ and $\Delta$ the Laplacian. For a sufficiently smooth domain and sufficiently smooth
functions $h$ and $g$, a solution $u_f$ is known to exist for every given positive $f\in H^\b(\O)$ with $\b>1+d/2$
(see \cite{Nickl23}, Proposition~6.1.5 or Chapter~8 in \cite{Gilbarg2015}).
Estimation of $f$ in the white noise version of this model was investigated in \cite{GiordanoNickl2020},
 who obtained contraction rates for certain Gaussian process priors on the function $\Phi^{-1}(f)$ for a given link function $\Phi$.

For $d=1$, this model reduces to the model considered in Section~\ref{SectionOneDimensionalDarcy}, which
already revealed interesting features. The equation for $d>1$ arises in many applied settings, but the challenges 
are substantial. The function $f$ can be recovered from $u_f$ only under certain conditions, and this may require 
that certain boundary values of $f$ are pre-given, in addition to the boundary values on $u_f$ in \eqref{EqDarcy}.
Moreover, if recovery is possible, then the solution map \eqref{EqSolutionOperator} is not explicit, but only
available as a numerical algorithm. Such difficulties can be avoided by taking multiple measurements
of the solution function $u_f$, for the same function $f$, but with different boundary functions $g$. This is
feasible in particular in an experimental setup, when the boundary function is under the control of the experimenter. 
Not only do multiple measurements simplify the inverse problem, but also better
recovery rates may be expected. In this section we apply our general method with both a single and multiple measurements,
in both cases with our operator $\L$ taken to be the Laplacian.

We start by a brief review of some aspects of the inversion map. 
A standard approach is the method of \emph{characteristics} (see e.g.\ \cite{Evans10}, p97--115), 
the set of curves $x: [\tau_0,\t]\to\O$ in the domain $\O$ defined through the gradient flows
 $x'(t)=\nabla u\bigl(x(t)\bigr)$, with varying initial values $x(\tau_0)$. 
For each characteristic, equation \eqref{EqDarcy} gives, for $u=u_f$,
$$\frac{d}{dt}f\bigl(x(t)\bigr)+f\bigl(x(t)\bigr)\Delta u \bigl(x(t)\bigr)=h\bigl(x(t)\bigr).$$
This ordinary differential equation can be solved by integrating factors, to give
\begin{equation}
\label{EqDarcySolutionAlongCharacteristics}
f\bigl(x(t)\bigr)=e^{-\int_{\tau_0}^t\Delta u\bigl(x(s)\bigr)\,ds}f\bigl(x(\tau_0)\bigr)+\int_{\tau_0}^t h\bigl(x(s)\bigr)e^{-\int_s^t\Delta u\bigl(x(r)\bigr)\,dr}\,ds.
\end{equation}
This equation completely specifies $f$ on the characteristic $\{x(t): t\in [\tau_0,\t]\}$ in terms of $\Delta u$ and the starting value
$f\bigl(x(\tau_0)\bigr)$. (It is attractive that the formula uses only the Laplacian $\Delta u$, which is the starting point
of our reconstruction method, but the method also needs the gradient $\nabla u$ to construct the characteristics.) 
The method of characteristics is to cover the complete domain $\O$ by characteristics, and thus obtain
an explicit solution to the inverse problem. This works best if the gradient $\nabla u$ never vanishes,
which (given sufficient smoothness) would allow to extend every given characteristic (one can start one at every interior point of $\O$) to the boundary of $\O$.
The domain is then covered by a set of characteristics starting at some boundary point (said to belong to the \emph{influx boundary})
and exiting the domain at another boundary point. For a full reconstruction the function $f$ then needs to be 
initialised on (only) the influx boundary. 

A nonvanishing gradient is considered desirable, but depends on the problem, in particular on the boundary function $g$.
If the gradient $\nabla u$ does vanish at some interior point of the domain, then a characteristic essentially stops there.
It is shown in \cite{Richter}, that the method of characteristics is then still applicable under the condition that the Laplacian $\Delta u$ is nonzero at such a point
and has the same sign at all zeros of the gradient, e.g.\ $\Delta u+\|\nabla u\|>0$ throughout $\O$.
In the latter case the domain $\O$ can be covered by characteristics as before, between boundary points, and characteristics starting at zeros of $\nabla u$ leading
to the boundary (with infinite time set), and $f$ can be recovered by integrating over the characteristics, as before. Interestingly, at a zero of $\nabla u$, equation
\eqref{EqDarcy} specifies that  $0+f\Delta u=h$, and hence the initial value of $f$ at such a point is given by the value of $h/\Delta u$, which is measured
and need not be externally specified. It may even be that all characteristics emanate from a point, or points, where the gradient vanishes (so that the influx boundary is empty)
and no pre-given values of $f$ are necessary for its recovery.

The influx boundary are the points where the gradient $\nabla u$ points into the interior of the domain $\O$.
As seen previously, the recovery of $f$ requires pre-given values on the influx boundary, and then $f$ is determined on $\O$ and also at the
remaining boundary points. This is somewhat tricky, as it indicates that one can also specify too many boundary values,
and rule out every function $f$, the more so since the inflow boundary depends on $u_f$ in general.

The paper \cite{Richter} also discusses cases where there is no (unique) solution $f$ to the inverse problem. 
In practice, such situations may be avoided if the measurements can be taken for specific boundary functions $g$.
The inverse problem can be further simplified by taking measurements for multiple boundary functions.
If $u_{f,1},\ldots, u_{f,d}$ satisfy \eqref{EqDarcy} for $h=0$ and boundary functions $g_1,\ldots, g_d$,  and the
same function $f$, then (with $[a_1,\cdots a_d]$ the matrix with columns $a_1,\ldots,a_d$),
$$\nabla f\cdot \bigl[\nabla u_{f,1}\cdots\nabla u_{f,d}\bigr]  +f\big[\Delta u_{f,1}\cdots \Delta u_{f,d}\bigr]=0.$$
If the boundary functions can be chosen such that $\det [\nabla u_{f,1}\cdots\nabla u_{f,d}]\not=0$ throughout the domain $\O$, then
this gives the explicit inversion formula
\begin{equation}
\label{EqInverseFormulaHybrid}
\frac {\nabla f}f= -\big[\Delta u_{f,1}\cdots \Delta u_{f,d}\bigr]\bigl[\nabla u_{f,1}\cdots\nabla u_{f,d}\bigr]^{-1}.
\end{equation}
For a convex domain $\O$, the function $f$ (or rather $\log f$) can be recovered from this up to a multiplicative (or additive) constant
by  integration along lines starting from a single point. The construction of appropriate
boundary functions $g_i$ and other aspects of this approach are studied in \cite{AlbertiCapdeboscq}.

\subsection{Single measurement}
In this section we consider the case of a single (noisy) measurement of the function $u_f$, following \cite{Nickl23} and \cite{GiordanoNickl2020}.
It is shown in  \cite{Richter} that every $u\in C^2(\bar \O)$ such that
\begin{equation}
\label{EqCu}
C(u):= \inf_{x\in\O}\bigl(\Delta u+\|\nabla u\|^2\bigr)(x)>0,
\end{equation}
arises as $u=u_f$ for some  function $f$ that is absolutely continuous along the characteristics of $u$.%
\footnote{As noted in \cite{Richter} and \cite{GiordanoNickl2020} condition \eqref{EqCu} is certainly
satisfied if $f$ and $h$ are stricitly positive and bounded, since \eqref{EqDarcy} implies 
$\inf h\le |\nabla \cdot (f\nabla u_f)|\le \|\nabla f\|\,\|\nabla u_f\|+f\Delta u_f$, whence
either $2\|\nabla u_f\|\ge \inf h/\|\nabla f\|_\infty$ or $2\Delta u_f\ge \inf h/\|f\|_\infty$. In \cite{Richter} it is also noted
that weakening the condition by replacing $\Delta u$ by its absolute value gives a very different problem, in which $f$ is no longer uniquely determined.}
The following lemma, which is a slight adaptation of Proposition~2.1.5 of \cite{Nickl23} shows that the inverse map $u_f\mapsto f$
is Lipschitz relative to the $\Sob^2(\O)$-norm, up to boundary values.

Let $\|f\|_{C^1(\O)}=\|f\|_\infty+\|\nabla f\|_\infty$. A proof of the lemma, following \cite{Nickl23}, is included in Section~\ref{SectionComplementsDarcy}.
(The assumed smoothness of the domain is needed only for validity of Green's formula; the lemma is also valid for e.g.\ $\O=[0,1]^2$.)

\begin{lemma}
\label{LemmaDarcyStable}
Let $\O$ be a smooth bounded domain in $\RR^d$ and let $g: \O\to\RR$ and $h: \partial \O\to\RR$ be given  smooth functions.
Let $u_f, u_{f_0}\in C^2(\O)$ be solutions to \eqref{EqDarcy} for given $f, f_0\in C^1(\O)$. Then
\begin{align*}
\|f-f_0\|_{L_2(\O)}^2\,C(u_f) e^{-2\|u_f\|_\infty}&\le 2\|f_0\|_{C^1(\O)}\,\|u_f-u_{f_0}\|_{\Sob^2(\O)}\|f-f_0\|_{L_2(\O)}\\
&\qquad\qquad+\int_{(\partial \O)_{1,f}}(f-f_0)^2\,dS \sup_{x\in\partial \O}\|\nabla u_f(x)\|,
\end{align*}
where the integral in the last term is a surface integral over the inflow boundary
$(\partial\O)_{1,f}:=\{x\in\partial \O: \nabla u_f\cdot \vec n(x)< 0\}$, for $\vec n$ the outer normal vector field on $\partial\O$.
In particular,  for every pair of positive functions $f,f_0\in \Sob^\b(\O)$, for $\b>1+d/2$,  with $f=f_0$ on $\partial \O$,
\begin{equation}
\|f-f_0\|_{L_2(\O)}\lesssim \frac{e^{2\|u_f\|_\infty}}{C(u_f)} \ \|\Delta u_f-\Delta u_{f_0}\|_{L^2(\O)}.
\label {EqRichardDarcy}
\end{equation}
\end{lemma}

Inequality \eqref{EqRichardDarcy} suggests that our approach with the operator $\L=\Delta$  equal to the Laplacian may yield 
contraction and coverage results relative to the $L_2(\O)$-norm on the functions $f$.
The multiplicative constant $e^{2\|u_f\|_\infty}/C(u_f)$ can be controlled by consistency
of the posterior distribution of $\Delta u_f$ relative to the uniform norm. 
Two potential difficulties are the  boundary conditions on $f$ and the fact that the formula is restricted to the
surface of Laplacians of solutions $u_f$ to \eqref{EqDarcy}.

Formula \eqref{EqRichardDarcy} presumes that $f$ and $f_0$ agree on the influx boundary. The preceding discussion
suggests that this presumption cannot be discarded and the first formula in the lemma shows that additional
terms must be added otherwise. This is intrinsic to the inverse problem, but both practically and theoretically inconvenient.

A fortunate situation arises when the influx boundary is empty, in which case $f$ can be reconstructed from $u_f$ without specifying boundary values,
and the inverse inequality \eqref{EqRichardDarcy} is valid.  A case of interest where this is true, is
when the boundary function $g$ in \eqref{EqDarcy} is constant and $f$ is positive.  Indeed, if $u_f=g$ is constant on $\partial \O$, 
and $f$ is positive, then $u_f$ cannot have a maximum in $\O$ 
by the maximum principle for elliptic partial differential equations (\cite{Gilbarg2015}, page~31) 
and hence $u_f$ will never increase when moving from a boundary point into the interior of $\O$, ensuring
that $(\partial\O)_{1,f}$ is empty.

Another case of interest is that the values of $f$ on the full boundary happen to be known. 
Then one could proceed by constructing a posterior distribution that is concentrated on the set of 
functions $u_f$ that agree with the known boundary values.\footnote{Every  positive  $f\in H^\b(\O)$ defines a solution  $u_f$. 
The solutions obtained from $f$ with the given boundary
values define a surface of functions that agree with these boundary values.}
In \cite{GiordanoNickl2020} it is assumed that the function $f$ is known and also constant 
on a neighbourhood of the boundary of $\O$. Under the same assumption, 
a posterior distribution on the Laplacian $v=\Delta u_f$, the starting point of our analysis,  could be modified
to correspond to functions $f$ with the given boundary conditions before inverting to the function $f$ by a solution operator \eqref{EqSolutionOperator}. 
For instance, if the gradient  $\nabla f$ vanishes near the boundary (as in \cite{GiordanoNickl2020}), then $\Delta u_f=h/f$ near the boundary,
by \eqref{EqDarcy}, and a given function $v$ can be modfied into a function $\bar v$
that takes the known value $h/f$  near the boundary. 
(This modification decreases the $L_2$-distance $\|v-\Delta u_{f_0}\|_{L_2(\O)}$ as long as $f_0$ has the same boundary values.)

Another limitation of \eqref{EqRichardDarcy} is that it applies only to solutions $u_f$ of \eqref{EqDarcy}.
This can be circumvented, in principle, by defining  a solution map $e: L_2(\O)\to L_2(\O)$ by
\begin{equation} 
\label{EqSolutionMapByMinimisation}
e(v)=\argmin_{f\in\F} \Bigl(f\mapsto \|\Delta u_f -v\|_{L_2(\O)}: u_f \text{ solves \eqref{EqDarcy}} \Bigr).
\end{equation}
Boundary conditions on $f$ can be inserted through the domain $\F$ of the minimization.
(We assume that the minimum is assumed, uniquely; otherwise, the argument can proceed with a
consistent choice of a near minimiser, within a negligible tolerance.) If $v=\Delta u_f$ for 
$f\in\F$, then clearly $e(v)=f$, whence \eqref{EqSolutionOperator} is satisfied for $f\in\F$.
Furthermore, by \eqref{EqRichardDarcy}, provided $\exp(2\|u_{e(v)}\|_\infty)/C(u_{e(v)})$ is bounded,
$$\|e(v)-f_0\|_{L_2(\O)}\lesssim \|\Delta u_{e(v)}-\Delta u_{f_0}\|_{L_2(\O)}
\le 2 \|v-\Delta u_{f_0}\|_{L_2(\O)},$$
in view of the triangle inequality and the fact that $\|v-\Delta u_{e(v)}\|_{L_2(\O)}\le \|v-\Delta u_{f_0}\|_{L_2(\O)}$ by the
definition of $e(v)$. 
If $C(u_{f_0})\ge c_0>0$ and $\|u_{f_0}\|_\infty\le d_0$, then consistency of the posterior distributions of 
$\Delta u_f$, $\nabla u_f $ and $u_f$ for the uniform norm, shows
 that restricting $\F$ to functions such that $C(u_f)\ge c_0/2$ and $\|u_f\|_\infty \le 2d_0$ 
has an asymptotically negligible effect on the induced posterior distribution of $e(v)$.
Then the $L_2$-contraction rate of a posterior distribution of the Laplacian $v=\Delta u_f$ is carried over
into the same  $L_2$-contraction rate of the induced posterior of the reconstructed function  $f=e(v)$.

For a concrete example for the domain  $\O=(0,1)^d$,
we may use the prior on the eigen expansion of the Laplacian, as discussed in Section~\ref{sec: Schrodinger}.
Let $(\Hil^s)$ be the smoothness scale generated by the eigenfunctions  \eqref{eq: eigenfunctions}.
We assume that the boundary values of $f$ are given and incorporated in such a way that
\eqref{EqRichardDarcy}  holds (see the preceding discussion, e.g.\ consider the case that $g$ is constant or $f$ is constant
in a neighbourhood of the boundary). 

\begin{theorem}\label{thm:Darcy}
Assume that $\Delta u_{f_0} \in \Hil^\b$, for $\b > d/2$ and $C(u_{f_0})>0$. 
(i). The $L_2$-contraction rate of the posterior distribution of $f$ to $f_0$ in the white noise model is 
$n^{-(\a\wedge \b)/(2\a+4+d)}$ if the regularity of the prior is chosen 
deterministic and equal to $\a$ with $\a>d/2$.
(ii). This $L_2$-posterior contraction rate is $l_n n^{-\b/( 2\b + 4+d)}$ for a slowly varying sequence $l_n$,
if $\a$ is chosen by the empirical Bayes or hierarchical Bayes methods. 
\end{theorem}

\begin{proof}
The contraction rates (i) and (ii) follow from Theorem~\ref{theorem: general uniform convergence}
  and Theorem~\ref{theorem: general uniform convergence_adaptive},
where we obtain both a contraction rate $n^{-(\a\wedge\b-\d)/(2\a+4+d)}$ for the Laplacian $\Delta u_f$ relative
to the $\Hil^\d$ norm and a contraction rate  $n^{-(\a\wedge\b)/(2\a+4+d)}$  for the Laplacian relative to the $L_2$-norm,
with $\a=\b$ in the case of (ii) and provided  $\a\wedge \b>\d$.
By Sobolev embedding, the first, with $\d>d/2$, implies posterior consistency for the Laplacian relative to the uniform norm,
thus allowing control over the constants $C(u_f)$ and $\|u_f\|_\infty$ needed to use the inversion
formula \ref{EqRichardDarcy} to obtain the contraction rate in $L_2$ for $f$ from the contraction rate for the Laplacian.
\end{proof}

While definition \eqref{EqSolutionMapByMinimisation} proves the feasibility of our approach, it is not practically 
useful without an (efficient) algorithm to solve the minimisation problem. This is similar to solving
the well known \emph{deterministic} inverse problem of computing $f$ from $u_f$, but better conditioned
as our starting point is the Laplacian $\Delta u_f$ instead of $u_f$. Given an estimate $v$ for $\Delta u_f$, we can use
\eqref{Equ=KLu+tildeg} to compute $u=Kv+\tilde g$ and $\nabla u=\nabla Kv+\nabla \tilde g$ 
as estimates of $u_f$ and $\nabla u_f$ and next $f$ by the method of
characteristics \eqref{EqDarcySolutionAlongCharacteristics}. The first two steps are particularly
easy if $v$ is given in terms of the eigenexpansion $v=\sum_i v_ih_i$ of $\Delta$,  in which case
$\nabla u=\sum_i \k_iv_i\nabla h_i+\nabla \tilde g$.

A numerical procedure implementing the third step \eqref{EqDarcySolutionAlongCharacteristics} is already provided in \cite{RichterNumerical}, 
exactly under the condition that $C(u)$ is positive, but under the condition that the target function $f$ is $C^2(\O)$ (see Theorem~2
 in \cite{RichterNumerical}; note that $f$ in the latter paper is our $g$ and $\alpha$ is our $f$). 
In the remainder of this section we adapt the algorithm to take the Laplacian $v=\Delta u$ and  gradient $w=\nabla u$ as inputs, and
show that the inversion is stable if the function $f$ is absolutely continuous (along discretised characteristics) and contained in $C^1(\O)$. 
(We have not investigated whether the latter condition is always reasonable, in particular in the case of a vanishing gradient, but
it cannot be much improved as the gradient $\nabla f$ appears in the defining equation. In Section~\ref{AppendixOneDimensionalDarcy}
the assumption is verified in the case $d=1$, even for a vanishing gradient, provided $v$ is sufficiently smooth.)

We consider the case that the domain is the two-dimensional unit square $\O=[0,1]^2$. The construction can be extended to 
less regular domains along the lines described in Section~4 of \cite{RichterNumerical}. For a given $\d>0$ so that $1/\d\in\NN$, 
consider the grid consisting of the points $x_{i,j}=(i\d, j\d)$, for $0\le i,j\le 1/\d$. It is shown in \cite{Richter} that for given $u\in C^2([0,1]^2)$,
given values on the influx boundary $\partial\O_{1,u}$ and smooth functions $g$ and $h$, there exists a function $f$ that is absolutely continuous
along the characteristics such that $u=u_f$, for $u_f$ solving the Darcy equation \eqref{EqDarcy} (and $f$ is given by \eqref{EqDarcySolutionAlongCharacteristics}).
We shall give a numerically efficient algorithm to construct numbers $\a_{i,j}$ so that $\max_{i,j}|\a_{i,j}-f(x_{i,j})|\lesssim \d^{\eta/2}$, for some $\eta>0$,
thus allowing the recovery of $f$ at any precision. The complexity of the algorithm is linear in the number of grid points.

For a function $u:[0,1]^2\to\RR$, abbreviate $u(x_{i,j})$ to $u_{i,j}$. We are seeking an approximation to the solution $f$
to the discretised equations $(\L_fu)_{i,j}=g_{i,j}$, for   $\L_f$ the operator given by $\L_fu=\nabla f\cdot \nabla u+f\Delta u$.
To this end, consider the approximation $\L_\a^\d$, defined by
\begin{equation}
\label{EqDefLad}
(\L_\a^\d u)_{i,j}=\begin{cases} \a_{i,j}\Delta u_{i,j},& \text{if }\|\nabla u_{i,j}\|< \sqrt\d,\\
\frac{\|\nabla u_{i,j}\|}{\|z_{i,j}\|}(\a_{i,j}-\a_{k,l})+\a_{i,j}\Delta u_{i,j},&\text{if }\|\nabla u_{i,j}\|\ge \sqrt\d,
\end{cases}
\end{equation}
where $(k,l)=(k_{i,j},l_{i,j})$ are the coordinates of another grid point  or a point in $\partial\O_{1,u}$ that is linked to $x_{i,j}$, as defined below,
and $z_{i,j}=x_{i,j}-x_{k,l}$. The  display mimics $(\L_fu)_{i,j}=\nabla f_{i,j}\cdot\nabla u_{i,j}+f_{i,j}\Delta u_{i,j}$, where $f$ is replaced by $\a$. In both cases in the
display the term $f_{i,j}\Delta u_{i,j}$ is simply copied into $\a_{i,j}\Delta u_{i,j}$, but the term $\nabla f_{i,j}\cdot \nabla u_{i,j}$ is approximated. In the first case,
when $\|\nabla u_{i,j}\|< \sqrt\d$, the latter term is approximated by 0, while the second case involves a vector $z_{i,j}$, which is chosen
in the direction of $\nabla u_{i,j}$, as follows.  For $f\in C^{1}$ the approximation $f(x_{i,j})-f(x_{i,j}-v)\approx \nabla f(x_{i,j})\cdot v$, valid for any small vector $v\in\RR^2$,  gives $\nabla f_{i,j}\cdot \nabla u_{i,j}\approx \|\nabla u_{i,j}\|/\|z_{i,j}\|\bigl(f(x_{i,j})-f(x_{i,j}-z_{i,j})\bigr)$, for a small vector $z_{i,j}$ in the direction of $\nabla u_{i,j}$. We arrive at the approximation in the display by choosing small $z_{i,j}$ in the direction of $\nabla u_{i,j}$ so that $x_{k,l}=x_{i,j}-z_{i,j}$ is a point of the grid or a point of the influx boundary. An explicit choice is 
\begin{equation}
\label{EqDefv}
z_{i,j}= \d \Bigl[\frac1{\sqrt\d}\frac{\nabla u_{i,j}}{\|\nabla u_{i,j}\|}\Bigr],
\end{equation}
\def\sgn{\mathop{\rm sign}\nolimits}%
where for a scalar $x$ the notation $[x]=\sgn(x)\lfloor|x|\rfloor$ denotes the  integer closest to $x$ inside the interval $[-|x|,|x|]$,
and  $[x]=([x_1],[x_2])$ for a vector $x=(x_1,x_2)$. This choice always gives a point $x_{i,j}-z_{i,j}$ in the grid $\d\mathbb{Z}^2$. If 
this point falls outside the unit square, then we redefine $z_{i,j}$ so that $x_{k,l}:=x_{i,j}-z_{i,j}$ is the point in $\partial [0,1]^2$
where the line from $x_{i.j}$ to the original $x_{i,j}-z_{i,j}$ crosses this boundary. We shall refer
to the latter type of points $x_{k,l}$ as belonging to the (discretised) influx boundary $\partial_d\O_{1,u}$ 
and assume that the values of $f$ at these points are prespecified.

For points $x_{i,j}$ in the influx boundary $\partial_d\O_{1,u}$ we set the value $\a_{i,j}$ equal to the pre-specified values $f(x_{i,j})$.
For the other points we define $\a_{i,j}$ as the solution to the equation $(\L_\a^\d u)_{i,j}=g_{i,j}$, thus
mimicking the discretised Darcy equation $(\L_f u)_{i,j}=g_{i,j}$. In view of \eqref{EqDefLad} this solution takes the form
\begin{equation}\label{eq:sol:discrete}
\alpha_{ij}=
\begin{cases}
g_{i,j}/\Delta u_{i,j}, & \text{if $\|\nabla u_{i,j}\|<\sqrt\d$,}\\
\frac{ g_{i,j}+\alpha_{k,l}\|\nabla u_{i,j}\|/\|z_{i,j}\| }{\Delta u_{ij}+\|\nabla u_{i,j}\|/\|z_{i,j}\|},& \text{if  $\|\nabla u_{i,j}\|\geq\sqrt\d$}.
\end{cases}
\end{equation}
Under the condition that $C(u)$ given in \eqref{EqCu} is bounded away from zero, these quotients are well defined for small enough $\d>0$.
In the first case of the display, the value $\a_{i,j}$ is expressed explicitly in the known values $g_{i,j}$ and $\Delta u_{i,j}$, but
the second case involves the solution $\a_{k,l}$ at the grid point $x_{k,l}$, in addition to these known values and $\nabla u_{i,j}$. The latter recursion
can be solved along a chain of recursions $x_{i',j'}\to x_{k',l'}$ that connects a point $x_{i,j}$ with $\|\nabla u_{i,j}\|\geq\sqrt\d$
to a point $x_{i_0,j_0}$ which either belongs to the influx boundary or satisfies $\|\nabla u_{i_0,j_0}\|<\sqrt\d$. The value
$\a_{i_0,j_0}$ is then determined and $\a_{i,j}$ can be computed by repeated substitutions. We show below that 
$u_{i,j}>u_{k,l}$ if $\|\nabla u_{i,j}\|\geq\sqrt\d$, so that the equations are then solved in the order of increasing
value of $u_{i,j}$. The strict decrease shows that the chain of recursions $x_{i',j'}\to x_{k',l'}$ cannot visit
any grid point twice and hence must end at a point $x_{i_0,j_0}$ after finitely many steps.

The proof of the following lemma is given in Section~\ref{SectionComplementsDarcy}. 

\begin{lemma}
\label{LemmaDiscretisationDarcy}
For $\O=[0,1]^2$ and given $u \in C^{2+\eta}(\O)$ with $C(u):=\Delta u+\|\nabla u\|^2>0$ and $\eta\in (0,1]$
such that $u=u_f$ for $f\in C^1(\O)$, the preceding algorithm gives $(\a_{i,j})$ such that, for small enough $\d>0$,
$$\max_{0\le i,j\le 1/\d}|\a_{i,j}-f(x_{i,j})|\lesssim \d^{\eta/2},$$
where the multiplicative constant depends on $\|u\|_{C^{2+\eta}(\O)}$, $C(u)$, $\|f\|_{C^1(\O)}$ and  $\|g\|_{C^1(\O)}$ only.
\end{lemma}

\subsection{Multiple measurements}
Suppose given multiple measurements $Y_{n,j}=u_{f,j}+n^{-1/2}\dot\WW_j$, for $u_{f,1},\ldots,u_{f,d}$ solutions
to \eqref{EqDarcy} for $h=0$ and boundary functions $g_1,\ldots, g_d$. If the matrix $[\nabla u_{f,1}\cdots \nabla u_{f,d}]$ is invertible throughout
the domain, then \eqref{EqInverseFormulaHybrid} gives an explicit reconstruction map for $\nabla f/f$. 
We can accommodate this in our setup by replacing the single measurement $Y_n$ by the
vector $(Y_{n,1},\ldots, Y_{n,d})$ and defining a solution operator $e$ as in \eqref{EqSolutionOperator} as operating on the
vector of Laplacians $\bigl(\Delta u_{f,1},\ldots, \Delta u_{f,d}\bigr)$, as follows.

If $K$ is the inverse of the Laplacian $\L=\Delta$ as defined in \eqref{eq: definition K gen}, then
$u_{f,j}=K\Delta u_{f,j}+\tilde g_j$ by \eqref{Equ=KLu+tildeg}, for $j=1,\ldots,d$,
and hence $\nabla u_{f,j}=\nabla K \Delta u_{f,j}+\nabla \tilde g_j$.
Our method is to infer $v_j=\Delta u_{f,j}$ from the inverse problem $Y_{n,j}-\tilde g_j=Kv_j+n^{-1/2}\dot\WW_j$.
Then $\nabla Kv_j+\nabla \tilde g_j$ is  the corresponding estimator of $\nabla u_{f,j}$. This motivates to define
the inverse operator 
\begin{align*}
\bar e(v_1,\ldots,v_d)&=-\bigl[v_1,\ldots, v_d\bigr] \bigl[\nabla K v_1+\nabla \tilde g_1,\ldots, \nabla K v_d+\nabla \tilde g_d\bigr]^{-1}.
\end{align*}
By \eqref{EqInverseFormulaHybrid} and the preceding reasoning
$\nabla f/f=e(\Delta u_{f,1},\ldots, \Delta u_{f,d})$. The map $\bar e$ is Lipschitz on the domain of functions $(v_1,\ldots, v_d)$
such that the norm of the inverse matrices $\bigl[\nabla K v_1+\nabla \tilde g_1,\ldots, \nabla K v_d+\nabla \tilde g_d\bigr]^{-1}$ is uniformly bounded. 
We can compose $\bar e$ with a map that recovers $f$ from $\nabla f/f$, given a boundary value, 
to construct a solution operator \eqref{EqSolutionOperator}.

For a concrete example for the domain  $\O=(0,1)^d$,
we use the prior on the eigen expansion of the Laplacian, as discussed in Section~\ref{sec: Schrodinger}.
Let $(\Hil^s)$ be the smoothness scale generated by the eigenfunctions \eqref{eq: eigenfunctions}.

\begin{theorem}
Assume that $\Delta u_{f_0,1},\ldots,\Delta u_{f_0,d} \in \Hil^\b$, for $\b > d/2-1$ and that the $(d\times d)$ matrix
$\bigl[\nabla u_{f_0,1}(x)\cdots\nabla u_{f_0,d}(x)\bigr]$ is continuously invertible for every $x\in\O$ with
inverses of uniformly bounded norms.
(i). The $L_2$-contraction rate of the posterior distribution of $\nabla f/f$ to $\nabla f_0/f_0$ in the white noise model is 
$n^{-(\a\wedge \b)/(2\a+4+d)}$ if the regularity of the prior is chosen 
deterministic and equal to $\a$ with $\a>d/2-1$.
(ii). This $L_2$-posterior contraction rate is $l_n n^{-\b/( 2\b + 4+d)}$ for a slowly varying sequence $l_n$,
if $\a$ is chosen by the empirical Bayes or hierarchical Bayes methods. 
\end{theorem}

\begin{proof}
Theorems~\ref{theorem: general uniform convergence}  and~\ref{theorem: general uniform convergence_adaptive}
give a posterior contraction rate $n^{-(\a\wedge\b-\d)/(2\a+4+d)}$ for the Laplacians $v_j=\Delta u_{f,j}$ relative
to the $\Hil^\d$ norm and a posterior contraction rate  $n^{-(\a\wedge\b)/(2\a+4+d)}$  for the Laplacians relative to the $L_2$-norm,
with $\a=\b$ in the case of (ii) and provided  $\a\wedge \b>\d$.

From the eigen expansion of $K$ (see \eqref{eq: eigenfunctions}), we find that $\nabla Kv=\sum_{i\in\NN^d}\k_iv_i\nabla h_i$, where
$\nabla h_i$ is the vector-valued function with $k$th coordinate $2^{d/2}\pi i_k\prod_{j\not =k}\sin(i_j\pi x_j)\cos(i_k\pi x_k)$ and
$\k_i\asymp 1/\|i\|^2$.
The latter functions without the factor $\pi i_k$ are orthonormal in $L_2$ and hence 
$\|\nabla Kv\|_{L_2(\O)}\lesssim \|v\|_{L_2(\O)}$ (easily). Furthermore, the same representation shows that
$\|\nabla Kv\|_\infty\le \sum_{i\in \NN^d}\k_i |v_i|\,\|\nabla h_i\|_\infty \lesssim \sum_{i\in \NN^d}\sqrt {\k_i} |v_i|$.
Let $\tilde v_1,\tilde v_2,\ldots$ be the array of values $(v_i: i\in \NN^d)$ ordered in a sequence by decreasing eigenvalues $\k_i$.
Then in view of Lemma~\ref{lemma: sorted eigenvalues}, 
the last series can be bounded by a multiple of $\sum_{\ell=1}^\infty \ell^{-1/d} |\tilde v_\ell|
\le (\sum_{\ell=1}^\infty \ell^{-(2+2\d)/d} )^{1/2}\|v\|_{\Hil^\d}$, by the Cauchy-Schwarz inequality.
For $\d>d/2-1$, the first series converges. We conclude that the operator $\nabla K$ is continuous both 
as an operator $\nabla K: L_2(\O)\to L_2(\O)$ and as an operator $\nabla K: \Hil^\d\to L_\infty(\O)$.

The second shows that the posterior distribution of $\nabla K v_j+\nabla \tilde g_j$ is consistent 
for $\nabla u_{f_0,j}$ relative to the uniform norm. Since the matrices $\bigl[\nabla u_{f_0,1}\cdots\nabla u_{f_0,d}\bigr]$ are
invertible uniformly throughout the domain $\O$, the same is true for the matrices 
$\bigl[\nabla K v_1+\nabla \tilde g_1\cdots \nabla K v_d+\nabla\tilde g_d\bigr]$,
for every $v_1,\ldots, v_d$ in a set of posterior mass tending to one. 

We can conclude that the map $\bar e: L_2(\O)^d\to L_2(\O)^d$ is Lipschitz on a set of posterior probability tending to one, 
so that the contraction rate of the posterior distributions of the Laplacians $v_j=\Delta u_{f,j}$ is carried over
into the same contraction rate for the function $\nabla f/f$, by Proposition~\ref{prop:stability}.
\end{proof}

\begin{remark}
The function $\nabla f/f$ determines $f$ up to a multiplicative constant, and the map $\nabla f/f\mapsto f$ is smooth
provided $f$ is bounded. Thus the contraction rates in the preceding theorem imply the same rates for $f$.
One might hope that the rates for $f$ are actually better, because of the integration involved in the map
$\nabla f/f\mapsto f$. We have not investigated this further.
\end{remark}

\section{Exponentiated Volterra operator}
\label{SectionExponentialVolterra}
For a known positive constant $g$, consider the solution $u_f: [0,1]\to\RR$ of the
ordinary differential equation
\begin{equation}\label{def:ODE}
	\left\{\begin{aligned}
		 u_f' &= f u_f, \qquad && \text{on }  (0,1),\\
		u_f &= g, \qquad && \text{at }  0.
	\end{aligned}\right.
\end{equation}
This equation takes the form \eqref{eq: general PDE} with $\O=(0,1)$ and $\Gamma=\{0\}$,
differential operator $\L u=u'$ and $c(u,f)=uf$. The inverse operator $K$ 
defined in \eqref{eq: definition K gen}  is the standard
Volterra operator $K: L_2([0,1]) \to L_2([0,1])$ defined by
\begin{align*}
	Kv(x) := \int_0^x v(s)\, ds.
\end{align*}
The function $\tilde g: (0,1)\to \RR$ as in \eqref{eq: definition g tilde gen} is the constant $\tilde g=g$.
It is immediate from the differential equation \eqref{def:ODE} that $f=u_f'/u_f$, where we note that
the function $u_f$ is nonzero, as follows from the explicit expression of the solution:
$u_f(x) = g \exp\left( \int_0^x f(s)\, ds \right)$. Thus the solution map
\eqref{EqSolutionOperator} is given by 
$$e(v)=\frac{v}{Kv+ g}.$$
Because $\|Kv\|_{L_2}\le\|Kv\|_\infty\le \|v\|_{L_2}$, the map $e: L_2([0,1]) \to L_2([0,1])$ is Lipschitz with respect to the
$L_2$-norm in a sufficiently small neighbourhood $V=\{v: \|v-v_0\|_{L_2}<c_0\}$  of any 
function $v_0$ such that $Kv_0+ g$ is bounded away from zero.
Therefore $L_2$-contraction rates for $v=u_f'$ in the linear Volterra inverse  problem 
translate immediately into $L_2$-contraction rates for $f$.

One possible prior is the one based on the eigenexpansion.
The eigenfunctions of the operator $K^\tp K: L_2(0,1)\to L_2(0,1)$ take the form (see e.g.\ \cite{Halmos})
$$h_i(x) = \sqrt{2}\cos\bigl((i-1/2)\pi x\bigr),\qquad i\in\NN,$$
with corresponding eigenvalues 
$$\k_i^2 =\frac1{\pi^2(i-1/2)^2}.$$
Hence the operator $K$ is smoothing \eqref{eq:K:smoothing} of order $p=1$ in its eigenscale $(\Hil^\b)_{\b\in\RR}$.

We endow $v=u_f'$ with the prior distribution $v=\sum_{i\ge 1}v_ih_i$, for $v_i\ind N(0, i^{-1-2\a})$,
for fixed $\a$ or with $\a$ determined by the empirical or hierarchical Bayes method.

\begin{theorem}
If $u_{f_0}' \in \Hil^\b$, then the posterior distribution of $f$ contracts in the $L_2$-norm at rate
$\e_n=n^{-(\a \wedge \b)/(1 + 2\a + 2)}$ if $\a$ is chosen fixed and at the rate $\e_n=l_nn^{-\b/(1 + 2\b + 2)}$ for a
slowly varying sequence $l_n$ if $\a$ is chosen by the empirical or hierarchical Bayes methods.
Furthermore, the credible sets \eqref{EqGeneralCredibleSet} with $\tilde C_n(\tilde Y_n)$ the credible ball 
in \eqref{EqCredibleBall} possess frequentist coverage tending to one and diameter of the order $\e_n$.
\end{theorem}

\begin{proof}
Because $K u_{f_0}'+g$ is bounded away from zero, the map $e: L_2([0,1]) \to L_2([0,1])$ is Lipschitz with respect to the $L_2$-norm.
Therefore, the results are an immediate consequence of Theorem~\ref{theorem: general uniform convergence_adaptive}
and Propositions~~\ref{prop:stability} and~\ref{proposition_credible_general}.
\end{proof}

\section{Numerical analysis}\label{sec: simulations}
We demonstrate the practical performance of our approach on the time-independent Schr\"odinger equation 
\eqref{eq: EqSchr\"odinger}, discussed in Section~\ref{sec: Schrodinger}. We consider synthetic data sets 
generated from various choices of true function $f_0$ in the Gaussian white noise observational model 
\eqref{eq: observation} and investigate the behaviour of the posterior distributions
corresponding to different choices of the prior. 

First we consider the one-dimensional case ($d=1$) with domain the unit interval. We set the boundary conditions 
for $u_f$ as $g(0) = 1$ and $g(1) = 2$, so that the function $\tilde g$ in \eqref{eq: definition g tilde gen} is given by
$\tilde{g}(x) = 1 + x$. The results are shown in Figures~\ref{fig: adaptation and zero boundary} and~\ref{fig: nonzero boundary}.
Each subpanel of the figures illustrates the posterior distribution corresponding to a single simulation of data.
From top to bottom, the four panels correspond to signal-to-noise ratios $n = 10^{4}, 10^{6}, 10^{8}$, and $10^{10}$.
From left to right the panels in the two figures correspond to different functions $f_0$ (Figure~\ref{fig: adaptation and zero boundary}) 
or different prior distributions (Figure~\ref{fig: nonzero boundary}).

For each plot, we took $500$ draws from the posterior distribution of the Laplacian $v=\Delta u_f$, 
and transformed these into draws of $f$ using the solution operator \eqref{eq: inversion formula}. 
We approximated the posterior mean of $f$ as the average
of these 500 draws (the red curves), as well as pointwise $2.5\%$ and $97.5\%$ quantiles to form credible bands (the green curves).
In each panel we also plot 20 draws selected randomly out of the 500 draws to visualize typical samples from the posterior
(shown as  dashed gray curves).

In Figure~\ref{fig: adaptation and zero boundary} the true functions (drawn in black) were 
$f_0(x) = \sum_{i=1}^{\infty} i^{-3/2}\sin(i) h_i$ (left four panels) 
and $f_0(x) = 1 - 4(x-1/2)^2 - \tfrac{3}{4}\exp(-500(x-1/2)^2)$ (right four panels). 
The first function is in $\Hil^\b$, for all  $\b < 1$. 
The prior was based on the eigenexpansion of the Laplacian with Gaussian coefficients
\eqref{eq: prior} and regularity hyper-parameter $\a$ set with the empirical Bayes approach 
by maximising the marginal likelihood function given in \eqref{eq: log likelihood Y given a}. 
For the first function the approach works well: the posterior mean provides an accurate estimate of the true
function and the credible band contains the true function $f_0$, thus resulting in reliable uncertainty quantification
from the frequentist perspective. For the second function the performance is less satisfying:
a larger sample size (smaller signal-to-noise ratio) is needed to detect the bump that is present in this function.
However, this behaviour appears to be less connected to the method as to the problem:
it is hard to detect local spikes, especially if the regularity of the function $f_0$ is unknown. 
Already the direct estimation problem using Gaussian process priors of the form \eqref{eq: prior} needs large sample
sizes.

\begin{figure}
	\includegraphics[width=0.65\textwidth]{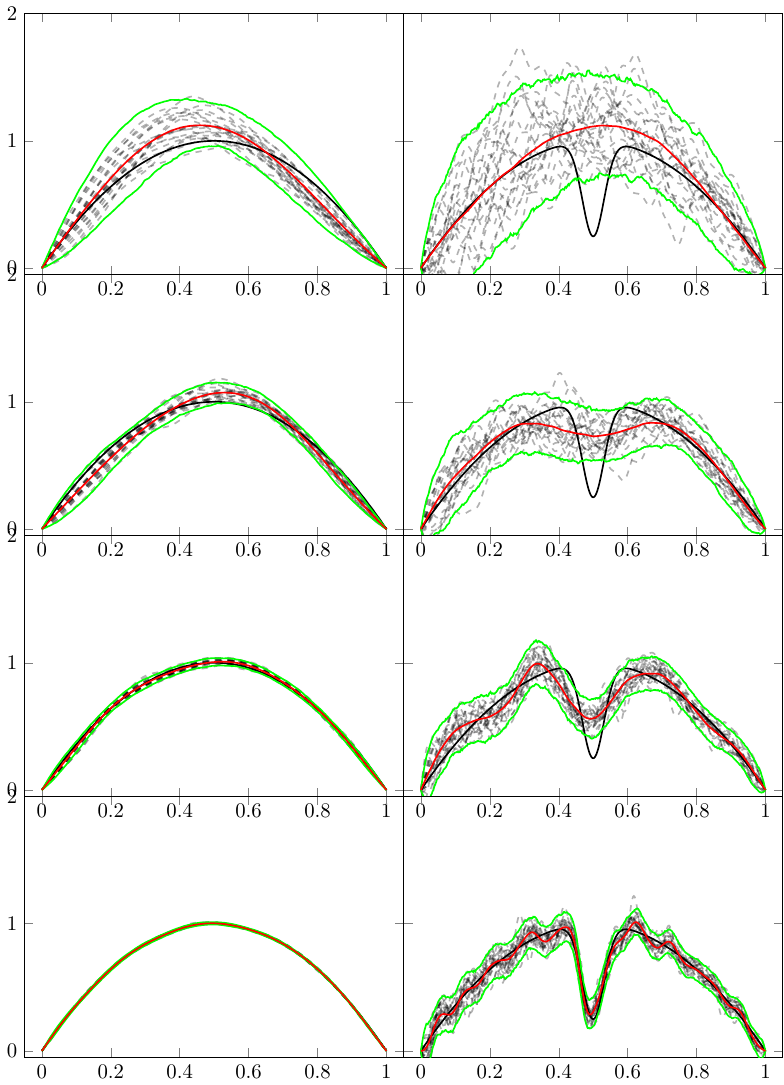}
\caption{Posterior distributions of $f$ based on data $Y_n=u_f+n^{-1/2}\dWW$ for  $n = 10^{4}, 10^{6}, 10^{8}, 10^{10}$
(top to bottom) and prior based on eigenbasis.
Black curve: true function $f_0$. Red curve: posterior mean. Green curves: pointwise  $95$\% credible band. Dashed black curves:
$20$ draws from the posterior distribution. Left panels: $f_0(x) = 1+4(x-1/2)^2$, right panels: $f_0(x) = 1 - 4(x-1/2)^2 - \tfrac{3}{4}\exp(-500(x-1/2)^2)$.}
	\label{fig: adaptation and zero boundary}
\end{figure}

The two true functions in Figure~\ref{fig: adaptation and zero boundary} vanish at the boundary,
which agrees with the prior set on the singular basis of the inverse Laplace operator.
In Figure~\ref{fig: nonzero boundary} we investigate the function $f_0(x) = 1 + \sin(3\pi x)$,
which takes the value 1 at both boundary points.
Our theoretical results show that even for this function our posterior contracts around the truth with the optimal $L_2$-norm and the $L_2$-credible ball provides high frequentist coverage. However, the pointwise behavior at the boundary using the singular value decomposition basis fails, as the posterior is supported on functions vanishing at the boundary. This is demonstrated in the four panels in the left column of 
Figure~\ref{fig: nonzero boundary}, which is based on the same prior as previously.
The posterior distribution shows good performance except near the boundary.
To overcome the boundary problem we also implemented a prior based on the B-spline basis, which does not restrict the
boundary. The resulting posterior distribution is shown in the right column of Figure~\ref{fig: nonzero boundary}.
Both the estimation and uncertainty quantification are more accurate at the boundary.

\begin{figure}
\includegraphics[width=0.65\textwidth]{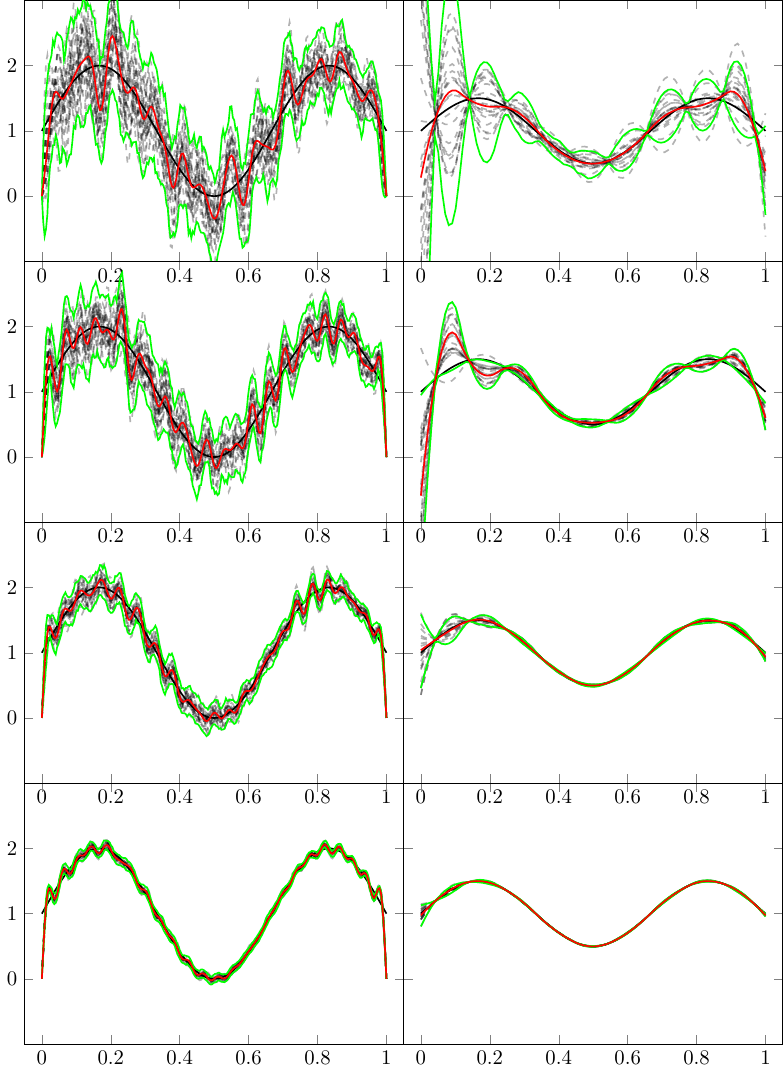}
\caption{Posterior distributions of $f$ based on data $Y_n=u_f+n^{-1/2}\dWW$ for  $n = 10^{4}, 10^{6}, 10^{8}, 10^{10}$
(top to bottom).
Black curve: true function $f_0(x) = \sum_{i=1}^{\infty} i^{-3/2}\sin(i) h_i$. 
Red curve: posterior mean. Green curves: pointwise  $95$\% credible band. Dashed black curves:
$20$ draws from the posterior distribution. Left panels: prior based on eigenbasis. Right panels: prior based on spline basis.}
	\label{fig: nonzero boundary}
\end{figure}

Next we demonstrate the behaviour of our approach in a two-dimensional example, with
true function and boundary condition on the domain $(0,1)^2$ given by
\begin{align}\label{def:f0:2d}
	f_0(x,y) &= 2x(x-1)y(y-1)(2+\sin(3\pi x)\sin(\pi y),\\
	g(x,y) &= 3 + xy^2 + 2y\sin(2\pi x) + x\cos(3\pi y).\nonumber
\end{align}
The four panels in the first column of Figure~\ref{fig: 3d plot} all show a heat map of the function $f_0$.
The other columns show heat maps of different aspects of the posterior distribution, each time based on
single data obtained using increasing signal-to-noise ratio $n = 10^{4}, 10^{6}, 10^{8}$ and $10^{10}$, from
top to bottom: the posterior mean (second column),  lower and upper $2.5\%$ pointwise quantiles of the posterior (third and fourth
column), and the absolute difference between the posterior mean and the true function (fifth column).
The prior was based on the eigenbasis \eqref{eq: eigenfunctions} of the Laplacian 
with Gaussian coefficients  \eqref{eq: prior} and regularity $\a$ set by the empirical Bayes method. 
One can observe that similarly to Figure~\ref{fig: adaptation and zero boundary},
we get good approximation and reliable uncertainty quantification, in this case even pointwise as $f_0$ vanishes at the boundary. 

\begin{figure}
	\includegraphics[width=\textwidth]{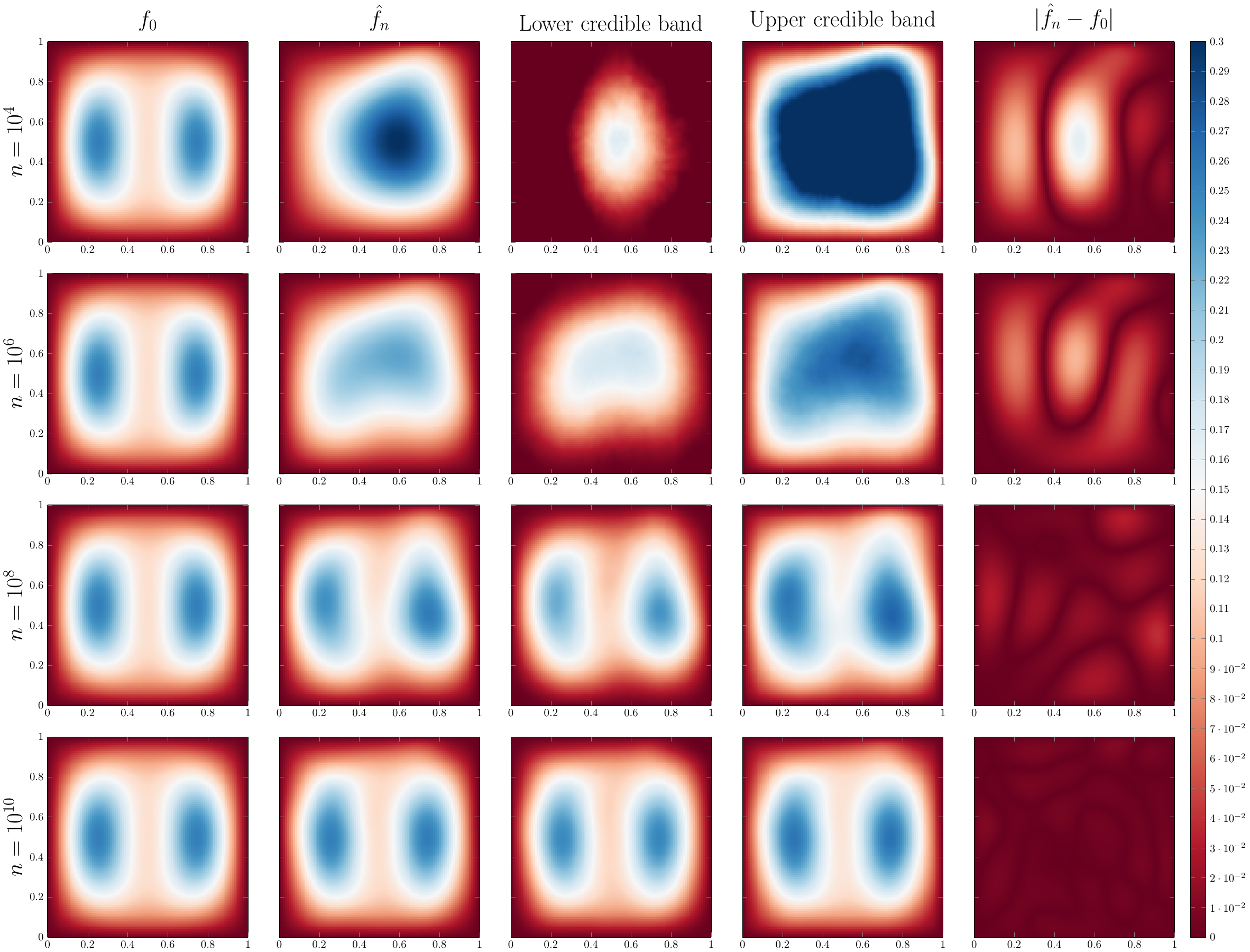}
	\caption{The true function $f_0$ given in \eqref{def:f0:2d}, the posterior mean, lower and upper $2.5\%$ pointwise posterior 
quantiles, and absolute approximation error of posterior mean,
resulting from the proposed Bayesian approach with signal-to-noise ratios $n = 10^{4}, 10^{6}, 10^{8}, 10^{10}$ 
in the PDE constrained  Gaussian white noise model  \eqref{eq: general PDE}-\eqref{eq: observation}.}
	\label{fig: 3d plot}
\end{figure}

\section{Discussion}\label{sec:discussion}
We have derived a novel, Bayesian linearisation method for non-linear inverse problems. Our approach potentially decreases computation time, as routines for solving linear inverse problems are simpler and faster than for  non-linear ones. Even if the linear inverse problem is non-conjugate (e.g.\ if the prior is not set on the singular value basis of $K$) and iterative methods are needed, our method may be substantially faster as it is does not require to solve the forward problem at each iteration step. 

We derived posterior contraction rates for Gaussian process priors, with both fixed and data driven regularity hyper-parameters (using either the empirical or hierarchical Bayes methods). Our results on adaptation are the first one on PDE-constrained, non-linear inverse problems. We also showed that the frequentist coverage guarantees of credible sets in the corresponding linear inverse problems translates to the non-linear inverse problems. Our results apply both to the benchmark Gaussian white noise model and to the practically more relevant discrete observational framework. Our results can be extended to priors going beyond the eigenfunctions of the differential operator $\L$. Our numerical analysis on synthetic data sets shows that the proposed approach has good numerical properties, although
proper handling of boundary conditions appears to require care.

A potential bottleneck of our approach is the solution operator \eqref{EqSolutionOperator}. As a proof of concept, we have covered a collection of benchmark PDE and ODE problems, where this operator takes an explicit form, as well as the Darcy flow problem, where the solution operator only exists in the form of a numerical algorithm. In general,
a numerical implementation of the operator may be necessary. We can borrow here from the field of (numerical) analysis,
where this is an active research area. Another interesting question for future research is scaling up our approach with various machine learning techniques, including distributed and variational inference. Our linearised setup provides a convenient framework to derive such approximation algorithms compared to the original non-linear setting. 

\section{Acknowledgments}
We thank Richard Nickl for sharing his insights during multiple discussions and pointing out references,
in particular regarding the Darcy problem.

\appendix

\section{Hilbert scales and Sobolev spaces}
\label{SectionHilbertAndSobolev}
Recall that $\Sob^s(\O)$ is the Sobolev space of order $s$ and that $\Sob_0^s(\O)$
is the subset of functions that vanish on the boundary of $\O$ (i.e.\ have vanishing trace).

Let $\L$ be a symmetric, second order, elliptic differential operator with $C^\infty$ coefficient functions
such that the function $v=0$ is the only solution to the problem $\L v=0$ on $\mathcal {O}$ and $v=0$ on $\partial\O$. For $u\in L_2(\O)$, let $Ku\in \Sob_0^1(\O)$ be the weak solution to 
\begin{equation*}
	\left\{\begin{aligned}
		\L Ku &= u, \qquad &&\text{ on } \O,\\
		Ku&= 0, \qquad &&\text{ on } \partial\O.
	\end{aligned}\right.
\end{equation*}
Then $K: L_2(\O)\to L_2(\O)$ is self-adjoint and compact (it is even bounded as an operator
$K: L_2(\O)\to \Sob_0^1(\O)$; see \cite{Evans10}, Theorem~6 in Section~6.2). Given its eigenfunctions $(h_j)$ and
eigenvalues $\k_j$ and $s\ge 0$,  let $\Hil^s$ be the set of functions $v=\sum_jv_jh_j\in L_2$ for which
$\|v\|_{\Hil^s}^2:=\sum_{j=1}^\infty v_j^2\k_j^{-s}$ is finite. For $s<0$, let $\Hil^s$ be the dual space of $\Hil^{-s}$.

\begin{proposition}
\label{PropositionSmoothingEllipticH0}
For  $s\ge 0$ such that $\partial \O\in C^{\lceil s\rceil+2}$, we have $\Hil^s\subset \Sob^s(\O)$, 
with $\|u\|_{\Sob^s(\O)}\lesssim \|u\|_{\Hil^s}$, for every $u\in \Hil^s$, with
$\|u\|_{\Sob^s(\O)}\asymp\|u\|_{\Hil^s}$ for every even integer $s$.
For $s=2$ we have $\Hil^2=\Sob^2(\O)\cap \Sob^1_0(\O)$ if $\partial \O\in C^2$.
\end{proposition}

\begin{proof}
The space $\Hil^s$ is equivalent to the set of functions $K^s v$, for $v\in L_2(\O)$, with the norm
$\|K^sv\|_{\Hil^s}=\|v\|_{L_2(\O)}$. 

For every $m\in\NN\cup\{0\}$, 
 under the condition that $\partial \O\in C^{m+2}$, 
Theorem~5 in \cite{Evans10} gives that $Ku\in \Sob^{m+2}(\O)$ if $u\in \Sob^m(\O)$,
with $\|Ku\|_{\Sob^{m+2}(\O)}\lesssim \|u\|_{\Sob^m(\O)}$. 

For $m=0$, this shows that $\Hil^2=KL_2\subset \Sob^2(\O)$ and $\|Ku\|_{\Sob^2(\O)}\lesssim \|u\|_{L_2(\O)}\asymp \|Ku\|_{\Hil^2}$.
By the definition of a weak solution with Dirichlet boundary condition, it is
also true that $KL_2\subset \Sob_0^1$ and hence $\Hil^2\subset \Sob^2(\O)\cap \Sob^1_0(\O)$
with $\|u\|_{\Sob^2(\O)}\lesssim \|u\|_{\Hil^2}$. 
Conversely if $u\Sob^2(\O)$, then $\L u\in L_2$ by the definitions of $\Sob^2(\O)$ and $\L$
(or see (5.12) in Chapter~4.5 of \cite{Taylor2011}). If also $u\in\Sob_0^1$,
then $w=u$ solves $\L w=\L u$ (trivially) on $\O$ and $w=0$ on $\partial \O$ and hence
$u=K\L u\in \Hil^2$, with norm $\|u\|_{\Hil^2}=\|\L u\|_{L_2}\lesssim \|u\|_{\Sob^2(\O)}$.
This concludes the proof that $\Hil^2=\Sob^2(\O)\cap \Sob^1_0(\O)$ and that the norms
are equivalent on this space.

Next we prove by induction that $\Hil^{2m}\subset \Sob^{2m}(\O)$ with equivalent norms,
for every $m\in\NN$. For $m=1$ the assertion is proved. 
From the definitions it follows that $\Hil^{2m+2}=K\Hil^{2m}$ with norms satisfying
$\|K^{m+1}u\|_{\Hil^{2m+2}}=\|u\|_{L_2}=\|K^mu\|_{\Hil^{2m}}$.
By Theorem~5 in \cite{Evans10} as quoted previously $K^{m+1}u= K(K^mu)$ is contained
in $\Sob^{2m+2}(\O)$ with norms 
satisfying $\|K^{m+1}u\|_{\Sob^{2m+2}(\O)}\lesssim \|K^mu\|_{\Sob^{2m}(\O)}$. 
Therefore if the assertion is true for $m$, then it is true for $m+1$.

The Sobolev spaces $\Sob^s(\O)$ for non-integer $s\ge0$ can be constructed as the 
interpolation spaces $[L^2, \Sob^m(\O)]_{s/m,2}$, for $s\le m\in \NN$ (see \cite{Adams}, Section~7.57 or
\cite{BerghLofstrom}.) This means that the norm 
$\|u\|_{\Sob^s(\O)}$ of $u\in\Sob^s(\O)$
is equivalent to
$$\biggl(\int_0^\infty \Bigl(\frac1{t^{s/m}}\inf _{u=u_0+u_m}\bigl(\|u_0\|_{L_2}+t\|u_m\|_{\Sob^m(\O)}\bigr)\Bigr)^2\,\frac{dt}{t}\biggr)^{1/2},$$
where the infimum is taken over all decompositions $u=u_0+u_m$ with $u_0\in L_2$ and $u_m\in\Sob^m(\O)$,
and the space $\Sob^s(\O)$ is exactly the set of functions $u$ for which the expression is finite.
By direct argument it can be seen that the space $\Hil^s$ is the interpolation
$[L^2, \Hil^m]_{s/m,2}$. Since $\Hil^m\subset \Sob^m(\O)$, for even $m$, with equivalent norms,
it follows from the interpolation formula that $\Hil^s\subset \Sob^s(\O)$ with $\|u\|_{\Sob^s(\O)}\lesssim \|u\|_{\Hil^s}$.
\end{proof}

\begin{lemma}\label{lemma: interpolation Hs}
For $0\le s \le m$, the space $\Hil^s$ is the interpolation space $\Hil^s = [L_2,\Hil^m]_{s/m,2}$.
\end{lemma}

\begin{proof}
Let $K(t,u)=\inf \bigl\{\|u_0\|_{L_2}+t\|u_m\|_{\Hil^m}: u=u_0+u_m\}$, where the infimum is taken
over $u_0\in L_2$ and $u_m\in \Hil^m$.
Because $a^2+b^2t^2\le (a+bt)^2\le 2a^2+2b^2t^2$, for any $a,b,t\ge0$, 
\begin{align*}
K^2(t,u)&\asymp\inf _{u=u_0+u_m}\bigl(\|u_0\|_{L_2}^2+t^2\|u_m\|_{\Hil^m}^2\bigr)\\
&=\inf_{u=u_0+u_m}\bigl(\sum_i(u_{i,0}^2+u_{m,i}^2t^2\k_i^{-m}\bigr)
=\sum_i u_i^2\wedge (u_i^2t^2\k_i^{-m}).
\end{align*}
It follows that the square interpolation norm is given by
\begin{align*}
\int_0^\infty\frac{K^2(t,u)}{t^{2s/m}}\,\frac{dt}t
&=\sum_i u_i^2\int_0^\infty\frac{1\wedge(t^2\k_i^{-2m})}{t^{2s/m}}\,\frac{dt}t\\
&=\sum_i u_i^2\Bigl(\int_0^{\k_i^{m/2}}t^{1-2s/m}\,dt\k_i^{-m}+
\int_{\k_i^{m/2}}^\infty \frac1{t^{2s/m+1}}\,dt\Bigr).
\end{align*}
The right side can be seen to be up to a multiple equivalent to $ \sum_i u_i^2\k_i^{-s}=\|u\|_{\Hil^s}^2$.
\end{proof}

\section{Complements for Schr\"odinger equation}
\label{SectionProofsSchrodinger}

\begin{lemma}\label{lem:smoothness:Luf}
Let $K: L_2\to L_2$ be a self-adjoint linear operator 
with eigenvalues $\k_j\asymp j^{-p/d}$ and such $K^{-\b/p}(\Sob^\b(\O))\subset L_2$.  
If $f \in \Sob^\b(\O)$  and $u_f \in \Sob^\b(\O)$ for some  $\b>d/2$, then $K^{-1}(fu_f) \in \Hil^\b$,
for $(\Hil^s)$ the Hilbert scale generated by the eigenfunctions and eigenvalues of $K$.
\end{lemma}

\begin{proof}
For every eigenfunction $h_j$ and $s\in\RR$, we have  $K^{-s}h_j=\k_j^{-s}h_j$.
Thus the self-adjointness of $K$ gives 
	\begin{align*}
		\langle f u_f,  h_j \rangle_{L_2}
		= \langle f u_f, \k_j^{\b/p} K^{-\b/p} h_j \rangle_{L_2}
		&= \k_j^{\b /p} \langle K^{-\b/p} (f u_f), h_j \rangle_{L_2}.
	\end{align*}
Since $\k_j\asymp j^{-p/d}$, it follows that 
\begin{align*}
	 \sum_{j=1}^{\infty} \langle f u_f, h_j \rangle_{L_2}^2 j^{2\b/d} 
\asymp \sum_{j=1}^{\infty} \langle K^{-\b/p} (f u_f), h_j \rangle_{L_2}^2.
\end{align*}
If $f \in \Sob^\b(\O)$ and $u_f \in \Sob^\b(\O)$, then $f u_f\in\Sob^\b(\O)$, since $\b>d/2$,
and hence $K^{-\b/p}(fu_f)\in L_2$, by assumption. Thus the right side of the display is finite,
and hence so is the left side, which implies $fu_f\in\Hil^\b$.
\end{proof}

\begin{lemma}\label{lemma: sorted eigenvalues}
Let $\{\k_i: i\in \NN^d\}$ be numbers with $C_1/\|i\|^p\le \k_i\le C_2/\|i\|^p$, for positive constants $C_1, C_2$ and $p>0$ and
$\|i\|^2=\sum_{j=1}^di_j^2$, for every $i=(i_1,\ldots, i_d)\in\NN^d$, and let $k_1\geq k_2\ge\cdots$  
denote the same values $\k_i$ sorted in decreasing order.
Then there exist positive constants $\tilde{C}_1,\tilde{C}_2$ that depend on $C_1,C_2,d$ only such that 
		\begin{align*}
			\tilde{C}_1 \ell^{-p/d} \leq k_\ell\leq \tilde{C_2} \ell^{-p/d},\qquad \ell\in \NN.
		\end{align*}
Furthermore, if $v_1,v_2,\ldots$ are the values of the multi-indexed array  $\{\nu_i: i \in \NN^d\}$
in corresponding order, then $\|v\|_{\Hil^\b}^2 =\sum_{\ell=1}^\infty\ell^{2\b/d}v_\ell^2$ satisfies
\begin{align*}
	\|v\|_{\Hil^\b}^2 \asymp \sum_{i \in \NN^d} \|i\|^{2\b}\nu_{i}^2 .
	\end{align*}
\end{lemma}

\begin{proof}
Let $z: \NN^d \to \NN$ be a bijection that corresponds to an ordering of the numbers $\k_i$. We show below 
that $z(i)\asymp \|i\|^d$, where the multiplicative constants depend on $C_1,C_2,d$ only. 
If this is true, then $k_{z(i)}=\k_i\asymp \|i\|^{-p}$, by the assumption on the $\k_i$, which is $z(i)^{-p/d}$, and the first
assertion of the lemma is proved.
Furthermore $\sum_{\ell\in \NN}\ell^{2\b/d}v_\ell^2=\sum_{i\in\NN^d}z(i)^{2\b/d}v_{z(i)}^2\asymp \sum_{i\in\NN^d}\|i\|^{2\b}\nu_i^2$,
proving the second assertion of the lemma.

It remains to prove that $z(i)\asymp \|i\|^d$, which we achieve by a separate lower and upper bound.
If $j\in\NN^d$ is such that $\|j\|\le d_1\|i\|$ for a sufficiently small constant $d_1>0$ that depends on $C_1,C_2,d$ only, then
the assumption on the numbers $\k_i$ shows that $\k_j>\k_i$ and hence $z(j)\le z(i)$. It follows that $z(i)\ge \#\bigl(j\in\NN^d: \|j\|\le d_1\|i\|\bigr)$,
which can be seen to be of the order $(d_1\|i\|)^d$ up to multiplicative constants depending on $d$ only. 
Conversely, because every $j\in\NN^d$ with $z(j)\le z(i)$ has
$\k_j\ge \k_i$, we have $z(i)= \#\bigl(j\in\NN^d: z(j)\le z(i)\bigr)\le \#\bigl(j\in\NN^d: \k_j\ge \k_i\bigr)$, which is bounded above by
$\#\bigl(j\in\NN^d: \|j\|\le d_2\|i\|\bigr)$ for some constant $d_2$ by the assumption on the numbers $\k_i$, which is bounded
above by a multiple of $\|i\|^d$.
\end{proof}

For convenience we restate and then prove Lemma  \ref{lemma: interpolation} below.
\begin{lemma}[Lemma \ref{lemma: interpolation}]\label{lemma: interpolation2}
Let  $(\Hil^s)$ be the scale generated by the eigenfunctions \eqref{eq: eigenfunctions} and 
let $\LL_n := \bigl\{\sum_{i\in\NN^d, \|i\|_\infty \leq n^{1/d}} c_i h_i: c_i \in \RR\big\}$.
If $\b > d/2$, then for every $v\in\Hil^\b$,
there exists an element $\mathcal{I}_n v \in \LL_n$ 
such that $\mathcal{I}_n v(x) = v(x)$ for every $x\in \{(\frac{2i}{2m+1})_{i=1,...,m}\}^d$, and
		\begin{equation}\label{eq: interpolation assumption2}
			\| \mathcal{I}_n v - v\|_{L_2} \lesssim  n^{-\b/d}\|v\|_{\Hil^\b}.
		\end{equation}
Furthermore, there exist constants $0<C_1<C_2<\infty$ such that
\begin{align*}
 C_1 \|v\|_{L_2} \leq \|v\|_{\LL_n} \leq C_2 \|v\|_{L_2},\qquad \forall v\in \LL_n.
\end{align*}
\end{lemma}

\begin{proof}
The proof of \eqref{eq: interpolation assumption2} is based on Section~2.3 in \cite{Reiss2008}. 
This part of the proof is not restricted to the singular value basis of the inverse Laplace operator, 
but holds more generally. For convenience we use the notation $m=n^{1/d}$ and define $h_i=0$, for any $i\in\mathbb{Z}^d$ such that $i_j<0$ for some $j\in\{1,...,d\}$. Furthermore, let $P_n: L_2([0,1]^d) \to \LL_n$ be the orthogonal projection on $\LL_n$.
	Then by triangle inequality 
	\begin{align}
	\|\mathcal{I}_n g - g\|_{L_2} \leq \|\mathcal{I}_n f - P_n f\|_{L_2} + \|P_n f - f\|_{L_2}.\label{eq:triangle:help}
	\end{align}
	The second term on the right-hand side can be bounded as
	\begin{align*}
		\| P_n f - f\|_{L_2}^2
		= \sum_{\|l\|_{\infty} > m } \langle f, h_l \rangle_{L_2}^2
		\leq \sum_{\|l\|_{\infty} > m} \langle f, h_l \rangle_{L_2}^2 \frac{\|l\|_{\infty}^{2\b}}{m^{2\b}}
		\leq m^{-2\b} \| f \|_{\Hil^\b}^2.
	\end{align*}
For the first term on the right-hand side of \eqref{eq:triangle:help}, note that
	\begin{align*}
		\| \mathcal{I}_n f - P_n f \|_{L_2}^2
		&= \sum_{\substack{l \in \mathbb{Z}^d \\ \|l\|_{\infty} \leq m}} \Big(\sum_{k \in\NN_0^d\backslash\{0\}} \langle f, h_{k(2m+1)+l}\rangle\Big)^2\\
		&\leq \sum_{\substack{l \in \mathbb{Z}^d \\ \|l\|_{\infty} \leq m}} \Big[ \Big(  \sum_{k \in\NN_0^d\backslash\{0\}} \|k(2m+1)+l\|_{\infty}^{2\b} \langle f, h_{k(2m+1)+l}\rangle^2\Big)\\
		&\qquad\qquad\qquad\times \Big( \sum_{k \in\NN_0^d\backslash\{0\}} \|k(2m+1) + l\|_{\infty}^{-2\b}\Big)\Big]\\
		&\lesssim \|f\|_{\Hil^\b}^2  \sup_{\substack{l \in \mathbb{Z}^d \\ \|l\|_{\infty} \leq m}} \Big(\sum_{k \in\NN_0^d\backslash\{0\}} \|k(2m+1)+l\|_{\infty}^{-2\b}\Big),
	\end{align*}
	where the last inequality follows from  $\|v\|_{\infty} \asymp \sqrt{\sum_{j=1}^d v_j^2}$ and Lemma~\ref{lemma: sorted eigenvalues}.
	Since for $\|l\|_{\infty} \leq m$ by triangle inequality
	\begin{align*}
		m^{-1} \| k(2m+1) + l\|_{\infty}
		\geq \|2k\|_{\infty} - \|l/m\|_{\infty}
		\geq 2\|k\|_{\infty} - 1,
	\end{align*}
	we find that
	\begin{align*}
		\| \mathcal{I}_n f - P_n f \|_{L_2}^2
		&\lesssim \|f\|_{\Hil^\b}^2 m^{-2\b} \Big(\sum_{k \in\NN_0^d\backslash\{0\}} (2\|k\|_{\infty} - 1)^{-2\b}\Big).
	\end{align*}
	As seen in the proof of Lemma~\ref{lemma: sorted eigenvalues}, for $\b >d/2$ we have
	\begin{align*}
		\sum_{k \in\NN_0^d\backslash\{0\}}(2\|k\|_{\infty} - 1)^{-2\b}
		\asymp \sum_{k \in\NN_0^d\backslash\{0\}} \|k\|_{\infty}^{-2\b}
		\asymp \sum_{i=1}^{\infty} i^{-2\b/d}=O(1).
	\end{align*}
	We conclude the proof of the first statement by applying the upper bounds derived above for the two terms on the right-hand side of \eqref{eq:triangle:help} with $m=n^{1/d}$.

To show that $\mathcal{I}_nf$ coincides with $f$ on the design points, note that for  $x\in \mathbb{X}=\{(\frac{2i}{2m+1})_{i=1,...,m}\}^d$, 
	\begin{align*}
		f(x)
		&= \sum_{\substack{l \in \mathbb{Z}^d \\ \|l\|_{\infty} \leq m}}\langle f, h_l\rangle_{L_2} h_l(x) +  \sum_{\substack{l \in \mathbb{Z}^d\\ \|l\|_{\infty} \leq m}} \sum_{k \in \NN_0^d\backslash\{0\}}  \langle f, h_{k(2m+1)+l}\rangle_{L_2} h_{k(2m+1)+l}(x)\\
		&= \sum_{\substack{l \in \mathbb{Z}^d \\ \|l\|_{\infty} \leq m}}\langle f, h_l\rangle_{L_2} h_l(x) + \sum_{\substack{l \in \mathbb{Z}^d\\ \|l\|_{\infty} \leq m}} \Big(\sum_{k \in \NN_0^d\backslash\{0\}} \langle f, h_{k(2m+1)+l} \rangle_{L_2}\Big) h_{l}(x),
	\end{align*}
	where the last equality follows from
	\begin{align*}
		h_{k(2m+1)+l}(x)
		&= 2^{d/2} \prod_{j=1}^d \sin((k(2m+1)+l)_j \pi x_{j})\\
		&= 2^{d/2} \prod_{j=1}^d \sin(l_j \pi x_{j} +  2k_j \pi r_j)
		= 2^{d/2} \prod_{j=1}^d \sin(l_j \pi x_{j})
		= h_{l}(x),
	\end{align*}
	where $r_j= x_j(2m+1)/2\in\NN$.

Finally, to prove the equivalence of the $L_2$-norm with the empirical Euclidean norm $\LL_n$, note that in view of Lemma~\ref{lemma: inner product},  for any $f = \sum_{\|l\|_{\infty} \leq m} c_l h_l \in \LL_n$,
	\begin{align*}
		\|f\|_{\LL_n}^2
		&= \sum_{\|l\|_{\infty} \leq m}\sum_{\|k\|_{\infty} \leq m}c_l c_k\frac{1}{n}\sum_{x \in \mathbb{X}}  h_l(x) h_k(x)\\
		&= \Big(1+\frac{1}{2m}\Big)^d\sum_{\|l\|_{\infty} \leq m} c_l^2
		= \Big(1+\frac{1}{2m}\Big)^d \|f\|_{L_2}^2.
	\end{align*}
\end{proof}

\begin{lemma}\label{lemma: inner product}
Let $ h_l(x)=2^{d/2}\prod_{j=1}^d \sin(l_j\pi x_j)$ for $l\in\NN^d$ and $x\in[0,1]^d$. Then for arbitrary $l, k \in \NN^d$ with $\|l\|_{\infty},\|k\|_{\infty} \leq m$, we have
	\begin{align*}
		\langle h_l, h_k \rangle_{\LL_n} = \begin{cases}
			(1+\frac{1}{2m})^d, & \text{if}\,\, l = k,\\
			0, & \textnormal{otherwise,}
		\end{cases}
	\end{align*}
where the empirical measure $\LL_n$ is defined on the equidistant grid $\mathbb{X}=\{(\frac{2i}{2m+1})_{i=1,...,m}\}^d$.\end{lemma}

\begin{proof}
First we consider the $d=1$ case and then extend the result for general $d\in\NN$. By standard trigonometric equivalence, for all $l,k\in \NN$
	\begin{align*}
		\frac{1}{m}\sum_{x\in\mathbb{X} } 2\sin\left( l \pi x\right) \sin\left(k \pi x\right)
		&= \frac{1}{m}\sum_{x\in\mathbb{X} }  \cos\left( (l - k) \pi x \right) - \cos\left((l + k)\pi x\right).
	\end{align*}
	Then, in view of Lemma~\ref{lemma: sum cosines}, for $1\leq l,k\leq m$, the right-hand side of the preceding display is equal to
	\begin{align*}
		\frac{1}{m}&\Big(m1_{2m+1| (l-k)}  - \frac{1}{2}1_{2m+1 \nmid (l-k)} - \big[m1_{2m+1|(l+k)} - \frac{1}{2}1_{2m+1 \nmid (l+k)}\big]\Big)\\
		&= \frac{1}{m}\left(m1_{l=k} - \frac{1}{2}1_{l \neq k} + \frac{1}{2}\right)
		= (1+\tfrac{1}{2m})1_{l = k}.
	\end{align*}
This finishes the proof for $d=1$.

To extend the results to $d>1$, note that
	\begin{align*}
		\langle h_l, h_k \rangle_{\LL_n}
		&= \frac{1}{n} \sum_{x\in \mathbb{X}} 2^{d/2}\prod_{j=1}^d \sin(l_j \pi x_j) 2^{d/2}\prod_{j=1}^d \sin(k_j \pi x_j)\\
		&= \prod_{j=1}^d \frac{1}{m} \sum_{x_j\in\mathbb{X}_j } 2\sin\left(l_j \pi x_j\right)\sin\left(k_j \pi x_j\right)
		= \begin{cases}
			(1+\frac{1}{2m})^d & l = k,\\
			0, & \text{otherwise.}
		\end{cases}
	\end{align*}
This concludes  the proof of the lemma.
\end{proof}

\begin{lemma}\label{lemma: sum cosines}
	For any $r,m \in \NN$,
	\begin{align*}
		\sum_{j=1}^m \cos\left(r \pi \frac{2j}{2m+1}\right) =
		\begin{cases}
			m, & 2m +1 | r,\\
			-\frac{1}{2}, & 2m + 1 \nmid r.
		\end{cases}
	\end{align*}
\end{lemma}

\begin{proof}
	For  $r = v (2m+1)$ with $v \in \NN$, we have
	\begin{align*}
		\sum_{j=1}^m \cos\left(r \pi \frac{2j}{2m+1}\right)
		= \sum_{j=1}^m \cos\left( 2 v \pi j \right)
		= m.
	\end{align*}
	For $r \in \NN$ such that $r \neq v(2m+1)$, for any $v \in \NN$, and $z = \exp\left(r i \pi \frac{2}{2m+1}\right)$, we have
	\begin{align*}
		\sum_{j=1}^m \cos\Big(r \pi \frac{2j}{2m+1}\Big)
		&= \sum_{j=1}^m \text{Re}\Big[ \exp \Big(r i \pi \frac{2j}{2m+1}\Big)\Big]
		= \text{Re}\Big(\sum_{j=1}^m z^j \Big)
		= \text{Re}\Big( \frac{z (z^m - 1)}{z-1}\Big).
	\end{align*}
	Furthermore, since $(z-1)(\bar{z}-1) = 2 - 2\text{Re}(z) = 2 - 2\cos(r \pi \frac{2}{2m+1})$, and
	\begin{align*}
		z(z^m-1)&(\bar{z}-1)
		= z^m -1 - z^{m+1} + z\\
		&= \exp\left(r i \pi \frac{2m+1}{2m+1}\right)\exp\left(-r i \pi \frac{1}{2m+1}\right) - 1\\
		&\quad - \exp\left(r i \pi \frac{2m+1}{2m+1}\right)\exp\left(r i \pi \frac{1}{2m+1}\right) + \exp\left(r i \pi \frac{2}{2m+1}\right)\\
		&= (-1)^r\left(\exp\left(-r i \pi \frac{1}{2m+1}\right) - \exp\left(r i \pi \frac{1}{2m+1}\right)\right) -1 + \exp\left(r i \pi \frac{2}{2m+1}\right),
	\end{align*}
	so $\text{Re}(z(z^m-1)(\bar{z}-1)) = \cos(r\pi\frac{2}{2m+1})-1$, we have
	\begin{align*}
		\text{Re}\left( \frac{z (z^m - 1)}{z-1}\right)
		= \frac{\cos(r \pi \frac{2}{2m+1}) - 1}{2 - 2\cos(r \pi \frac{2}{2m+1})}
		= -\frac{1}{2}.
	\end{align*}
	finishing the proof of the lemma.
\end{proof}

\section{Complements for heat equation with absorption}
\label{SectionProofsForParabolic}

\begin{lemma}\label{lemma: bounded eigenfunctions parabolic}
Consider the operator $K: L_2(\O\times [0,1])\to L_2(\O\times [0,1])$ defined in  \eqref{eq: definition K gen} 
for $\L = \frac{\partial}{\partial t} - \frac{1}{2}\Delta$ and $\O= (0,1)^d$. 
For $i\in \NN^d$, define $\mu_i=\pi^2 \sum_{j=1}^d i_j^2$ and for $k\in\NN$ let $\nu_{i,k}$ be the (unique) solution
 to the equation $\nu_{i,k}/\tan\nu_{i,k} =-\mu_i/2$ in the interval $\bigl((k-1/2)\pi, k\pi\bigr)$.
Then the eigenvalues  of the operator $K^\tp K$  are given by $\l_{i,k}=1/(\nu_{i,k}^2+\mu_i^2/4)$, for $(i,k) \in \NN^d\times \NN$, 
and satisfy
\begin{align*}
\frac{\pi^{-2}}{k^2+ \pi^2(\sum_{j=1}^d i_j^2)^2/4} \leq \l_{i,k} \leq \frac{\pi^{-2}}{(k-1/2)^2 +  \pi^2(\sum_{j=1}^d i_j^2)^2/4},
\quad (i,k) \in \NN^d\times \NN.
\end{align*} 
The corresponding normalised eigenfunctions are given by
\begin{align}\label{eq: eigenfunctions KTK parabolic}
	h_{i,k}(x,t) := \frac{-\frac{1}{\tan\nu_{i,k}}\sin(\nu_{i,k} t) + \cos(\nu_{i,k} t)}
{\sqrt{\frac{- \frac{1}{\tan\nu_{i,k}} + \frac{\nu_{i,k}}{(\sin\nu_{i,k})^2}}{2\nu_{i,k}}}}2^{d/2}\prod_{j=1}^d \sin(i_j \pi x_j).
	\end{align}
The set of eigenfunctions $\{h_{i,k}: (i,k) \in \NN^d\times \NN\}$ is uniformly bounded in $L_\infty$.
\end{lemma}

\begin{proof}
Equation \eqref{eq: definition K gen}  gives $u=\L Ku$ on $\tilde\O=\O\times(0,1)$, whence one or two partial integrations yield
$$\langle u, K^\tp v\rangle_{L_2(\tilde\O)}=\langle \partial_t Ku-\tfrac12 \Delta Ku, K^\tp v\rangle_{L_2(\tilde\O)}
=\langle Ku, -\partial_tK^\tp v-\tfrac12\Delta K^\tp v\rangle_{L_2(\tilde\O)},$$
for every $v$ for which $K^\tp v(x,1)=K^\tp v(0,t)=0$ for every $x\in\O$ and $t\in (0,1)$.
Here we also use the boundary conditions on $Ku$ specified in \eqref{eq: definition K gen} to see that the boundary values to the partial integrations vanish.
It follows that the adjoint operator $K^\tp: L_2\bigl(\O\times [0,1]\bigr) \to L_2\bigl(\O\times [0,1]\bigr)$ of $K$ is the solution operator $v\mapsto K^\tp v$  of the homogeneous differential equation
	\begin{equation}\label{eq: definition adjoint K parabolic}
		\begin{caseAlign}
			{-\frac{\partial}{\partial t} K^\tp v - \frac{1}{2}\Delta K^\tp v} &= v,\qquad & (x,t) \in \O\times(0,1),\\
			K^\tp v(x,t) &= 0, \qquad  & x \in \partial \O,\, t\in(0,1],\\
			K^\tp v(x,1) &= 0, \qquad & x \in \O.
		\end{caseAlign}
	\end{equation}
Evaluating this at $v=Ku$, we find $Ku=-\partial_t K^\tp K u-\tfrac12 \Delta K^\tp K u$. By \eqref{eq: definition K gen}  this implies
that $\Delta (-\partial_t K^\tp K u-\tfrac12 K^\tp K u)=u$ on $\O\times(0,1)$, and moreover the function on the right inherits the boundary conditions of $Ku$. We infer that 
the operator $K^\tp  K$ is given as the solution operator $u\mapsto K^\tp K u$ of the homogeneous differential equation
	\begin{equation}
		\begin{caseAlign}
			\big(-\frac{\partial^2}{\partial t^2}+ \frac{1}{4}\Delta^2\big) K^\tp K u &= u, \qquad & (x,t) \in \O\times(0,1),\\
			K^\tp K u(x,t) &= 0, \qquad & x \in \partial \O,\, t\in(0,1],\\
			K^\tp K u(x,1) &= 0, \qquad & x \in \O,\\
			(-\frac{\partial}{\partial t} - \frac{1}{2}\Delta) K^\tp K u(x,t) &= 0, \qquad & x \in \partial \O,\, t\in(0,1],\\
			(-\frac{\partial}{\partial t} - \frac{1}{2}\Delta) K^\tp K u(x,0) &= 0, \qquad & x \in \O.
		\end{caseAlign}
	\end{equation}
Thus  an eigenfunction $h$ of $K^\tp  K$ with eigenvalue $\l$ satisfies $h = (-\frac{\partial^2}{\partial t^2} + \frac{1}{4} \Delta^2)\l h$ on $\O\times(0,1)$,
and satisfies the given boundary conditions.

For an eigenfunction of the form $h(x,t) = g(t) f(x)$ and $\l>0$, this becomes
	\begin{equation}
		gf = -\l \frac{\partial^2}{\partial t^2} gf + \l \frac{1}{4} g \Delta^2 f.
	\end{equation}
	Dividing both sides by $gf$ results in
	\begin{equation}
		1 = -\l \frac{\frac{\partial ^2}{\partial t^2} g}{g} + \frac{1}{4}\l \frac{\Delta^2 f}{f}.
	\end{equation}
	Since the left-hand side is constant, and the two terms on the right do not share any variable, they must be constants. 
Therefore, we are lead to the following systems of equations, for some constant $c\in\RR$,
	\begin{equation}\label{eq:heat:linear:1}
		\begin{caseAlign}
			\frac{\partial^2}{\partial t^2}g(t) &= -\frac{c}{\l}g(t),  \qquad &t \in (0,1),\\
			g(t) &= 0, \qquad &t = 1,
		\end{caseAlign}
	\end{equation}
	and
	\begin{equation}\label{eq:heat:linear:2}
		\begin{caseAlign}
			\Delta^2 f(x) &= 4\frac{1-c}{\l} f(x), \qquad &x \in \O,\\
			f(x) &= 0, \qquad &x \in \partial\O,\\
		\end{caseAlign}
	\end{equation}
together with the mixed boundary conditions
\begin{equation*}
\begin{caseAlign}
			(-\frac{\partial}{\partial t} - \frac{1}{2}\Delta) g(t)f(x) &=0, \qquad &x \in \O, t= 0,\\
			(-\frac{\partial}{\partial t} - \frac{1}{2}\Delta) g(t)f(x) &= 0, \qquad &x \in \partial \O,\, t\in (0,1].
\end{caseAlign}
\end{equation*}
We shall show that these equations yield a complete set of eigenfunctions.

The system of equations \eqref{eq:heat:linear:2} describes the eigenfunctions of the square Laplacian with zero boundary condition. These 
eigenfunctions are the same as those of the Laplacian with the eigenvalues squared. The system is solved
by the functions $f(x) = 2^d \prod_{j=1}^d \sin(i_j \pi x_j)$, for $i_j \in \NN$, with corresponding eigenvalues 
$4(1-c)/\l=\m^2$, for $\mu=\pi^2\sum_{j=1}^d i_j^2$ the eigenvalues of minus the Laplacian $-\Delta$.
Since $\Delta f=-\mu f$, the first mixed boundary condition 
becomes $(-g'(0)+ g(0)\mu/2)f(x)=0$, for $x\in\O$.
This implies the boundary condition $-g'(0)+ g(0)\mu /2=0$ of Robin form for the problem \eqref{eq:heat:linear:1}.

It follows from the preceding that $c<1$. Under the (Robin) boundary conditions of \eqref{eq:heat:linear:1}, partial integration gives
$\langle g'',g\rangle_{L_2(0,1)}=-g(0)^2\m/2-\int_0^1 g'(t)^2\,ds<0$, for $g\not=0$. It follows that the operator $g\mapsto g''$ is
negative definite and hence possesses only negative eigenvalues. We may thus restrict \eqref{eq:heat:linear:1} 
to $c>0$. Under this restriction the general solution of this second order ordinary differential equation is
$g(t) = c_1 \sin(\nu t) + c_2 \cos(\nu t)$, for $\nu = \sqrt{c/\l}$. The boundary conditions at $t=1$ and $t=0$ 
give $ c_1 \sin\nu + c_2 \cos\nu=0$ and $-c_1 \nu+ c_2 \mu/2=0$. A nonzero solution $(c_1,c_2)$ requires that $(\sin \nu)\mu/2+\nu\cos\nu=0$.
For $\nu>0$ this is possible only if $\nu\not=k\pi$, for $k\in\NN$, and then $\nu\cos\nu/\sin \nu=-\mu/2$.
The function $\nu\mapsto \nu\cos\nu/\sin\nu$ is strictly decreasing on every interval $\bigl((k-1)\pi,k\pi\bigr)$, for $k\in \NN$, 
with value 0 at $(k-1/2)\pi$ and limit $-\infty$ as $\nu\uparrow k\pi$.
It follows that for every $k \in \NN$ and $\mu>0$, there exists a unique $\nu_k \in \bigl((k-1/2)\pi,k\pi\bigr)$ such that 
$\nu_k\cos\nu_k/\sin\nu_k=- \mu/2$.

The equations $4(1-c)/\l=\mu^2$ and $\nu=\sqrt{c/\l}$ imply that $1=(\mu^2/4+\nu^2)\l$. Thus for
every $i\in\NN^d$ and $k\in \NN$, we have obtained an eigenvalue $\l_{i,k}=1/(\mu_i^2/4+\nu_{i,k}^2)$, 
for $\mu_i=\pi^2\sum_{j=1}^di_j^2$ and $\nu_{i,k}/\tan \nu_{i,k}=-\mu_i/2$ with $\nu_{i,k}$ contained in the interval
$\bigl((k-1/2)\pi,k\pi\bigr)$. This clearly satisfies the inequalities in the lemma.
	
The eigenfunctions $f$ as given are already normalised in $L_2(\O)$.
To normalise $(x,t)\mapsto g(t)f(x)$ in $L_2\bigl(\O\times [0,1]\bigr)$, we require that $\int_0^1 g(t)^2 \,dt = 1$, which gives
	\begin{align*}
		g(t) = \frac{-\frac{1}{\tan\nu}\sin(\nu t) + \cos(\nu t)}{\sqrt{\frac{- \frac{1}{\tan\nu} + \frac{\nu}{(\sin\nu)^2}}{2\nu}}}.
	\end{align*}
	For $\nu$ satisfying $\nu/\tan\nu =- \mu/2$, we have $\tan\nu = -{2\nu}/{\mu} < 0$.
Thus the norming constant is lower bounded by
	\begin{align*}
		\sqrt{\frac{-\frac{1}{\tan\nu} + \frac{\nu}{(\sin\nu)^2}}{2\nu}}\geq \sqrt{0 + \frac{1}{2(\sin\nu)^2}} \geq \sqrt{\frac{1}{2}}.
	\end{align*}
	Thus the second term of $g$ is bounded.
	The absolute value of the constant in front of the first term can be bounded using the definitions of $\nu$ and $\mu$ as
	\begin{align*}
		\biggl|\frac{\frac{1}{\tan\nu}}{\sqrt{\frac{-\frac{1}{\tan\nu} + \frac{\nu}{(\sin\nu)^2}}{2\nu}}}\biggr|
		= \sqrt{\frac{2\nu}{-\tan\nu + \frac{\nu}{(\cos\nu)^2}}}
		= \left(\frac{1}{\mu} + \frac{1}{2(\cos\nu)^2}\right)^{-1/2}
		\leq \sqrt{2}.
	\end{align*}
	We see that $g$ is uniformly bounded, and thus also the eigenfunctions are uniformly bounded.

It remains to show that all eigenfunctions of  $K^\tp K$ are of the form \eqref{eq: eigenfunctions KTK parabolic}.
The functions $(f_{i}(x))_{i \in \NN^d}$ form an orthonormal basis of $L_2(\O)$, and the $(g_k)_{k \in \NN}$ 
form an orthonormal basis of $L_2([0,1])$, since they are the eigenfunctions of a Sturm-Liouville problem.
Then the products of these functions form an orthonormal basis of $L_2\bigl(\O\times [0,1]\bigr)$ and give
the complete set of eigenfunctions. This concludes the proof of the lemma.
\end{proof}

\begin{lemma}\label{lemma: sorted eigenvalues2}
Let $\{\k_{i,k}: i\in \NN^d,k\in\NN\}$ be numbers with $C_1/(\|i\|^p+k^q)\le \k_{i,k}\le C_2/(\|i\|^p+k^q)$, for  constants $C_1, C_2>0$ and $p,q>0$ and
$\|i\|^2=\sum_{j=1}^di_j^2$, for every $i=(i_1,\ldots, i_d)\in\NN^d$, and let $k_1\geq k_2\ge\cdots$  
denote the same values $\k_{i,k}$ sorted in decreasing order.
Then there exist positive constants $\tilde{C}_1,\tilde{C}_2$ that depend on $C_1,C_2,d$ only such that 
		\begin{align*}
			\tilde{C}_1 \ell^{-p/(d+p/q)} \leq k_\ell\leq \tilde{C_2} \ell^{-p/(d+p/q)},\qquad \ell\in \NN.
		\end{align*}
Furthermore, if $v_1,v_2,\ldots$ are the values of the multi-indexed array  $\{\nu_{i,k}: i \in \NN^d,k\in\NN\}$
in corresponding order, then $\|v\|_{\Hil^\b}^2 =\sum_{\ell=1}^\infty\ell^{2\b/(d+1)}v_\ell^2$ satisfies
\begin{align*}
	\|v\|_{\Hil^\b}^2 \asymp \sum_{i\in \NN^d,k\in\NN} \bigl(\|i\|+k^{q/p}\bigr)^{2\b(d+p/q)/(d+1)}\nu_{i,k}^2 .
	\end{align*}
\end{lemma}

\begin{proof}
Let $z: \NN^d\times\NN \to \NN$ be a bijection that corresponds to an ordering of the numbers $\k_{i,k}$. We show in the next paragraph
that $z(i,k)\asymp (\|i\|^p+k^q)^{d/p+1/q}$, where the multiplicative constants depend on $C_1,C_2,d$ only. 
If this is true, then $k_{z(i,k)}=\k_{i,k}\asymp (\|i\|^p+k^q)^{-1}$, by the assumption on the $\k_{i,k}$, which is $z(i,k)^{-p/(d+p/q)}$, and the first
assertion of the lemma is proved.
Furthermore $\sum_{\ell\in \NN}\ell^{2\b/(d+1)}v_\ell^2=\sum_{i\in\NN^d,k\in\NN}z(i,k)^{2\b/(d+1)}v_{z(i,k)}^2
\asymp \sum_{i\in\NN^d,\in\NN}\bigl(\|i\|^p+k^q\bigr)^{2\b(d/p+1/q)/(d+1)}\nu_{i,k}^2$,
proving the second assertion of the lemma.

It remains to prove that $z(i,k)\asymp  (\|i\|^p+k^q)^{d/p+1/q}$, which we achieve by a separate lower and upper bound.
If $(j,l)\in\NN^d\times\NN$ is such that $\|j\|^p+l^q\le d_1(\|i\|^p+k^q)$ for a sufficiently small constant $d_1>0$ that depends on $C_1,C_2$ only
($d_1<C_1/C_2$ works), then
the assumption on the numbers $\k_{i,k}$ shows that $\k_{j,l}>\k_{i,k}$ and hence $z(j,l)\le z(i,k)$. It follows that 
$z(i,k)\ge \#\bigl((j,l)\in\NN^{d+1}: \|j\|^p+l^q\le d_1(\|i\|^p+k^q)\bigr)$,
which is bounded below by $\#\bigl(j\in\NN^d: \|j\|^p\le d_1(\|i\|^p+k^q)/2\bigr)
\#\bigl(l\in\NN: l^q\le d_1(\|i\|^p+k^q)/2\bigr)$, which is of the order $(\|i\|^p+k^q)^{d/p+1/q}$
up to multiplicative constants depending on $d_1$ and $d$ only. Conversely, because every $(j,l)\in\NN^{d+1}$ with $z(j,l)\le z(i,k)$ has
$\k_{j,l}\ge \k_{i,k}$, we have $z(i,k)= \#\bigl((j,l)\in\NN^{d+1}: z(j,l)\le z(i,k)\bigr)\le  \#\bigl((j,l)\in\NN^{d+1}: \k_{j,l}\ge \k_{i,k}\bigr)$, 
which is bounded above by $\#\bigl((j,l)\in\NN^{d+1}: \|j\|^p+l^q\le d_2(\|i\|^p+k^q)\bigr)$ for  $d_2=C_2/C_1$ 
by the assumption on the numbers $\k_{i,k}$, which is bounded above by 
$\#\bigl(j\in\NN^d: \|j\|^p\le d_2(\|i\|^p+k^q)\bigr)\#\bigl(l\in\NN: l^q\le d_2(\|i\|^p+k^q)\bigr)$, 
which is bounded above by a multiple of $(\|i\|^p+k^q)^{d/p+1/q}$.
\end{proof}

\section{Complements for One-dimensional Darcy equation}
\label{AppendixOneDimensionalDarcy}
In Theorem~\ref{cor:darcy} it is assumed that $u_{f_0}'$ is bounded away from 0. This can be relaxed to the condition
$\inf_{0<x<1} \bigl(|u_{f_0}'(x)|+u_{f_0}''(x)\bigr)>0$, as imposed in \cite{Richter,Nickl23} and Section~\ref{SectionDarcy}. 
Under the latter condition any zero $x_0$ of $u_{f_0}'$ has $u_{f_0}''(x_0)>0$ and hence is a point of minimum of $u_{f_0}$. This implies
that there can be at most one zero $x_0\in(0,1)$, in which case the differential equation \eqref{EqDarcyOneDimensional} can be integrated to find
 $ (f_0u_{f_0}')(x)=\int _{x_0}^xh(s)\,ds$, for every $x\in[0,1]$, and hence 
$$f_0(x)=\begin{cases}\frac{H(x)-H(x_0)}{u_{f_0}'(x)}, &x\not=x_0,\\
\frac{h(x_0)}{u_{f_0}''(x_0)},& x=x_0.
\end{cases}$$
The value at $x_0$ follows by continuity or directly from \eqref{EqDarcyOneDimensional}.
An interesting feature is that this inversion formula requires no boundary value of $f_0$, as in Section~\ref{SectionOneDimensionalDarcy}.

To make this inversion strategy work with noisy data, we would have to replace $u_{f_0}'$ by a draw from the posterior. Because this appears
unstable, we instead define an inverse operator using the second derivative. Let $V$ be the set of continuous functions $v: [0,1]\to\RR$ such
that the map $x\mapsto \int_0^x v(s)\,ds+ g(1)-g(0)-\int_0^1\int_0^tv(s)\,ds\,dt$ possesses a single zero at a point $x_v\in(0,1)$, with $v(x_v)>0$,
and define a map $e: V\to L_2(0,1)$ by
\begin{equation}
\label{EqInversionMapOneDimensionalDarcyWithZero}
e(v)=\begin{cases} \frac{H(x)-H(x_v)}{\int_{x_v}^x v(s)\,ds},& \text{if } x\not= x_v,\\
\frac{h(x_v)}{v(x_v)},& \text{if } x=x_v.\end{cases}
\end{equation}
(The function $v$ in this definition stands for an estimate of $u''$ and the constant $b=g(1)-g(0)-\int_0^1\int_0^tv(s)\,ds\,dt$
in its primitive function $x\mapsto \int_0^x v(s)\,ds+ b$ is chosen so that the double primitive function 
$x\mapsto u(x):=\int_0^x\int_0^t v(s)\,ds+g(0)+ bx$ 
satisfies the boundary conditions of \eqref{EqDarcyOneDimensional} at $x\in\{0,1\}$.)

\begin{lemma}
Suppose that the solution $u_{f_0}\in C^2[0,1]$ to \eqref{EqDarcyOneDimensional} for given $f_0$ 
satisfies  $\inf_{0<x<1} \bigl(|u_{f_0}'(x)|+u_{f_0}''(x)\bigr)>0$ with $u_{f_0}'(x_0)=0$ for $x_0\in(0,1)$.
Then there exists $\eta>0$ so that any continuously differentiable $v$ with $\|v-u_{f_0}''\|_\infty<\eta$ 
is contained in $V$, and the function $f=e(v)$ satisfies $u_f''=v$ with
$\|f-f_0\|_{L^2(0,1)}\le c_0 \|v-u_{f_0}''\|_{L^2(0,1)}$, for a constant $c_0$ that depends on $f_0$ only.
\end{lemma}

\begin{proof}
Define $u'(x)=b+\int_0^xv(s)\,ds$ and $u(x)=a+bx+\int_0^x\int_0^t v(s)\,ds\,dt$, where the constants $a$ and $b$ are determined
so that $u(0)=g(0)$ and $u(1)=g(1)$, i.e.\ $a=g(0)$ and $b=g(1)-g(0)-\int_0^1\int_0^tv(s)\,ds\,dt$. Because $u_{f_0}$ satisfies
the same boundary conditions, it follows that $u_{f_0}'(x)=b_0+\int_0^xv_0(s)\,ds$ and $u_{f_0}(x)=a_0+b_0x+\int_0^x\int_0^t v_0(s)\,ds\,dt$,
for $v_0=u_{f_0}''$, where $a_0=g(0)$ and $b_0=g(1)-g(0)-\int_0^1\int_0^tv_0(s)\,ds\,dt$. 
We conclude that $\|u'-u_{f_0}'\|_\infty\le 2 \|v-v_0\|_{L^2(0,1)}$.

Because any zero of $u_{f_0}'$ is a point of minimum of $u_{f_0}$, the point $x_0$ is a unique zero of $u_{f_0}'$,
with $u_{f_0}''(x_0)>0$. It follows that there exist
positive constants $\d$ and $c$ such that 
$$\sup_{0\le x\le x_0-\d}u_{f_0}'(x)\le -c,\qquad \inf_{x_0+\d\le x\le 1}u_{f_0}'(x)\ge c,
\qquad \inf_{x_0-\d\le x\le x_0+\d}u_{f_0}''(x)\ge c.$$
If $\|v-v_0\|_\infty<c/2$, then $\|u'-u_{f_0}'\|_\infty<c$ and hence $u'$ is negative on $[0,x_0-\d]$ 
and  positive on $[x_0+\d,1]$, and
$u''=v$ is positive on $[x_0-\d,x_0+\d]$. This implies that $u'$ has a unique zero $x_v$ in $[0,1]$, contained in $[x_0-\d,x_0+\d]$ and with $v(x_v)=u''(x_v)>0$.

Define $f(x)=\bigl(H(x)-H(x_v)\bigr)/u'(x)$, for $x\not=x_v$ and $f(x_v)=h(x_v)/u''(x_v)$. This implies that $u'f=H-H(x_v)$ and hence
$(u'f)'=h$.
Because $u$ satisfies the boundary conditions at $x\in\{0,1\}$ by construction, it follows that $u=u_f$ solves \eqref{EqDarcyOneDimensional}.
Since $u'(x)=\int_{x_v}^xv(s)\,ds$, it follows that $f=e(v)$. 

It is immediate from its definition that the function $f$ is continuously differentiable at every $x\not= x_v$.
Since $v$ is continuously differentiable, the function $u$ is three times continuously differentiable. By Taylor expansion it can be seen
that, as $x\rightarrow x_v$,
\begin{align*}
\frac{H(x)-H(x_v)}{u'(x)}&=\frac{h(x_v)}{u''(x_v)}+(x-x_v)\Bigl(\frac{h'(x_v)}{2u''(x_v)}-\frac{u'''(x_v)h(x_v)}{2u''(x_v)^2}\Bigr)+o(x-x_v),\\
f'(x)&=\frac{u'(x)h(x)-\bigl(H(x)-H(x_v)\bigr)u''(x)}{u'(x)^2}\ra \frac{h'(x_v)}{2u''(x_v)}-\frac{u'''(x_v)h(x_v)}{2u''(x_v)^2}.
\end{align*}
It follows that $f$ is continuously differentiable throughout $[0,1]$.

The inequality $\|f-f_0\|_{L^2(0,1)}\le c_0 \|v-u_{f_0}''\|_{L^2(0,1)}$ is the one-dimensional case of Lemma~\ref{LemmaDarcyStable}, where
the influx boundary is empty, as $u'(0)<0$ and $u'(1)>0$, and $\|f'\|_\infty$ and $\|f\|_\infty$ are finite.
\end{proof}

The solution map \eqref{EqInversionMapOneDimensionalDarcyWithZero} allows to transform a posterior distribution for $v=u_f''$
into a posterior distribution of $f$. By the preceding lemma this map preserves the $L_2$-contraction rate of $v$, provided
it can be ascertained that the posterior distribution of $v$ concentrates on sufficiently small uniform balls around $v_0=u_{f_0}''$.
The latter can be ascertained by supremum norm consistency of the posterior of $u_f''$ to $u_{f_0}''$, similarly to Theorem~\ref{cor:darcy}. 
The resulting rate is $n^{-(\a \wedge \b)/(1+2\a + 4)}$ if $u_{f_0}''$ is smooth of order $\b$. Uniform consistency can be guaranteed 
if $\b>1/2$. The contraction rate is slower compared to  Theorem~\ref{cor:darcy}, which may be the result 
of the milder assumptions on $u_{f_0}'$.

\section{Complements for Darcy equation}
\label{SectionComplementsDarcy}

\begin{proof}[Proof of of Lemma~\ref{LemmaDarcyStable}]
If $u_f$ and $u_{f_0}$ satisfy the partial differential equation \eqref{EqDarcy} with the same function $h$, then
 $\nabla\cdot(f\nabla u_f)= \nabla\cdot(f_0\nabla u_{f_0})$. Subtracting $\nabla\cdot (f_0\nabla u_f)$ on both sides yields the equation
$\nabla\cdot\bigl((f-f_0)\nabla u_f\bigr)=- \nabla\cdot\bigl(f_0\nabla(u_{f}- u_{f_0})\bigr)$. We multiply
both sides of this equation by $(f-f_0)e^{-u_f}$ and next integrate over $\O$. 

On the right side we obtain, with $dx$ Lebesgue measure,
$$-\int _\O\bigl(\nabla f_0\cdot\nabla(u_f-u_{f_0})+f_0\Delta (u_f-u_{f_0})\bigr)(f-f_0)e^{- u_f}\,dx.$$
We bound $\nabla f_0$, $f_0$ and $e^{-u_f}$ by their supremum norms, and next use the Cauchy-Schwarz inequality
to see that this is bounded in absolute value by the first term on the right side of the lemma times $e^{\|u_f\|_\infty}/2$.

On the left side we obtain, with $k=f-f_0$,
\begin{align*}
&\int_\O \bigl(\nabla k\cdot\nabla u_f +k\Delta u_f\bigr) ke^{- u_f}\,dx
=\int_\O \nabla (\thalf k^2)\cdot \nabla (-e^{- u_f})\,dx+\int_\O k^2(\Delta u_f) e^{- u_f}\,dx\\
&\qquad=\int_\O \thalf k^2\,\Delta (e^{- u_f})\,dx+
\int_{\partial\O} \thalf k^2\Bigl(\frac{\partial (-e^{- u_f})}{\partial \vec n}\Bigr)\,dS+\int_\O k^2(\Delta u_f) e^{- u_f}\,dx\\
&\qquad=\int_\O \thalf k^2\bigl(\|\nabla u_f\|^2+\Delta u_f\bigr)e^{- u_f}\,dx+
\int_{\partial\O} \thalf k^2e^{-u_f}\bigl(\nabla u_f\cdot \vec n\bigr)\,dS,
\end{align*}
where in the second equality we used Green's formula on the first term  (partial integration, e.g.\ \cite{Evans10}, Theorem~3 in Appendix~C2,
with $\partial/\partial\vec n$ 
the directional derivative in the direction of the outer normal vector $\vec n$ on $\partial \O$) 
and in the last equality the identity 
$\Delta (e^{- u})= e^{-u}\bigl(\|\nabla u\|^2-\Delta u\bigr)$. The first term on the far right is bounded below
by $\thalf \|k\|_{L_2(\O)}^2 C(u_f) e^{-\|u_f\|_\infty}$. The integrand in the second term is bounded below by zero 
on the complement of $(\partial \O)_{1,f}$, by the definition of the influx boundary, 
and bounded below by $-\thalf h^2 e^{\|u_f\|_\infty}\|\nabla u_f\|_\infty$ everywhere, so that 
the surface integral is bounded below by minus the second term on the right of the bound in the lemma times $e^{\|u_f\|_\infty}$.
Combining this with the preceding paragraph and rearranging the terms gives the bound of the lemma.

Because by construction the function $u_f-u_{f_0}$ vanishes on the boundary of $\O$, its Sobolev
norm on the right is equivalent to the $L_2(\O)$-norm of its Laplacian $\Delta (u_f-u_{f_0})$ (e.g.\ \cite{Lions1972}, Theorem~II.5.4 or
\cite{Nickl23}, (A.29)). 

If $f=f_0$ on $\partial \O$, then the second term on the right in the first bound of the lemma vanishes, and the bound reduces
to \eqref{EqRichardDarcy}, with the multiplicative constant $2\|f_0\|_{C^1(\O)}$. The assumption that $f,f_0\in \Sob^\b(\O)\subset C^1(\O)$ 
is sufficient for the existence of $u_f$ and $u_{f_0}$ in $C^2(\O)$ (see \cite{Nickl23}, Proposition~6.1.5).
\end{proof}

\begin{proof}[Proof of Lemma~\ref{LemmaDiscretisationDarcy}]
The proof consists of five steps, (i)-(v).

(i). We have $\|z_{i,j}\|\le \sqrt\d$ and
$\bigl\|{z_{i,j}}/{\|z_{i,j}\|}-{\nabla u_{i,j}}/{\|\nabla u_{i,j}\|}\bigr\|\le 2\sqrt 2\sqrt\d$.

The first assertion is immediate from the definition of $z_{i,j}$. To prove the second, let $\tilde z_{i,j}$ be as $z_{i,j}$, but without the truncation to the grid. 
Thus $\tilde z_{i,j}=\sqrt\d\,\nabla u/\|\nabla u\|$ and hence $\tilde z_{i,j}/\|\tilde z_{i,j}\|=\nabla u_{i,j}/\|\nabla u_{i,j}\|$.
Now $\bigl\|z_{i,j}/\|z_{i,j}\|-\tilde z_{i,j}/\|\tilde z_{i,j}\|\bigr\|\le 2\|z_{i,j}- \tilde z_{i,j}\|/\|\tilde z_{i,j}\|$, for any vectors $z,\tilde z$,
which is bounded above by $2\sqrt 2\d/\sqrt\d$ in the present case. The possible redefinition of $z_{i,j}$ if $x_{k,l}$ is a boundary point,
does not change the direction of $z_{i,j}$ and hence does not affect $z_{i,j}/\|z_{i,j}\|$.

(ii). If $\|\nabla u_{i,j}\|\ge \sqrt\d$, then $u_{i,j}-u_{k,l}\ge \|z_{i,j}\|\bigl(\|\nabla u_{i,j}\|(1-4\d)-\|\nabla u\|_{\text{Lip}}\sqrt \d\bigr)$.
In particular, $u_{i,j}>u_{k,l}$, for all sufficiently small $\d>0$.
 
The left side $u_{i,j}-u_{k,l}=u(x_{i,j})-u(x_{k,l})$ is equal to
\begin{align*}
&\int_0^1 \nabla u(x_{i,j}-s z_{i,j})\cdot z_{i,j}\,ds
\ge \nabla u(x_{i,j})\cdot z_{i,j}-\|\nabla u\|_{\text{Lip}}\|z_{i,j}\|^2\\
&\qquad\qquad=\|\nabla u_{i,j}\|\,\| z_{i,j}\|\Bigl(1-\thalf \Bigl\|\frac{z_{i,j}}{\|z_{i,j}\|}-\frac{\nabla u_{i,j}}{\|\nabla u_{i,j}\|}\Bigr\|^2\Bigr)-\|\nabla u\|_{\text{Lip}}\|z_{i,j}\|^2,
\end{align*}
since $x\cdot y=1-\|x-y\|^2/2$, for $x,y$ with $\|x\|=\|y\|=1$. Next we apply (i).

(iii). For $2D(u)=\min_{x\in\O} (\Delta u+\|\nabla u\|)(x)$ and $\|f\|_{\partial_d\O_{1,u}}=\max_{x\in\partial_d\O_{1,u}} |f(x)|$, we have,
for $D(u)(1-4\d)-\|\nabla u\|_{\text{Lip}}\sqrt\d>0$,
$$\max_{i,j}|\a_{i,j}|\le \Bigl(\frac{\|g\|_\infty}{D(u)}\vee \|f\|_{\partial_d\O_{1,u}}\vee \sqrt\d\Bigr)\,\exp\Bigl(\frac{{3(\max u-\min u)}}
{D(u)(1-4\d)-\|\nabla u\|_{\text{Lip}}\sqrt\d}\Bigr).$$

For the proof consider three cases.

If $\|\nabla u_{i,j}\|<\sqrt\d$, then $\Delta u_{i,j}\ge 2D(u)-\sqrt\d>0$ and $|\a_{i,j}|=|g_{i,j}|/\Delta u_{i,j}$ by the first case of \eqref{eq:sol:discrete}, and the bound 
follows.

If $x_{i,j}\in \partial_d\O_{1,u}$, then $\a_{i,j}$ is set equal to $f(x_{i,j})$ and the bound is immediate.

If $\|\nabla u_{i,j}\|\ge\sqrt\d$, then $\a_{i,j}$ is defined by the second case of \eqref{eq:sol:discrete} as a recursion on $\a_{k,l}$. 
If $\Delta u_{i,j}>D(u)$, then the inequality $(a+b)/(c+d)\le (a/c)\vee (b/d)$, valid for any nonnegative numbers $a,b$ and positive numbers $c,d$, gives that
$$|\a_{i,j}|\le \frac{\|g\|_\infty}{D(u)}\vee |\a_{k,l}|.$$
If $\Delta u_{i,j}\le D(u)$, then $\|\nabla u_{i,j}\|\ge D(u)$. Since the definition of $D(u)$ gives that always 
$\Delta u_{i,j}+\|\nabla u_{i,j}\|/\|z_{i,j}\|\ge D(u)+\|\nabla u_{i,j}\|(1/\|z_{i,j}\| -1)$,
definition \eqref{eq:sol:discrete} gives
\begin{align*}
|\a_{i,j}|&\le \frac{\|g\|_\infty+|\a_{k,l}|\,\|\nabla u_{i,j}\|/\|z_{i,j}\|}{\|\nabla u_{i,j}\|(1/\|z_{i,j}\|-1)}
\le \frac{\|g\|_\infty}{D(u)\|(1/\|z_{i,j}\|-1)}+\frac{|\a_{k,l}|}{1-\|z_{i,j}\|}\\
&\le \Bigl(\frac{\|g\|_\infty}{D(u)}\vee |\a_{k,l}|\Bigr)\frac{\|z_{i,j}\|+1}{1-\|z_{i,j}\|}.
\end{align*}
Combining the two inequalities in the last two displays, we see that whenever $\|\nabla u_{i,j}\|\ge\sqrt\d$, then $\b_{i,j}:=\|g\|_\infty/D(u)\vee |\a_{i,j}|$
satisfies $\b_{i,j}\le \b_{k,l}\,r_{i,j}$, where $r_{i,j}=1$ if $\Delta u_{i,j}>D(u)$ and $r_{i,j}=\bigl(1+\|z_{i,j}\|\bigr)/\bigl(1-\|z_{i,j}\|\bigr)$
otherwise. Iterating this from $x_{i,j}$ along its chain  of points $x_{k',l'}$ until the first point $x_{i_0,j_0}$, which is in the influx boundary or has $\|\nabla u_{i_0,j_0}\|<\sqrt\d$, we find
$$|\a_{i,j}|\le \Bigl(\frac{\|g\|_\infty}{D(u)}\vee |\a_{i_0,j_0}|\Bigr)\prod_{\Delta u_{k',l'}\le D(u)} \frac{\|z_{k',l'}\|+1}{1-\|z_{k',l'}\|}
\le \Bigl(\frac{\|g\|_\infty}{D(u)}\vee |\a_{i_0,j_0}|\Bigr) e^{3\sum \|z_{k',l'}\|},$$
since $(1+u)/(1-u)\le e^{3u}$, for $u<1/2$,
where the sum in the exponent is over the vectors $z_{k',l'}$ in the chain connected to $x_{i,j}$ with $\Delta u_{k',l'}\le D(u)$.
Since these vectors satisfy $\|\nabla u_{k',l'}\|\ge D(u)$, by (ii) this sum times 
$D(u)(1-4\d)-\|\nabla u\|_{\text{Lip}}\sqrt \d$ is bounded above by $\sum (u_{i',j'}-u_{k',l'})\le \max u-\min u$.

(iv). If $f\in C^1([0,1]^2)$, then $\max_{i,j} \bigl| (\L_f^\d u)_{i,j}-g_{i,j}\bigr|\le c\d^{\eta/2}$, for
$c=\|\nabla f\|_\infty(1+2\sqrt2\|\nabla u\|_\infty+\|\nabla u\|_{\text{Lip}})+\|g-f\Delta u\|_{\text{Lip}}$.

If $\|\nabla u_{i,j}\|<\sqrt\d$, then the first case of \eqref{EqDefLad} gives $(\L_f^\d u)_{i,j}=f_{i,j} \Delta u_{i,j}$, while $g_{i,j}=\nabla f_{i,j}\cdot \nabla u_{i,j}+f_{i,j}\Delta u_{i,j}$, by the
Darcy equation, so that the assertion is true, with multiplicative constant $c=\|\nabla f\|_\infty$ (and $\eta=1$).

If $\|\nabla u_{i,j}\|\ge\sqrt\d$, then the second case  of \eqref{EqDefLad} gives $(\L_f^\d u)_{i,j}=\bigl(\|\nabla u_{i,j}\|/\|z_{i,j}\|\bigr)(f_{i,j}-f_{k,l})+f_{i,j} \Delta u_{i,j}$.
Here
\begin{align*}
f_{i,j}-f_{k,l}&=f(x_{i,j})-f(x_{i,j}-z_{i,j})=\int_0^1\nabla f(x_{i,j}-sz_{i,j})\cdot z_{i,j}\,ds\\
&=\frac{\|z_{i,j}\|}{\|\nabla u_{i,j}\|}\int_0^1\nabla f(x_{i,j}-sz_{i,j})\cdot \nabla u_{i,j}\,ds+ R_{i,j},
\end{align*}
where $R_{i,j}\le \|\nabla f\|_\infty\|z_{i,j}\|2\sqrt 2\sqrt\d$, in view of (i). We can replace $\nabla u_{i,j}$ in the integral
by $\nabla u(x_{i,j}-sz_{i,j})$ at the cost of increasing the remainder $R_{i,j}$ by at most $ \|\nabla f\|_\infty\|z_{i,j}\|^2\|\nabla u\|_{\text{Lip}}/\|\nabla u_{i,j}\|$,
and next use the Darcy equation to rewrite $\nabla f(x_{i,j}-sz_{i,j})\cdot \nabla u(x_{i,j}-sz_{i,j})=g(x_{i,j}-sz_{i,j})-f(x_{i,j}-sz_{i,j})\Delta u(x_{i,j}-sz_{i,j})$,  which
is $g_{i,j}-f_{i,j}\Delta u_{i,j}$ up to an error bounded by $\|z_{i,j}\|^\eta\|g-f\Delta u\|_{C^\eta}$.
We thus obtain
\begin{align*}
&\Bigl|\frac{\|\nabla u_{i,j}\|}{\|z_{i,j}\|}(f_{i,j}-f_{k,l}) -( g_{i,j}-f_{i,j}\Delta u_{i,j})\Bigr|\\
&\qquad\le \|\nabla f\|_\infty\|\nabla u_{i,j}  \|2\sqrt 2\sqrt\d
+\|\nabla f\|_\infty \|\nabla u\|_{\text{Lip}}\|z_{i,j}\| +\|z_{i,j}\|^\eta\|g-f\Delta u\|_{C^\eta}.
\end{align*}
Here $\|z_{i,j}\|\le  \sqrt\d$, by (i).

(v). By construction $(\L_\a^\d u)_{i,j}=g_{i,j}$, while $|(\L_f^\d u)_{i,j}-g_{i,j}|\le c\d^{\eta/2}$,  by (iv). Thus $(\L_{\a-f}^\d u)_{i,j}=\bar g_{i,j}$, for a
function $\bar g$ with $\|\bar g\|_\infty\le c\d^{\eta/2}$. By construction the function $\a-f$ vanishes on the influx boundary. 
Then the analogue of (iii), applied to $\a-f$ instead of $\a$ and with $\|g\|_\infty\le c\d^{\eta/2}$ and $\|f\|_{\partial_d\O_{1,u}}=0$, gives the result of the lemma.
\end{proof}

\section{Smoothness norm contraction rates in the Gaussian white noise linear inverse problems}\label{sec:contract:smoothness:GWN}
In this section we extend the results of \cite{Yan2020} to give rates of contraction relative to 
smoothness norms in the context of the Gaussian white noise linear inverse problem \eqref{def:model:linear}. We obtain general results for an abstract scale of spaces of numerical sequences. By taking these sequences equal to the coefficients of functions relative to a suitable basis, 
the scale may be identified with concrete function spaces.
For instance, for a suitable wavelet basis the scale can (partially) be identified with 
the scale of Sobolev spaces and the general result will give contraction relative to Sobolev norms of positive order. 
By Sobolev embedding this implies consistency (at a rate) for other norms such as the uniform norm.
Alternatively, the sequences can be taken equal to coefficients relative to the eigenbasis
of an operator of interest, giving a contraction rate to a norm intrinsic to the problem.

\subsection{Smoothness scale}
Fix a sequence $1\le \r_1\le\r_2\le\ldots\uparrow\infty$, and  for $s\in \RR$ consider the Hilbert spaces
$$\Hil^s=\bigl\{g=(g_i): \NN\to \RR : \sum_{i\in \NN} \r_i^{2s}g_i^2<\infty\bigr\},$$
with inner product and norm defined by 
$$\langle g, h\rangle_{\Hil^s}=\sum_{i\in \NN} \r_i^{2s} g_i h_i,\qquad 
\|g\|_{\Hil^{s}}=\sqrt{\sum_{i\in \NN} \r_i^{2s}g_i^2}.$$
The spaces are nested: if $s\ge t$, then $\Hil^s\subset \Hil^t$. By the Cauchy-Schwarz inequality
the inner product and norm satisfy, for $s\ge 0$ and every $t\in \RR$,
\begin{equation}
	\label{EqDualNormIdentity}
	\|g\|_{\Hil^{t-s}}=\sup_{g\in \Hil^{t+s}: \|g\|_{\Hil^{t+s}}\le 1}\langle g,h\rangle_{\Hil^t},\qquad g\in \Hil^{t-s}.
\end{equation}
The map defined by $(C_sg)_i=\r_i^{s}g_i$ is an isometry $C_s: \Hil^{t+s}\to \Hil^t$, and 
$\langle g, h\rangle_{\Hil^s}=\langle C_sg, C_sh\rangle _{\Hil^0}$. The subspace $V_j=\{g\in \Hil_0: g_i=0\text{ for } i>j\}$
has the properties, for $s<t$, for $P_j$ the orthogonal projection onto $V_j$,
\begin{align*}
	\|g-P_jg\|_{\Hil^{s}}^2&=\sum_{i>j}\r_i^{2s}g_i^2\le \r_j^{-2(t-s)}\|g\|_{\Hil^t}^2,\\
	\|g\|_{\Hil^t}^2&\le \r_j^{2(t-s)}\|g\|_{\Hil^{s}}^2,\qquad\text{ if } g\in V_j.
\end{align*}
Thus it satisfies the conditions of a smoothness scale as in \cite{Yan2020}.

We consider a prior distribution on the sequences given as the law of the random
sequence $\GG=\bigl( \r_i^{-\a}i^{-1/2}\Ccor_i  Z_i\bigr)$,  for $Z_i\iid N(0,1)$ and a sequence $(\Ccor_i)$
such that $c_\pi\le \Ccor_i\le C_\pi$, for some constants $0<c_\pi<C_\pi$, for all $i$. Then 
$$\E\|\GG\|_{\Hil^{s}}^2=\sum_{i\in\NN} \r_i^{-2(\a-s)}i^{-1}\Ccor_i^2.$$
If $\r_i\gtrsim i^r$, for some $r>0$, then this is finite for every  $s<\a$, and hence the prior concentrates
on the space $\Hil^s$, for $s<\a$: it is regular of nearly order $\a$.

\begin{example}
\label{ExampleWaveletOrdering}
	Given a wavelet basis $\{\psi_{j,k}: j\in \NN_0, k=1,\ldots, 2^{jd}\}$ for functions on a bounded domain in $\RR^d$,
	we can consider the sequences of coefficients $(h_{j,k})$ of expansions of given functions $h$ in the basis, linearly ordered
	by level and location as $h_{0,1}, h_{1,1},\ldots,h_{1,2^d}, h_{2,1},\ldots,h_{2,2^{2d}}, h_{3,1},\ldots$. 
	The ordering defines the map
	$$(j,k)\mapsto i:=1+2^d+\cdots+2^{(j-1)d}+k=\frac{2^{jd}-1}{2^d-1}+k.$$ 
	Set $\r_i=2^j$ and $\bar h_i=h_{j,k}$ if $i\in\NN$ is the image of $(j,k)$ under this map. Then the square norm
	$\sum_j\sum_k 2^{js}h_{j,k}^2$ is equal to $\sum_i \r_i^{2s}\bar h_i^2$ and hence the smoothness
	space with norm the square root of $\sum_j\sum_k 2^{js}h_{j,k}^2$ corresponds to the space $\Hil^s$.
	
	The wavelet basis can be chosen so that the scale of spaces $\Hil^s$, with $s$ in a bounded interval,
	corresponds (partially) to the scale of Sobolev spaces.
	
	By the preceding display $i\asymp 2^{jd}$, as $j\ra\infty$, and hence
	$\r_i\asymp  i^{1/d}$. The Gaussian prior $\sum_j\sum_k 2^{-j(\a+d/2)}Z_{j,k}\psi_{j,k}$, for $Z_{j,k}\iid N(0,1)$,
	in the doubly indexed sequence space,
	corresponds to the Gaussian prior $\bigl( \r_i^{-\a-d/2}\bar Z_i\bigr)$, for $\bar Z_i\iid N(0,1)$ in the sequence space.
	This is the prior $\GG$ for $\r_i^{-d/2}=i^{-1/2}\Ccor_i$, i.e.\ the constants $\Ccor_i$ account for the discrepancy
	between $i^{1/d}$ and $\r_i$.
\end{example}


From now on we  consider a sequence
$\r_i$ such that $c_0 i^{1/d}\le \r_i\le C_0 i^{1/d}$, for constants $C_0\ge c_0>0$ and some $d>0$.
The power $1/d$ is included so that the index $s$ of $\Hil^s$ can be interpreted as ``smoothness''
of a function on a $d$-dimensional domain, but is irrelevant for the results. On the other hand,
the polynomial increase of $\r_i$ will be important for the form of the results. Allowing
$\r_i\asymp i^{1/d}$ rather than setting exactly $\r_i=i^{1/d}$ and including the constants $\Ccor_i$ in
the prior ensure that the construction applies to the ``exact'' weights $2^j$ in the wavelet basis, but
is inessential for the results.

\subsection{Inverse problem}
Let $L$ be an arbitrary separable 
Hilbert space, with inner product $\langle\cdot,\cdot\rangle_L$  and norm $\|\cdot\|_L$.
Let $K: \Hil^0\to L$ be a linear operator such that, for some given $p\ge 0$ and constants $C\ge c>0$,
\begin{equation}
	\label{EqSmoothingProperty}
	c\|f\|_{\Hil^{-p}}\le \|K f\|_L\le C \|f\|_{\Hil^{-p}}, \qquad f\in \Hil^0.
\end{equation}
For a standard white noise process $\xi$ in $L$, we observe
$$Y_n=Kf+\frac 1{\sqrt n}\xi.$$
We put the $\a$-regular prior $\GG$ as indicated on $f$ and denote the corresponding
posterior distribution of $f$ by  $\Pi_n(f\in\cdot\given Y_n)$.

\begin{theorem}\label{thm: contr:diff:norm}
	Assume that $c_0 i^{1/d}\le \r_i\le C_0 i^{1/d}$, for constants $C_0\ge c_0>0$ and $d>0$.
	Let $K: \Hil^0\to L$ be a linear operator satisfying \eqref{EqSmoothingProperty} for every $f\in \Hil^0$.
	If $f_0\in \Hil^\b$ and $0<\d<\a\wedge \b$, then there exists a constant $M$ such that
	$\E_{f_0} \Pi_n\bigl(f: \|f-f_0\|_{\Hil^{\d}}\le M n^{-(\a\wedge \b-\d)/(2\a+2p+d)}\given Y_n\bigr)\ra 1$.
\end{theorem}

\begin{proof}
	For positive constants to  $c_1$, $c_2$ be determined later in the proof  and a sequence of constants $c_{3,n}$ set
\begin{equation}\e_n=c_1\Bigl(\frac1n\Bigr)^{\frac{\a\wedge\b+p}{2\a+2p+d}},\quad
	\h_n=c_2\Bigl(\frac1n\Bigr)^{\frac{\a\wedge\b-\d}{2\a+2p+d}},\quad
	\r_{j_n}=c_{3,n} n^{\frac 1{2\a+2p+d}}.
        \label{EqDefEnHn}
\end{equation}
	By the assumption that $\r_i\asymp i^{1/d}$ , 
	we can find $j_n\in \NN$ so that the corresponding $c_{3,n}$ is bounded from below and above by positive constants.
	
	The smoothing property of $K$ implies that $K$ is injective. Therefore its restriction $K: V_j\to L$ to the
	finite-dimensional subspace $V_j \subset \Hil^0$ is invertible on its range $KV_j$. 
	Let $R_j: L\to V_j$ be the map $R_j=K^{-1}Q_j$ formed
	by first applying the orthogonal projection $Q_j: L\to KV_j$ of the Hilbert space $L$ onto $KV_j$ and
	next applying the inverse $K^{-1}$. (The map $R_jK: \Hil^0\to V_j\subset \Hil^0$ is called the ``Galerkin
	solution'' to $Kf$.) 
	
	Let $p_f^{(n)}$ be a density of $Y_n$ relative to a suitable dominating measure.
	The proof is based on the following claims: we can choose $c_1$ and $c_2$ so that
	\begin{itemize}
		\item[(1)] The events $A_n:=\bigl\{\int p_f^{(n)}/p_{f_0}^{(n)}(Y_n)\,d\Pi(f)\ge e^{-2n\e_n^2}\bigr\}$
		satisfy $P_{f_0}^{(n)}(A_n^c)\ra 0$.
		\item[(2)] $\Pi\bigl(\|Kf-Kf_0\|_L<\e_n\bigr)\ge e^{-n\e_n^2}$.
		\item[(3)] $\Pi\bigl(\|R_{j_n} K f-f\|_{\Hil^{\d}}>\h_n\bigr)\le e^{-4n\e_n^2}$.
		\item[(4)] There exist $M>0$ and tests $\t_n$ such that $P_{f_0}^{(n)} \t_n\ra 0$ and
		$P_f^{(n)}(1-\t_n)\le e^{-4n\e_n^2}$, for every $f\in \F_n$, where
		$$\F_n=\bigl\{f: \|f-f_0\|_{\Hil^{\d}}>M\h_n,\|R_{j_n} K f-f\|_{\Hil^{\d}}\le \h_n\bigr\}.$$
	\end{itemize}
	Claim (1) is the usual evidence lower bound, 
	as in Lemma~8.21 of \cite{GhosalvdVbook2017}. Its validity  follows from the prior mass condition (2)
	and the fact that the Kullback-Leibler divergence and variation between $P_f^{(n)}$ and $P_{f_0}^{(n)}$ are
	equal to $n\|Kf-Kf_0\|_L^2/2$ and twice this quantity, respectively (e.g.\ Lemma~8.30 in the same reference
	or Lemma~9.1 in \cite{Yan2020}).
	Claim (2) is the content of Lemma~\ref{LemmaSmallBall}. Claims (3) and  (4) are proved at the end of this proof.
	
	Given claims (1)--(4) the proof of the theorem can be finished as follows.
	We bound $\Pi_n\bigl(f: \|f-f_0\|_{\Hil^{\d}}> M \h_n\given Y^{(n)}\bigr)$ by
	\begin{align*}&\t_n+ 1_{A_n^c}+\Pi_n\bigl(\|R_{j_n}Kf-f\|_{\Hil^{\d}}\ge \h_n\given Y_n\bigr)
		+\Pi_n(\F_n\given Y_n)1_{A_n}(1-\t_n).
	\end{align*}
	The expectation under $P_{f_0}^{(n)}$ of the first two terms tend to zero
	by the first part of claim (4) and claim (1). The expectation of the third term tends
	to zero by claims (3) and (2), in view of the remaining mass condition; see Theorem~8.20(iii') in 
	\cite{GhosalvdVbook2017}. To bound the fourth term, we use Bayes' formula to write
	the posterior probability of $\F_n$
	as the quotient of $\int_{\F_n} p_f^{(n)}/p_{f_0}^{(n)}(Y_n)\,d\Pi(f)$ and $\int p_f^{(n)}/p_{f_0}^{(n)}(Y_n)\,d\Pi(f)$, 
	and bound the denominator on the event $A_n$ below by $e^{-2n\e_n^2}$. This shows that the expectation of the
	fourth term is bounded above by
	$$e^{2n\e_n^2} P_{f_0}^{(n)}\int_{\F_n} \frac{p_f^{(n)}}{p_{f_0}^{(n)}}\,d\Pi(f)(1-\t_n)
	\le e^{2n\e_n^2} \sup_{f\in \F_n}P_f^{(n)}(1-\t_n).$$
	This tends to zero by the second part of (4).
	
{\sl 	Proof of claim (3).} We apply Lemma~\ref{LemmaTailBoundPriorGalerkin}
	with $j=j_n$ as specified at the beginning of the proof. 
We choose $t$ such that $bt^2j_n\r_{j_n}^{2(\a-\d)}=4n\e_n^2$, thus
reducing the right side of the lemma to $e^{-4n\e_n^2}$. We always have that
$\h_n=c_2(\r_{j_n}/c_{3,n})^{-(\a\wedge \b-\d)}\ge c_2(\r_{j_n}/c_{3,n})^{-(\a-\d)}$. Therefore it suffices to show that
also $\h_n\gtrsim t=   2\sqrt n\e_nj_n^{-1/2}\r_{j_n}^{\d-\a}/\sqrt b$. This follows, since $n\e_n^2=c_1^2n^{(2\a-2(\a\wedge\b)+d)/(2\a+2p+d)}$
and $j_n\asymp n^{d/(2\a+2p+d)}$.
	
	{\sl Proof of claim (4).} To construct the tests in (4), we note that observing $Y_n$ is equivalent to observing a Gaussian
	stochastic process $\bigl(Y_n(g): g\in L\bigr)$ with mean $\bigl(\langle Kf,g\rangle_L: g\in L\bigr)$
	and covariance function $\E Y_n(g) Y_n(h)=\langle g,h\rangle_L$. In particular, 
	for a given orthonormal basis $(\psi_i)_{i\le j}$ of $KV_j$, we observe
	the variable $\sum_{i=1}^j Y_n(\psi_i)\psi_i$, which we shall denote by 
$Q_jY_n$. (It can formally be understood to be the projection of $Y_n$ onto $KV_j$.)
It can be represented in distribution as 
	$$Q_jY_n=Q_jKf+\frac1{\sqrt n}\xi_j,$$
	for the Gaussian variable $\xi_j=\sum_{i=1}\langle \xi,\psi\rangle_L\psi$
 with mean zero and  $\E \langle \xi_j,g\rangle_L^2=\|Q_jg\|_L^2$.
	Then $R_jQ_jY_n= R_jKf+n^{-1/2}R_j \xi_j$ is a well defined variable with values in $V_j\subset \Hil^0$, with
	$R_j \xi_j$ a Gaussian random element in $V_j\subset \Hil^0$ with strong and weak second moments
	\begin{align*}
		\E\bigl\|R_j \xi_j\bigr\|_{\Hil^{0}}^2&\le \|R_j\|_{\Hil^{0}}^2\,\E\|\xi_j\|_L^2
		=\|R_j\|_{\Hil^{0}}^2\,\E\sum_{i\le j} \langle\xi_j,\psi_i\rangle_L^2=\|R_j\|_{\Hil^{0}}^2j\lesssim \r_j^{2p} j,\\
		\sup_{\|f\|_{\Hil^{0}}\le 1}\E \langle R_j \xi_j, f\rangle_{\Hil^{0}}^2
		&=\!\sup_{\|f\|_{\Hil^{0}}\le 1}\E \langle  \xi_j, R_j^*f\rangle_L^2
		=\! \sup_{\|f\|_{\Hil^{0}}\le 1}\| Q_jR_j^*f\|_L^2\le \|R_j^*\|_{\Hil^{0}}^2\lesssim \r_j^{2p}.
	\end{align*}
	In both cases the  inequality on the norms $\|R_j\|_{\Hil^{0}}=\|R_j^*\|_{\Hil^{0}}$ of the operators $R_j: L\to \Hil^0$ and
	its adjoint $R_j^*: \Hil^0\to L$ at the far right side follow from \eqref{EqGalerkinREstimate}.
	The first inequality implies that the first moment 
	$\E\bigl\|R_j \xi_j\bigr\|_{\Hil^{0}}$ of the variable $\|R_j\xi_j\|_{\Hil^{0}}$ is bounded above by $\r_j^p\sqrt j$. 
	By Borell's inequality (e.g. Lemma~3.1 in \cite{LedouxTalagrand} and subsequent discussion or
	Proposition~A.2.1 in \cite{vdVWellner2nd}),
	applied to the Gaussian random variable $R_j\xi_j$ in $\Hil^0$, we see that
	there exist constants $a,b>0$ such that, for every $t>0$,
	$$\Pr\bigl(\|R_j\xi_j\|_{\Hil^{0}}>a \r_j^{p}\sqrt j+t\bigr)\le e^{-bt^2/\r_j^{2p}}.$$
	Since $R_j\xi_j\in V_j$, we have $\|R_j\xi_j\|_{\Hil^{\d}}\le \r_j^{\d}\|R_j\xi_j\|_{\Hil^{0}}$ and hence
\begin{equation}
\Pr\bigl(\|R_j\xi_j\|_{\Hil^{\d}}>a \r_j^{p+\d}\sqrt j+\r_j^\d t\bigr)\le e^{-bt^2/\r_j^{2p}}.
\label{EqBorelGeneral}
\end{equation}
	Choose $j=j_n$ and $t^2=4n\e_n^2\r_{j_n}^{2p}/b$ to reduce the right side to $e^{-4n\e_n^2}$. Then 
	$\r_{j_n}^\d t=2\sqrt n\e_n\r_{j_n}^{p+\d}/\sqrt b\simeq\sqrt n\h_n$ and 
	$\r_{j_n}^{p+\d}\sqrt{j_n}\lesssim \sqrt n \h_n$. It follows that, for a sufficiently large constant $c_4$,
	\begin{align}
	P_f^{(n)}\bigl(\|R_{j_n}Y_n-R_{j_n}Kf\|_{\Hil^{\d}}> c_4 \h_n\bigr)\le e^{-4n\e_n^2}.\label{eq:borel:Hd}
	\end{align}
	We define
	$$\t_n=1\bigl\{\|R_{j_n}Y_n-R_{j_n}Kf_0\|_{\Hil^{\d}}> c_4 \h_n\bigr\}.$$
	It is immediate that $P_{f_0}^{(n)}\t_n\le e^{-4n\e_n^2}\ra 0$.
	Since $f_0\in \Hil^\b$ and $\d<\b$, we have 
	\begin{align}
	\|f_0-R_{j_n}Kf_0\|_{\Hil^{\d}}\le \r_{j_n}^{-(\b-\d)}\|f_0\|_{\Hil^\b}\le c_6\h_n\|f_0\|_{\Hil^\b}.\label{eq:UB:bias:test}
	\end{align}
	It follows that $\|R_{j_n}Kf-f\|_{\Hil^{\d}}<\h_n$ implies that
	$$\|R_{j_n}Kf-R_{j_n}Kf_0\|_{\Hil^{\d}}\ge\|f-f_0\|_{\Hil^{\d}}-\h_n-\h_n c_6\|f_0\|_{\Hil^\b}.$$
	If also $\|f-f_0\|_{\Hil^{\d}}>M\h_n$, then the right side is bounded below by $(M-1-c_6\|f_0\|_{\Hil^\b})\h_n$, and hence
	\begin{align*}\|R_{j_n}Y_n-R_{j_n}Kf\|_{\Hil^{\d}}
		&\ge (M-1-c-6\|f_0\|_{\Hil^\b})\h_n-\|R_{j_n}Y_n-R_{j_n}Kf_0\|_{\Hil^{\d}}.
	\end{align*}
	If $\t_n=0$, then the norm on the far right is bounded above by $c_4\h_n$. 
	It follows that $P_f^{(n)}(1-\t_n)\le \Pr\bigl(\|R_{j_n}Y_n-R_{j_n}Kf\|_{\Hil^{\d}}\ge c_5\h_n\bigr)$
	for $c_5=M-1-c_6\|f_0\|_{\Hil^\b}-c_4$, whenever $f\in\F_n$. For $M\ge 1+c_6\|f_0\|_{\Hil^\b}+2c_4$, the latter
	probability is bounded above by $e^{-4n\e_n^2}$.
\end{proof}

\begin{lemma}
	\label{LemmaGalerkin}
	If $K: \Hil^0\to L$ is a bounded linear operator satisfying \eqref{EqSmoothingProperty} for every $f\in \Hil^0$, 
	then the norms of the operators $R_j: L\to \Hil^0$ and $R_j K: \Hil^0\to \Hil^0$ and 
	$R_j K-I: \Hil^t\to \Hil^s$ satisfy
	\begin{align}
		\|R_j\|_{L\to \Hil^0} &\lesssim \r_j^p,\label{EqGalerkinREstimate}\\
		\|R_j K\|_{\Hil^s\to \Hil^s} &\lesssim 1,\qquad s\ge 0,\label{EqGalerkinRKEstimate}\\
		\|R_j K-I\|_{\Hil^t\to \Hil^s} &\lesssim \r_j^{-(t-s)},\qquad t\ge s\ge 0.\label{EqGalerkinRKminIEstimate}
	\end{align}
\end{lemma}

\begin{proof}
	For $s=0$ the lemma is the same as Lemma~A.2 in \cite{Yan2020}. For the extension 
	of \eqref{EqGalerkinRKEstimate} to general $s\ge 0$,
	we first use the triangle inequality to see that $\|R_jKf\|_{\Hil^{s}}\le \|R_jKf-P_jf\|_{\Hil^{s}}+\|P_jf\|_{\Hil^{s}}$, for
	$P_j: \Hil^s\to V_j\subset \Hil^s$ the orthogonal projection onto $V_j$. Since $P_j$ simply sets the coefficients
	of index larger than $j$ to zero, we have $\|P_jf\|_{\Hil^{s}}\le \|f\|_{\Hil^{s}}$, for any $s$. Because  $R_jKf-P_jf\in V_j$
	and $R_jK$ is the identity on $V_j$, we have 
	\begin{align*}
		\|R_jKf-P_jf\|_{\Hil^{s}}&\le \r_j^s\|R_jK(I-P_j)f\|_{\Hil^{0}}\le \r_j^s \|R_j\|_{L\to \Hil^0} \|K(I-P_j)f\|_L\\
		&\lesssim \r_j^s \r_j^p \|(I-P_j)f\|_{\Hil^{-p}}\le \|f\|_{\Hil^{s}},
	\end{align*}
	where we used \eqref{EqGalerkinREstimate} in the second last step and
	$\|(I-P_j)f\|_{\Hil^{-p}}\le \r_j^{-(p+s)}\|f\|_{\Hil^{s}}$ in the last. We conclude
	that $\|R_jKf\|_{\Hil^{s}}\lesssim \|f\|_{\Hil^{s}}$, for every $f$, which is \eqref{EqGalerkinRKEstimate}.
	
	Because $(R_jK-I)f=(R_jK-I)(I-P_j)f$, for every $f\in \Hil^s$, we have 
	$$\|(R_jK-I)f\|_{\Hil^{s}}\le \bigl(\|R_jK\|_{\Hil^s\to \Hil^s}+1\bigr)\|f-P_jf\|_{\Hil^{s}}.$$
	The right side is bounded by a multiple of $\|f-P_jf\|_{\Hil^{s}}\le \r_j^{-(t-s)}\|f\|_{\Hil^t}$, for $t\ge s$.
	This proves \eqref{EqGalerkinRKminIEstimate}.
\end{proof}

\begin{lemma}
	\label{LemmaSmallBall}
	Let $\a,\b,p>0$ and let
	$K: \Hil^0\to L$ be a bounded linear operator satisfying \eqref{EqSmoothingProperty} for every $f\in \Hil^0$.
Let $\GG=\bigl( \r_i^{-\a}i^{-1/2}\Ccor_i  Z_i\bigr)$,  for $Z_i\iid N(0,1)$ and constants $\Ccor_i$
that are bounded away from 0 and $\infty$.
	Then for every $f_0\in \Hil^\b$ there exists a constant $C$ so that, for every $\e<1$,
	$$\log \Pr\bigl(\|K\GG-Kf_0\|_L<C\e\bigr)\ge \Bigl(\frac1\e\Bigr)^{\frac{2\a-2\b+d}{p+\b}}
	+\Bigl(\frac1\e\Bigr)^{\frac{d}{p+\a}}.$$
	Consequently, there exists a constant $c_1$ such that $\e_n=c_1n^{-(\a\wedge\b+p)/(2\a+2p+d)}$ satisfies
	$\Pr\bigl(\|K\GG-Kf_0\|_L<\e_n\bigr)\ge e^{-n\e_n^2}$.
\end{lemma}

\begin{proof}
	By the smoothing property \eqref{EqSmoothingProperty}, the left side is bounded below by $\log \Pr\bigl(\|\GG-f_0\|_{\Hil^{-p}}<C_1\e\bigr)$.
	This is bounded below by, for $\HH$ the reproducing kernel Hilbert space of $\GG$,
	$$\inf_{h\in\HH: \|h-f_0\|_{\Hil^{-p}}<C_1\e} \|h\|_\HH^2-\log \Pr\bigl(\|\GG\|_{\Hil^{-p}}<C_1\e\bigr).$$
	The reproducing kernel Hilbert space of the Gaussian variable $\GG$ consists of the sequences
	$\bigl(\r_i^{-\a}i^{-1/2}\Ccor_iw_i\bigr)$, for $w\in \ell_2$, with square norm $\sum_iw_i^2$, or equivalently
	the sequences $(h_i)$ with square norm $\sum_i \r_i^{2\a}ih_i^2/\Ccor_i$.
	Since $i\simeq \r_i^d$, this norm is equivalent to the norm of $\Hil^{\a+d/2}$.
	Thus the first term in the preceding display is bounded above by a multiple of 
	$$\inf\Bigl\{ \sum_i \r_i^{2\a+d}h_i^2:  \sum_i \r_i^{-2p}(h_i-f_{0,i})^2<\e^2\Bigr\}.$$
	If $\a+d/2\le\b$, we can choose $h=f_0$ to see that the infimum is bounded above by
	$\|f_0\|_{\Hil^{\a+d/2}}^2\le \|f_0\|_{\Hil^\b}^2$. For  $\a+d/2>\b$, we choose $h_i=f_{0,i}$ for $i\le J$ and $h_i=0$ for $i>J$
	with $\r_J\simeq (1/\e)^{1/(p+\b)}$. Then 
	\begin{align*}
		\|h-f_0\|_{\Hil^{-p}}^2&= \sum_{i>J}\r_i^{-2p}f_{0,i}^2\le \r_J^{-(2p+2\b)}\sum_i \r_i^{2\b}f_{0,i}^2\lesssim\e^2,\\
		\|h\|_{\Hil^{\a+d/2}}^2&= \sum_{i\le J}\r_i^{2\a+d}f_{0,i}^2\le \r_J^{2\a+d-2\b}\sum_i \r_i^{2\b}f_{0,i}^2\lesssim
		\Bigl(\frac1\e\Bigr)^{\frac{2\a-2\b+d}{p+\b}}.
	\end{align*}
	This gives the first term of the lower bound. 
	
	The centered small probability can be bounded using
	direct computation or from entropy bound of the unit ball of $\Hil^{\a+\d/2}$ in $\Hil^{-p}$. 
	\end{proof}

\begin{lemma}
	\label{LemmaTailBoundPriorGalerkin}
Let $\GG=\bigl( \r_i^{-\a}i^{-1/2}\Ccor_i  Z_i\bigr)$,  for $Z_i\iid N(0,1)$ and constants $\Ccor_i$
that are bounded away from 0 and $\infty$.
	For all $\d<\a$, there exist constants $a,b>0$ so that, for every $j\in\NN$ and  $t>0$,
	$$\Pr\bigl(\|R_jK-I)\GG\|_{\Hil^{\d}}>a \r_j^{-(\a-\d)}+t\bigr)\le e^{-b t^2 j\r_j^{2(\a-\d)} }.$$
	Consequently, for every $b_1>0$, there exists $a_1$
	such that $\Pr\bigl(\|R_jK-I)\GG\|_{\Hil^{\d}}>a_1 \r_j^{-(\a-\d)}\bigr)\le e^{-b_1 j}$, for every $j\in\NN$.
\end{lemma}

\begin{proof}
	By Borell's inequality a centered Gaussian variable $G$ in a separable Banach space
	satisfies $\Pr\bigl(\|G\|\gtrsim \E \|G\|+t\bigr)\le e^{-bt^2/\sigma^2}$, for every $t>0$, where
	$\sigma^2=\sup_{b^*} \E (b^*G)^2$ is the supremum of the second moments of the variables
	$b^*G$ if $b^*$ ranges over the unit ball of the dual space of the Banach space.
	The variable $G=(R_jK-I)\GG$ is centered Gaussian in  $\Hil^\d$.
	Thus it suffices to show that
	$\E \|(R_jK-I)\GG\|_{\Hil^{\d}}\lesssim \r_j^{-(\a-\d)}$ and 
	$\E \langle (R_jK-I)\GG, g\rangle_{\Hil^{\d}}^2\lesssim j^{-1}\r_j^{-2(\a-\d)}\|g\|_{\Hil^{\d}}^2$,
	for every $g\in \Hil^\d$. Here we  can bound the first strong moment by the root of the second strong moment. 
	
	We have $\langle \GG, f\rangle_{\Hil^{\d}}=\sum_{i\in\NN} \r_i^{-\a}i^{-1/2}\Ccor_i Z_if_i\r_i^{2\d}$ and hence
	$$\E \langle \GG, f\rangle_{\Hil^{\d}}^2=\sum_{i\in\NN}\r_i^{-2\a+4\d}i^{-1}\Ccor_i^2f_i^2=\langle f, C_{2(\d-\a)}D^2f\rangle_{\Hil^{\d}},$$
	for $(D f)_i= f_ii^{-1/2}\Ccor_i$.
	Therefore, as a map into $\Hil^\d$ the variable $(R_jK-I)\GG$  has covariance operator 
	$(R_jK-I)C_{2(\d-\a)}D^2(R_jK-I)_\d^*;= SS_\d^*$, for $S=(R_jK-I)C_{\d-\a}D$, and 
	$(R_jK-I)_\d^*$ and $S_\d^*$  the adjoint operators of $(R_jK-I): \Hil^\d\to \Hil^\d$ and  $S_\d: \Hil^\d\to \Hil^\d$.
	(The operators $C_s: \Hil^\d\to \Hil^\d$, for $s\le 0$, and $D^2:  \Hil^\d\to \Hil^\d$ are self-adjoint, and commute,
	with roots $C_{s/2}$ and $D$.)
	
	The strong second moment is equal to the trace of the covariance operator.
	Therefore, for any orthonormal basis $(\phi_i)$ of $\Hil^\d$,
	\begin{align*}
		\E  \|(R_jK-I)\GG\|_{\Hil^{\d}}^2&=\tr\bigl(SS_\d^*)=\sum_{i\in \NN}\langle SS_\d^*\phi_i,\phi_i\rangle_{\Hil^{\d}}^2
		=\sum_{i\in \NN}\|S_\d^*\phi_i\|_{\Hil^{\d}}^2\\
		&=\sum_{i\in \NN}\sum_{j\in\NN}\langle S_\d^*\phi_i,\phi_j\rangle_{\Hil^{\d}}^2
		=\sum_{i\in \NN}\sum_{j\in\NN}\langle \phi_i,S\phi_j\rangle_{\Hil^{\d}}^2=\sum_{j\in\NN}\|S\phi_j\|_{\Hil^{\d}}^2.
	\end{align*}
	We apply this with the orthonormal basis consisting of the sequences $\phi_i=e_i/\r_i^\d$, for 
	$e_i=(0,\ldots, 0,1,0,\ldots)$ the sequence with a 1 in the $i$th coordinate and zeros elsewhere, to
	see that this is equal to 
	\begin{align*}
		&\sum_{i\in \NN}\bigl\|(R_jK-I)C_{\d-\a}De_i\r_i^{-\d}\bigr\|_{\Hil^{\d}}^2
		=\sum_{i\in \NN}\bigl\|(R_jK-I)e_i\bigr\|_{\Hil^{\d}}^2\,\r_i^{-2\a}i^{-1}\Ccor_i^2\\
		&\qquad\lesssim \sum_{i>j}\|R_jK-I\|_{\Hil^\d\to \Hil^\d}^2\|e_i\|_{\Hil^{\d}}^2\,\r_i^{-2\a}i^{-1}\le \sum_{i>j}\r_i^{-2(\a-\d)}i^{-1},
	\end{align*}
	since $R_jK-I$ vanishes on $V_j$, the norm of $R_jK-I: \Hil^\d\to \Hil^\d$ is bounded by a constant by
	\eqref{EqGalerkinRKminIEstimate}, and $\|e_i\|_{\Hil^{\d}}=\r_i^{\d}$. For $\d<\a$, the right side is bounded
	by a multiple of $\r_j^{-2(\a-\d)}$.
	
	The weak second moment is the supremum over $\|g\|_{\Hil^{\d}}\le 1$ of 
	\begin{align*}
		&\E \langle (R_jK-I)\GG, g\rangle_{\Hil^{\d}}^2=\langle g, SS_\d^*g\rangle_{\Hil^{\d}}=\|S_\d^*g\|_{\Hil^{\d}}^2
		=\|C_{\d-\a}D(R_jK-I)_\d^*g\|_{\Hil^{\d}}^2\\
		&\qquad\lesssim\|(R_jK-I)_\d^*g\|_{\Hil^{2\d-\a-d/2}}^2=\sup_{\|f\|_{\Hil^{\a+d/2}}\le 1}\langle f, (R_jK-I)_\d^*g\rangle_{\Hil^{\d}}^2,
	\end{align*}
	by \eqref{EqDualNormIdentity}, applied with $t-s=\d+(\d-\a-d/2)$ and $t+s=\d-(\d-\a-d/2)=\a+d/2$.
	The right side is equal to
	\begin{align*}
		&\sup_{\|f\|_{\Hil^{\a+d/2}}\le 1}\langle (R_jK-I) f, g\rangle_{\Hil^{\d}}^2
		\le \sup_{\|f\|_{\Hil^{\a+d/2}}\le 1}\bigl\|(R_jK-I) f\bigr\|_{\Hil^{\d}}^2 \|g\|_{\Hil^{\d}}^2\\
		&\qquad\le \sup_{\|f\|_{\Hil^{\a+d/2}}\le 1}\|R_jK-I\|_{\Hil^{\a+d/2}\to \Hil^\d}^2\|f\|_{\Hil^{\a+d/2}}^2 \|g\|_{\Hil^{\d}}^2
		\le \r_j^{-2(\a+d/2-\d)}\|g\|_{\Hil^{\d}}^2,
	\end{align*}
	by \eqref{EqGalerkinRKminIEstimate}. Thus the weak second moment is bounded by $\r_j^{-2(\a+d/2-\d)}
	\lesssim\r_j^{-2(\a-\d)} j^{-1}$.
	
	This concludes the proof of the first probability bound. The second bound follows by choosing
	$t=\sqrt{b_1/b} \, \r_j^{-(\a-\d)}$ in the first bound.
\end{proof}

\section{Smoothness norm contraction rates in the discrete observation linear inverse problems}
\label{sec:contract:smoothness:regression}
In this section we derive similar results as in Section~\ref{sec:contract:smoothness:GWN}, but in the context of the discrete observational framework. 
We extend the contraction rate results derived in Chapter 8 of \cite{Yan2020thesis} (or \cite{Yan2024})
to smoothness norms. 

The setup is the same as in  Section~\ref{sec:contract:smoothness:GWN}, and so is the general aim:
to recover a function $f$ from the noisy observation of its image $Kf$
under a given operator $K: G^0\to L$. However, presently the Hilbert space $L$ is a space
of functions on a domain $\O$, and for fixed points $x_{n,1},\ldots, x_{n,n}$ in $\O$, 
the observations are $Y_n(1),\ldots, Y_n(n)$ given by, for i.i.d.\ standard normal
variables $Z_i$,
\begin{equation}
	Y_n(i) =  Kf(x_{n,i}) + Z_i, \qquad i = 1, \ldots, n.\label{def:regression}
\end{equation}
We tackle this problem by putting it in the continuous setup, using an interpolation scheme.
To make this possible we assume that the image space $L$ is also embedded as $L=\Gil^0$ in
a scale of spaces $(\Gil^s)_{s\in\RR}$, in which the discretisation error can be controlled.

Define the empirical inner product and norm on the design points by
\begin{equation}\label{eq: empirical inner product}
	\langle f, g \rangle_{\LL_n} := \frac{1}{n}\sum_{i=1}^n g(x_{n,i}) h(x_{n,i}),
\qquad\quad \|f\|_{\LL_n} := \sqrt{\langle f, f \rangle_{\LL_n}}.
\end{equation}
We assume that, for every $n \in \NN$, there exists an $n$-dimensional subspace $\LL_n\subset L$
such that, for fixed constants $0<C_1<C_2<\infty$,
\begin{align}
\label{eq:regression_W_n_isomorphism2}
C_1 \|f\|_L \leq \|f\|_{\LL_n} \leq C_2 \|f\|_L,\qquad f\in \LL_n.
\end{align}
Under this condition there exist for every $f\in L$ a unique element $\mathcal{I}_nf$ of $\LL_n$
that interpolates $f$ at the design points, i.e.\ $\mathcal{I}_nf(x_i)=f(x_i)$, for every $i=1,\ldots, n$.
We assume that,  for every $s$ in some interval $(s_d,S_d)\subset (0,\infty)$, and numbers $\d_{n,s}$,
\begin{align}
\label{eq:regression_stable_reconstruction2}
\|\mathcal{I}_n f - f \|_L\le \d_{n,s} \|f\|_{\Gil^s}.
\end{align}
Fix an arbitrary orthonormal basis $e_{n,1},\ldots, e_{n,n}$ of $\LL_n$ relative to $\langle\cdot,\cdot \rangle_{\LL_n}$.
and given the observations $Y_n(1),\ldots, Y_n(n)$ in  \eqref{def:regression}, define
\begin{equation}
\label{EqDefinitionContinuousSignal}
Y^{(n)}=\sumin \frac1n\sumjn Y_n(j)e_{n,i}(x_{n,j}) \,e_{n,i}.
\end{equation}
This variable contains the same information as the original observations $Y_n(1),\ldots, Y_n(n)$
(and thus leads to the same posterior distribution),
and embeds the discrete data as a ``continuous signal'' into the space $\LL_n\subset L_2(\O)$.
The variable can be decomposed as
\begin{align}
Y^{(n)}=\sumin \langle Kf, e_{n,i}\rangle_{\LL_n}e_{n,i}+
\frac1n\sumjn Z_j\sumin e_{n,i}(x_{n,j}) \,e_{n,i}
=\mathcal{I}_n Kf+\frac 1{\sqrt n}\xi^{(n)},
\label{def:embedding2}
\end{align}
where  $\xi^{(n)}$ is a Gaussian random variable with values in the space in $\LL_n$. 
This connects the regression model to the Gaussian white noise model
and allows the application of the techniques developed for the latter idealised model.

To simplify the notation, we write $K_n=\mathcal{I}_nK$. 
Consider the same random sequence prior $\GG=\bigl( \r_i^{-\a}i^{-1/2}\Ccor_i  Z_i\bigr)$ 
as in Section~\ref{sec:contract:smoothness:GWN}, with $Z_i\iid N(0,1)$ and constants $\Ccor_i$ 
such that $c_\pi\le \Ccor_i\le C_\pi$, for some constants $0<c_\pi<C_\pi$.


\begin{theorem}\label{thm: contr:diff:norm:regression}
Assume that $c_0 i^{1/d}\le \r_i\le C_0 i^{1/d}$, for constants $C_0\ge c_0>0$ and $d>0$.	
Let $K: \Hil^0\to \Gil^0$ be a linear operator satisfying \eqref{EqSmoothingProperty} for every $f\in \Hil^0$ and such that $K: \Hil^\b\to \Gil^{\b+p}$ is continuous.	
If $f_0\in \Hil^\b$, for $s_d<\b+p<S_d$ and $\d<\a\wedge \b$, and 
\eqref{eq:regression_W_n_isomorphism2}-\eqref{eq:regression_stable_reconstruction2} hold with $n\d_{n,\b+p}^2\ra 0$,
 then there exists a constant $M$ such that
	$\E_{f_0} \Pi_n\bigl(f: \|f-f_0\|_{\Hil^{\d}}\le M n^{-(\a\wedge \b-\d)/(2\a+2p+d)}\given Y^{(n)}\bigr)\ra 1$.
\end{theorem}
	
\begin{proof}
  Define $\e_n$, $\h_n$ and $j_n$ as in \eqref{EqDefEnHn}, for given $c_1,c_2,c_{3,n}$.
  We show below that there exist constants $c_1,c_2,M, M_1>0$ such that
\begin{itemize}
\item[(1*)] The events $A_n:=\bigl\{\int p_f^{(n)}/p_{f_0}^{(n)}(Y^{(n)})\,d\Pi(f)\ge e^{-2n\e_n^2}\bigr\}$
		satisfy $P_{f_0}^{(n)}(A_n^c)\ra 0$.
\item[(2*)]  $\Pi\bigl(\|Kf-Kf_0\|_{\LL_n}<\e_n\bigr)\gtrsim e^{-n\e_n^2}$.
\item[(3*)]  $\Pi\bigl(\|R_{j_n} K_n f-f\|_{\Hil^{\d}}> M_1\h_n\bigr)\le e^{-4n\e_n^2}$.
\item[(4*)] There exist tests $\t_n$ such that $P_{f_0}^{(n)} \t_n\ra 0$ and
		$\sup_{f\in\F_n}P_f^{(n)}(1-\t_n)\le e^{-4n\e_n^2}$, where
		$$\F_n=\bigl\{f: \|f-f_0\|_{\Hil^{\d}}>M\h_n,\|R_{j_n} K_n f-f\|_{\Hil^{\d}}\le M_1\h_n\bigr\}.$$
	\end{itemize}
Then the proof can be completed exactly as the proof of Theorem~\ref{thm: contr:diff:norm},
	
{\sl Proof of claim (1*).} 
The Kullback-Leibler divergence and variation between the (normal) distributions of $(Y_1,\ldots, Y_n)$ 
under two functions $f$ and $f_0$ are given by $n\|K f-K f_0\|_{\LL_n}^2/2$ and twice this quantity, respectively. 
Hence togther with (2*), this is the usual evidence lower bound (as in  Lemma~8.21 of \cite{GhosalvdVbook2017}).

{\sl Proof of claim (2*)}. In view of Lemma~\ref{LemmaInterpolation}, $\|Kf - Kf_0\|_{\LL_n}\leq C_2 \|K_n f - K_nf_0\|_L$.
By the triangle inequality and \eqref{eq:regression_stable_reconstruction2}, since 
$s_d<\b+p<S_d$ by assumption,
\begin{align*}
\|K_nf - K_nf_0\|_L 
&\le\| K_nf - Kf \|_L+  \|K_nf_0 - Kf_0\|_L+\|Kf - Kf_0 \|_L \\
&\le \d_{n,\b+p}  \|Kf\|_{\Gil^{\b+p}}+\d_{n,\b+p}\|K f_0\|_{\Gil^{\b+p}}+\|Kf - Kf_0 \|_L.
\end{align*}
Since $Kf_0\in\Gil^{\b+p}$ by assumption and $\d_{n,\b+p}\ll \e_n$, the second term on the right is $o(\e_n)$, and negligible
for verifying (2*). Because $Kf$ is a centered Gaussian random element in $\Gil^{\b+p}$ by assumption,
Borell's inequality gives that $\Pr\bigl(\d_{n,\b+p} ( \|Kf\|_{\Gil^{\b+p}}>\e_n\bigr)\le e^{-c\e_n^2/\d_{n,\b+p}^2}$, for some
constant $c$, which is $o(e^{-n\e_n^2})$, since $n\d_{n,\b+p}^2\ra 0$, by assumption. We now use that
$\Pr(X+Y<\e)\ge \Pr(X<\e/2)-\Pr(Y>\e/2)$, valid for any random variables $X,Y$, to see that
the left side of (2*) is bounded below by $\Pi\bigl(\|Kf-Kf_0\|_L<\e_n/(2C_2)\bigr)$.
By Lemma~\ref{LemmaSmallBall} we can adjust the constant $c_1$ in $\e_n$ so that this is bounded
below by $e^{-n\e_n^2}$.

{\sl Proof of claim (3*).} By the triangle inequality,
\begin{align*}
\|R_{j_n} K_n f-f\|_{\Hil^\d}\leq \|R_{j_n} K f-f\|_{\Hil^{\d}}+\|R_{j_n} K f-R_{j_n} K_n f\|_{\Hil^{\d}}.
\end{align*}
The second term on the right hand side is bounded above by
\begin{align*}
\r_{j_n}^{\d}\|R_{j_n} K f-R_{j_n} K_n f\|_{\Hil^{0}}&\lesssim\r_{j_n}^{\d+p}\| K_n f-Kf \|_L
\lesssim  \r_{j_n}^{\d+p}\d_{n,\b+p}\|Kf\|_{\Gil^{\b+p}},
\end{align*}
where the second inequality follows by  \eqref{EqGalerkinREstimate} and the third by
\eqref{eq:regression_stable_reconstruction2}.  By Borell's inequality applied to the
variable $Kf$, the prior probability that
the right side is bigger than a multiple of $\h_n$ is smaller than
$\exp(-c\h_n^2\rho_{j_n}^{-2\d-2p}\d_{n,\b+p}^{-2})\ll e^{-n\e_n^2}$,  since   $\h_n/\r_{j_n}^{\d+p}\asymp \e_n$ and
$n\d_{n,\b+p}^2\ra 0$. Therefore the probability on the left hand side of (3*) is bounded above
by $	\Pi\bigl(\|R_{j_n} K f-f\|_{\Hil^{\d}}>M_1\h_n/2\bigr)\leq e^{-4n\eps_n^2}$,
in view of claim (3) from the proof of Theorem~\ref{thm: contr:diff:norm}.

{\sl Proof of claim (4*).}
For a given constant $M_0>0$, to be determined later, consider  the test
$\tau_n=1\bigl\{\|R_{j_n}Y^{(n)}-R_{j_n}K_nf_0\|_{\Hil^{\d}}\ge M_0 \eta_n\bigr\}$.

For $Y^{(n)}$ given in \eqref{def:embedding2} we have 
\begin{equation}
R_j Y^{(n)}=R_j {K}_n f+\frac1{\sqrt n}R_j \xi^{(n)}.
\label{EqRY}
\end{equation}
The variable $R_j \xi^{(n)}=R_j Q_j\xi^{(n)}$ is a centered Gaussian random element in $V_j$ with strong and weak second moments,
\begin{align*}
\E\bigl\|R_j \xi^{(n)}\bigr\|_{\Hil^{0}}^2&\le \|R_j\|_{L\mapsto \Hil^{0}}^2\E\|Q_j\xi^{(n)}\|_L^2
\lesssim\r_j^{2p} j,\\
\sup_{\|f\|_L\le 1}\E \langle R_j \xi^{(n)}, f\rangle_{\Hil^{0}}^2
&=\sup_{\|f\|_L\le 1}\E \langle \xi^{(n)}, R_j^*f\rangle_L^2
\lesssim \sup_{\|f\|_L\le 1}\| R_j^*f\|_L^2\le \|R_j^*\|_{L\mapsto \Hil^{0}}^2\lesssim \r_j^{2p}.
\end{align*}
In both cases the  inequality on $\|R_j\|_{L\mapsto \Hil^{0}}=\|R_j^*\|_{L\mapsto \Hil^{0}}$ 
at the far right side follows from \eqref{EqGalerkinREstimate}, and we also use 
that the covariance operator of $\xi^{(n)}$ is bounded above by a multiple of the projection onto $\LL_n$,
and hence the identity, so that the covariance operator of $Q_j \xi^{(n)}$ is bounded above
by a multiple of $Q_j$. Therefore \eqref{EqBorelGeneral} is valid for $R_j\xi^{(n)}$ in the place
of $R_j\xi_j$.  
Choose $j=j_n$ and $t^2=4n\e_n^2\r_{j_n}^{2p}/b$ to reduce the right side of the latter equation to $e^{-4n\e_n^2}$ and
conclude that for a sufficiently large constant $M_0$ 
	\begin{align*}
	P_f^{(n)}\bigl(\|R_{j_n}Y^{(n)}-R_{j_n}K_nf\|_{\Hil^{\d}}> M_0 \h_n\bigr)\le e^{-4n\e_n^2}=o(1).
	\end{align*}
This is the equivalent of \eqref{eq:borel:Hd}, with $K$ replaced by $K_n$.

This and the definition of $\tau_n$  immediately give that $P_{f_0}^{(n)}\t_n\ra 0$. 
To bound the type-2 error, we first note that in view of  assertion \eqref{EqGalerkinREstimate} and assumption \eqref{eq:regression_stable_reconstruction2},
\begin{align*}
\|R_{j_n}Kf_0-R_{j_n}K_nf_0\|_{\Hil^{\d}}&\leq \r_{j_n}^{\d} \|R_{j_n}Kf_0-R_{j_n}K_nf_0\|_{\Hil^{0}}\\
&\leq\r_{j_n}^{\d} \|R_{j_n}\|_{L \mapsto \Hil^{0}}\,\|K_n f_0- Kf_0 \|_L \\
& \lesssim \r_{j_n}^{\d+p}\d_{n,\b+p}\|Kf_0\|_{\Gil^{\b+p}}\ll  \eta_n.
 \end{align*}
Combining this with \eqref{eq:UB:bias:test} and the triangle inequality, we see that there exists $c_0>0$ such that
	$$ \|f_0-R_{j_n}K_nf_0\|_{\Hil^{\d}}\leq  \|f_0-R_{j_n}Kf_0\|_{\Hil^{\d}}+ \|R_{j_n}Kf_0-R_{j_n}K_nf_0\|_{\Hil^{\d}}\leq c_0  \eta_n.$$
It follows that $\|R_{j_n}K_nf-f\|_{\Hil^{\d}}<M_1\h_n$ and $\|f-f_0\|_{\Hil^{\d}}>M\h_n$ imply
	$$\|R_{j_n}K_nf-R_{j_n}K_nf_0\|_{\Hil^{\d}}\ge\|f-f_0\|_{\Hil^{\d}}-M_1\h_n-c_0\h_n \geq(M-M_1-c_0)\h_n.$$
Then, for $\t_n=0$ and $f\in\F_n$,
	\begin{align*}\|R_{j_n}Y^{(n)}-R_{j_n}K_nf\|_{\Hil^{\d}}
		&\ge (M-M_1-c_0)\h_n-\|R_{j_n}Y^{(n)}-R_{j_n}K_nf_0\|_{\Hil^{\d}}\\
		&\geq  (M-M_1-c_0-M_0)\h_n.
	\end{align*}
It follows that $P_f^{(n)}(1-\t_n)\le \Pr\bigl(\|R_{j_n}Y^{(n)}-R_{j_n}K_nf\|_{\Hil^ {\d}}\ge c_5\h_n\bigr)$, 
	for $c_5=M-M_1-c_0-M_0$,  which is bounded above by $e^{-4n\e_n^2}$ for sufficiently large $M$,
by the last inequality of the preceding paragraph.
\end{proof}

The following next lemma is Lemma 8.3 of \cite{Yan2020thesis} or Lemma 2.9 of \cite{Yan2024}.

\begin{lemma}
\label{LemmaInterpolation}
If \eqref{eq:regression_W_n_isomorphism2} holds, then the interpolation map $\mathcal{I}_n$
is well defined, and for any orthonormal basis $e_{n,1},\ldots, e_{n,n}$ of $\LL_n$ relative to
the empirical inner product \eqref{eq: empirical inner product} and everyevery $g\in L$:
\begin{itemize}
\item $\mathcal{I}_ng$ is the orthogonal projection of $g$ onto $\LL_n\subset L$.
relative to inner product \eqref{eq: empirical inner product}, 
\item  $C_1\|\mathcal{I}_ng\|_L\le \|g\|_{\LL_n}\le C_2\|\mathcal{I}_ng\|_L$.
\end{itemize}
Furthermore, the Gram matrix $\bigl(\langle e_{n,i}, e_{n,j}\rangle_L\bigr){}_{i,j=1..n}$ is bounded below and above by the
identity matrix up to multiplicative constants that do not depend on $n$.
\end{lemma}

\section{Empirical Bayes contraction in smoothness norm}\label{sec:EB:delta}
We extend the contraction results of \cite{Knapik2016} from the squared norm $\|v\|_{L_2}^2=\sum_{i=1}^\infty v_i^2$
to the squared norm $\|v\|_{G^{\d}}^2=\sum_{i=1}^\infty v_i^2i^{2\d}$. We employ the same notation, and follow the proof closely.

For given $\m$ define the function   $h_n(\cdot;v):(0,\infty)\to [0,\infty)$ by
\begin{equation}\label{eq: h}
h_n(\a;v)=\frac{1+2\a+2p}{n^{1/(1+2\a+2p)}\log n}\sum_{i=1}^{\infty}\frac{n^2i^{1+2\a} v_{0,i}^2\log i}{(i^{1+2\a}\k_i^{-2}+n)^2}.
\end{equation}
For positive constants $0 < l<L$ we define lower and upper bounds as
\begin{align}
\underline{\a}_n(v)&=\inf\{\a>0: h_n(\a;v)>l\}\wedge\sqrt{\log n}\label{eq: LB}, \\
\overline{\a}_n(v)&=\inf\{\a>0: h_n(\a;v)>L(\log n)^2\}\label{eq: UB}.
\end{align}
The infimum over the empty set is considered $\infty$.

The following lemma and theorem are identical as in \cite{Knapik2016} and need
no new proof.

\begin{lemma}\label{lem: abar}
For any $l, L > 0$  the following statements hold.
\begin{enumerate}
\item[(i)]
For all $\beta , R > 0$, there exists $c_0 > 0$ such that, for $n$ large enough,
\[
\inf_{\|v\|_{G^{\b}} \le R} \underline\a_n(v) \ge \beta -\frac{c_0}{\log n}
\]
\item[(ii)]
If $v_{,i} \not = 0$ for some $i \ge 2$, then $\overline\a_n(v) \le  (\log n)/(2\log 2) - 1/2- p$ for $n$ large enough.
\end{enumerate}
\end{lemma}

\begin{theorem}\label{thm: AlphaMagnitude}
For every $R > 0$
the constants $l$ and $L$ in \eqref{eq: LB} and \eqref{eq: UB} can be chosen such that
\[
\inf_{\|v\|_{G^{\b}} \le R} \Pr_v \Big(\arg\max_{\a \in [0,\log n]} \ell_n(\a) \in \bigl[\underline{\a}_n(v), \overline{\a}_n(v)\bigr]\Big) \to 1.
\]
\end{theorem}

The following two theorems change the norm to $\|\cdot\|_{G^{\d}}$, with a resulting
decrease in the rate. The first theorem concerns the plug-in empirical Bayes
posterior, the second theorem the hierarchical posterior, where the assumption
on the hyper prior of $\a$ is the same as in \cite{Knapik2016}.
Set $\ell_n=(\log n)^{{ 3/2}}(\log\log n)^{1/2}$.

\begin{theorem}\label{thm: ConvergenceEB}
For every $\beta, \gamma, R > 0$, { $\d\in (-p-1/2,\b)$} and $M_n \to \infty$,
\[
\sup_{\|v_0\|_{G^{\b}} \le R} \E_{v_0} \Pi_{\hat{\a}_n}\bigl( v: \|v-v_0\|_{G^\d} \geq M_n\ell_nn^{-{(\b-\d)}/(1+2\b+2p)}\, \big|\, Y_n\bigr) \to 0.
\]
\end{theorem}

\begin{theorem}\label{thm: ConvergenceHB}
If the hyper  density satisfies Assumption~1, then 
for every $\beta, \gamma, R > 0$, { $\d\in (-p-1/2,\b)$} and $M_n \to \infty$ we have
\[
\sup_{\|v_0\|_{G^{\b}} \le R} \E_{v_0} \Pi\bigl(v:  \|v-v_0\|_{G^\d} \geq M_n\ell_nn^{-{(\b-\d)}/(1+2\b+2p)}\, \big|\, Y_n\bigr) \to 0.
\]
\end{theorem}

\subsection{Proof of Theorem~\ref{thm: ConvergenceEB}}\label{sec: ProofSob}
We assume the exact equality $\kappa_i = i^{-p}$ for simplicity. Let $\e_n=\ell_n n^{-{(\b-\d)}/(1+2\b+2p)}$.

By Markov's inequality and Theorem \ref{thm: AlphaMagnitude},
\begin{align*}
& \sup_{\|v_0\|_{G^{\b}} \le R} \E_{v_0} \Pi_{\hat{\a}_n}\bigl(\|v-v_0\|_{G^\d} \geq M_n\e_{n}\, \big|\, Y_n\bigr)\\
&\qquad\qquad\le \frac{1}{M_n^2\e_n^2} \sup_{\|v_0\|_{G^{\b}} \le R} \E_{v_0} \sup_{\a\in[\underline{\a}_{n,0}, \overline{\a}_{n,0}\wedge \log n]}
R_n(\a) + o(1),
\end{align*}
where
\[
R_n(\a) = \int \|v-v_0\|_{G^\d}^2\,\Pi_\a(dv \given Y_n).
\]
By the explicit form of the posterior distribution 
\begin{equation}
\label{eq: PostExp}
R_n(\a)
=\sum_{i=1}^\infty (\hat{v}_{\a,i}-v_{0,i})^2{ i^{2\d}}
+\sum_{i=1}^{\infty}\frac{i^{2p+{ 2\d}}}{i^{1+2\a+2p}+n},
\end{equation}
where $\hat{v}_{\a,i}=ni^{p}(i^{1+2\a+2p}+n)^{-1}Y_{n,i}$
is the $i$th coefficient of the posterior mean. 

We divide the Sobolev-ball $\|v\|_{G^{\b}} \le R$ into two subsets
\begin{align*}
P_n  & = \{v: \|v \|_{G^{\b}} \le R, \ \overline{\a}_n(v) \leq (\log n)/\log 2-1/2-p\},\\
Q_n  & = \{v: \|v \|_{G^{\b}} \le R, \ \overline{\a}_n(v) > (\log n)/\log 2-1/2-p\}.
\end{align*}
{ For $\a=(\log n)/\log 2-1/2-p$,  we have $n^{1/(1+2\a+2p)}=\sqrt 2$. In the proof below the summation index $i$ is
often cut at the latter quantity. If $v\in Q_n$, then $i=1$ is the only index below the  cut-off. Also if $v\in Q_n$, the
upper bound $\overline\a_n(v)$ may be infinite. If $\overline\a_n(v)<\infty$, then we have
the bound $h_n\bigl(\overline\a_n(v);v\bigr)\le L (\log n)^2$.}

\subsubsection{Bound for the expected posterior risk over $P_n$}
\label{sec: 61}
For ease of notation abbreviate $\underline\a_{n,0}=\underline\a_n(v_0)$ and $\overline\a_{n,0}=\overline\a_n(v_0)$.
In this section we  prove that
\begin{equation}\label{eq: PnExp}
\sup_{v_0\in P_n} \sup_{\a\in[\underline\a_{n,0}, \overline\a_{n,0}]} \E_{v_0} R_n(\a) = O(\e_n^2).
\end{equation}
The expectation of the first term in \eqref{eq: PostExp} can be split into a square bias and a variance term. This gives
\begin{equation}\label{eq: PostRisk}
\E_{v_0} R_n(\a)=\sum_{i=1}^\infty \frac{i^{2+4\a+4p{+2\d}}v_{0,i}^2}{(i^{1+2\a+2p}+n)^2}
+{n}\sum_{i=1}^\infty \frac{i^{2p{+2\d}}}{(i^{1+2\a+2p}+n)^2}
+\sum_{i=1}^\infty \frac{i^{2p{+2\d}}}{i^{1+2\a+2p}+n}.
\end{equation}
The second and third terms in
\eqref{eq: PostRisk} are deterministic and free of $v_0$, and the second is bounded by
the third. By Lemma~\ref{lem: LogM} (with $m=0$, $l=1$, $r=1+2\a+2p$ and $s=2p{+2\d}$
and ${-p-1/2<\d<\a}$), the third term  is for  $\a \ge\underline{\a}_{n,0}$ further bounded by
\begin{align*}
 n^{-\frac{{2(\a-\d)}}{1+2\a+2p}}
\leq  n^{-\frac{{2(\underline{\a}_{n,0}-\d)}}{1+2\underline{\a}_{n,0}+2p}}.
\end{align*}
In view of Lemma \ref{lem: abar}.(i),  the right-hand side is bounded
by a constant times $n^{-{2(\beta-\d)}/(1+2\beta+ 2p)}$ for large $n$.

It remains to consider the first sum in \eqref{eq: PostRisk}, which
we divide into three parts and show that each of the parts has the stated order. First we note that
\begin{equation*}
\sum_{i > n^{1/(1+2\b+2p)}} \frac{i^{2+4\a+4p{+2\d}}v_{0,i}^2}{(i^{1+2\a+2p}+n)^2}
\le \sum_{i > n^{1/(1+2\b+2p)}}v_{0,i}^2{ i^{2\d}} \leq \|v_0\|^2_{G^{\b}} n^{-2{(\b-\d)}/(1+2\b+2p)}.
\end{equation*}
For $0< \a\leq\log n/(2\log 2)-1/2-p$ and { sufficiently large $n$}, 
the maximum of the function $i \mapsto
i^{1+2\a + 4p{+2\d}}/\log i$ over the interval $[2, n^{1/(1+2\a + 2p)}]$ is attained
at $i = n^{1/(1+2\a + 2p)}$. (As long as $r:=1+2\a+4p+2\d\ge0$, the function is increasing at least on $[e^{1/r},\infty)$,  as $x\mapsto x^r/\log x$ is, and its value 
on $[2,e^{1/r}]$ is no smaller than its value at $n$ for sufficiently large $n$.) It follows that for $\a > 0$,
\begin{align*}
& \sum_{i \le n^{1/(1+2\a+2p)}} \frac{i^{2+4\a+4p{+2\d}}v_{0,i}^2}{(i^{1+2\a+2p}+n)^2} \\
& \quad = \frac{v_{0,1}^2}{(1+n)^2}
+ \frac1{n^2}\sum_{2 \le i \le n^{1/(1+2\a+2p)}} \frac{(i^{1+2\a+4p{+2\d}}/\log i)n^2i^{1+2\a}v_{0,i}^2\log i}{(i^{1+2\a+2p}+n)^2} \\
& \quad \le \frac{v_{0,1}^2}{(1+n)^2} + n^{-\frac{2{(\a-\d)}}{1+2\a + 2p}}h_n(\a;v_0).
\end{align*}
For $\a>\log n/(2\log 2)-1/2-p$, we have $n^{1/(1+2\a+2p)}<2$ and  the second term 
of the preceding display disappears, and for $v_0 \in P_n$, we have that $\overline\a_{n,0}$ is finite, whence
$h_n(\overline\a_{n,0};v_0)\le L(\log n)^2$. Since $n^{1/(1+2\overline\a_{n,0}+2p)} \le n^{1/(1+2\a+2p)}$
for $\a \le \overline\a_{n,0}$, the preceding implies that
\begin{align*}
 \sup_{v_0\in P_n} \sup_{\a\in[\underline{\a}_{n,0}, \overline{\a}_{n,0}]} 
\sum_{i \le n^{1/(1+2\bar\a_n+2p)}} \frac{i^{2+4\a+4p{+2\d}}v_{0,i}^2}{(i^{1+2\a+2p}+n)^2} 
\lle \frac{R^2}{n^2} + L n^{-\frac{2{(\underline{\a}_{n,0}-\d)}}{1+2\underline{\a}_{n,0} + 2p}} \log^2n.
\end{align*}
By Lemma \ref{lem: abar}, $\underline{\a}_{n,0} \ge \beta - c_0/\log n$ for a constant $c_0 > 0$
(only depending on $\beta, R, p$). Hence,
using  that { $x \mapsto (x-\d)/(1+x+2p)$ is increasing for every $\d+1+2p > 0$} the right-hand side is bounded
by a constant times $n^{-2{(\beta-\d)}/(1+2\beta+ 2p)}\log^2 n$.

To complete the proof
 we deal with the terms between $n^{1/(1+2\overline{\a}_{n,0}+2p)}$ and $n^{1/(1+2\b+2p)}$.
Let $J=J(n)$ be the smallest integer such that $\overline{\a}_{n,0}
/(1+1/\log n)^J \leq \b$. One can see that $J$ is bounded above by a multiple of $(\log n)(\log\log n)$ for any positive $\b$.
We partition the summation range under consideration into $J$ pieces using the auxiliary numbers
\[
b_j = 1+ 2\frac{\overline{\a}_{n,0}
}{(1+1/\log n)^j}+2p, \qquad j =0, \ldots, J.
\]
Note that the sequence $b_j$ is decreasing. Now we have
\[
\sum_{i= n^{1/(1+2\overline{\a}_{n,0}+2p)}}^{n^{1/(1+2\b+2p)}}
\frac{i^{2+4\a+4p{+2\d}}v_{0,i}^2}{(i^{1+2\a + 2p}+n)^2}
\leq \sum_{j=0}^{J-1}\sum_{i=n^{1/b_j}}^{n^{1/b_{j+1}}}{ i^{2\d}}v_{0,i}^2 
\leq 4\sum_{j=0}^{J-1}\sum_{i=n^{1/b_j}}^{n^{1/b_{j+1}}}\frac{ni^{b_j}{ i^{2\d}}v_{0,i}^2}{(i^{b_{j+1}}+n)^2}.
\]
Since $(b_j-b_{j+1})\log n = b_{j+1}-1-2p$, it holds for $n^{1/b_j}\leq i \leq n^{1/b_{j+1}}$ that
 $i^{b_j-b_{j+1}} \leq n^{1/\log n} = e$. On the same interval $i^{2p{+2\d}}$ is bounded by $n^{{(2p+2\d)}/b_{j+1}}$. 
Therefore the right hand side of the preceding display is further bounded by a constant times
\begin{align*}
&\sum_{j=0}^{J-1}\sum_{i=n^{1/b_j}}^{n^{1/b_{j+1}}}\frac{ni^{b_{j+1}}{ i^{2\d}}v_{0,i}^2\log i}{(i^{b_{j+1}}+n)^2}
\leq \sum_{j=0}^{J-1}n^{(2p{+2\d})/b_{j+1}-1}\sum_{i=n^{1/b_j}}^{n^{1/b_{j+1}}}\frac{n^2i^{b_{j+1}-2p}v_{0,i}^2\log i}{(i^{b_{j+1}}+n)^2{ \log n^{1/b_j}}}\\
&\leq \sum_{j=0}^{J-1}n^{(2p{+2\d})/b_{j+1}-1}h_n\biggl(\frac{\overline{\a}_{n,0}
}{(1+1/\log n)^{j+1}};v_0\biggr)n^{1/b_{j+1}}\frac{ b_j}{b_{j+1}}\\
&\lesssim {\sum_{j=0}^{J-1}n^{(1+2p{+2\d}-b_{j+1})/b_{j+1}}h_n(b_{j+1}/2-1/2-p; v_0)}\\
&\lesssim n^{-\frac{2\b/(1+1/\log n){-2\d}}{1+2\b/(1+1/\log n)+2p}}\sum_{j=0}^{J-1}h_n(b_{j+1}/2-1/2-p; v_0).
\end{align*}
In the last step we used the fact that for $j\le J$,  $b_j\ge b_J\ge 1+2\b/(1+1/\log n)+2p$, by the definition of $J$.
Because $b_{j}/2-1/2-p \le\overline{\a}_{n,0}$ for every $j\ge0$, it follows from
the definition of $\overline{\a}_{n,0}$ that $h_n(b_{j}/2-1/2-p; v_0)$
is bounded above by $L(\log n)^2$, and we recall that $J=J(n)$ is bounded
above by a multiple of $(\log n)(\log\log n)$. Finally we note that
\begin{align*}
n^{-\frac{2\b/(1+1/\log n){-2\d}}{1+2\b/(1+1/\log n)+2p}}  \lesssim n^{-2{(\b-\d)}/(1+2\b+2p)}.
\end{align*}
Therefore the first sum in \eqref{eq: PostRisk} over the range $[n^{1/(1+2\overline{\a}_{n,0}+2p)},
n^{1/(1+2\b+2p)}]$ is bounded above by a multiple of $n^{-2{(\b-\d)}/(1+2\b+2p)}(\log n)^{ 3}(\log\log n)$, in
the appropriate uniform sense over $P_n$. Putting the bounds above together we conclude 
the proof of \eqref{eq: PnExp}.

\subsubsection{Bound for the centered posterior risk over $P_n$}\label{sec: brisk}

We show in this section that for the set $P_n$ we also have
\[
\sup_{v_0\in P_n}\E_{v_0} \sup_{\a\in[\underline{\a}_{n,0},\overline{\a}_{n,0}]}\Bigl|\sum_{i=1}^\infty \bigl(\hat{v}_{\a,i}-v_{0,i}\bigr)^2{ i^{2\d}} - \E_{v_0} \sum_{i=1}^\infty \bigl(\hat{v}_{\a,i}-v_{0,i}\bigr)^2{ i^{2\d}}\Bigr| = O(\e_n^2).
\]
Using the explicit expression for the posterior mean $\hat{v}_{\a,i}$,
 we see that the random variable in the supremum is the absolute value of $\VV(\a)/n-2\WW(\a)/\sqrt{n}$, where
\[
\VV(\a)=\sum_{i=1}^\infty \frac{n^2\k_i^{-2}{ i^{2\d}}}{(i^{1+2\a}\k_i^{-2}+n)^2}(Z_i^2-1), 
\qquad \WW(\a)= \sum_{i=1}^\infty \frac{ni^{1+2\a}\k_i^{-3}v_{0,i}{ i^{2\d}}}{(i^{1+2\a}\k_i^{-2}+n)^2}Z_i.
\]
We deal with the two processes separately.

For the process $\VV$, Corollary 2.2.5 in \cite{vdVWellner} implies that
\[
\E_{v_0} \sup_{\a\in[\underline{\a}_{n,0},\infty)}|\VV(\a)| \lle
\sup_{\a\in[\underline{\a}_{n,0},\infty)}\sqrt{\var_{v_0}\VV(\a)} +
\int_0^{\diam_n}\sqrt{N(\e, [\underline{\a}_{n,0},\infty), d_n)}\,d\e,
\]
where $d^2_n(\a_1, \a_2) = \var_{v_0}\bigl(\VV(\a_1)-\VV(\a_2)\bigr)$ and $\diam_n$ is the $d_n$-diameter of
$[\underline{\a}_{n,0},\infty)$.
Since $\var_{v_0} Z_i^2 = 2$, the variance of $\VV(\a)$ is equal to
\[
\var_{v_0}\VV(\a) =  2n^4\sum_{i=1}^\infty \frac{i^{4p{+4\d}}}{(i^{1+2\a+2p}+n)^4}\lesssim n^{(1+4p{+4\d})/(1+2\a+2p)}\vee \log n,
\]
by Lemma~\ref{lem: LogM} (with $m=0$, $l=4$, $r=1+2\a+2p$ and $s=4p{+4\d}$).
{ The maximum with  $\log n$ arises for $\d\le -p-1/4$ (it can be the maximum with a constant if $\d<-p-1/4$).}
It follows that $\diam_n \lle r_n:=n^{(1+4p{+4\d})/{ (2+4\underline{\a}_{n,0}+4p)}}\vee \log n$.
To  compute the covering number of the interval $[\underline{\a}_{n,0}, \infty)$
we first note that for $0 < \a_1 < \a_2$,
\begin{align*}
\var_{v_0}\bigl(\VV(\a_1)& -\VV(\a_2)\bigr) = \sum_{i=2}^\infty \biggl(\frac{n^2i^{2p{+2\d}}}{(i^{1+2\a_1+2p}+n)^2}-
\frac{n^2i^{2p{+2\d}}}{(i^{1+2\a_2+2p}+n)^2}\biggr)^2\var Z_i^2\\
&\leq 2\sum_{i=2}^\infty \frac{n^4i^{4p{+4\d}}}{(i^{1+2\a_1+2p}+n)^4}
\leq 2n^4\sum_{i=2}^\infty i^{-4-8\a_1-4p{+4\d}} \lle n^42^{-{ 4}\a_1},
\end{align*}
for ${ \d<\b\le \a_1}$. Hence for $\e > 0$,
a single $\e$-ball covers the whole interval $[K\log(n/\e), \infty)$ for some constant $K > 0$.
By Lemma~\ref{lem: VarVW}, the distance $d_n(\a_1, \a_2)$ is bounded above by a multiple of 
$|\a_1-\a_2|r_n(\log n)$. 
Therefore the covering number of the interval $[\underline{\a}_{n,0}, K\log(n/\e)]$ 
relative to $d_n$ is bounded above by a multiple of $(\log n)r_n(\log(n/\e))/\e$. 
Since $\diam_n\lesssim r_n$, this gives an entropy integral
$${\int_0^{r_n}\sqrt{\frac{(\log n)r_n \log(n/\e)}{\e}}\,d\e
\lesssim r_n\sqrt{\log n}\int_0^1\sqrt{\frac{ \log (n/r_n)+\log (1/\h)}{\h}}\,d\h\lesssim r_n(\log n)}.$$
Combining everything we see that
\begin{align*}
\E_{v_0}\sup_{\a\in[\underline{\a}_{n,0},\infty)}\frac{|\VV(\a)|}n
&\lesssim \frac{r_n\log n}n
\lesssim n^{-\frac{{ 1/2+}2\underline{\a}_{n,0}{-2\d}}{1+2\underline{\a}_{n,0}+2p}}(\log n)\vee (\log n)^2/n.
\end{align*}
This is smaller than a multiple of $\e_n^2$ in view of Lemma~\ref{lem: abar}.(i).

It remains to deal with the process $\WW$.
The variance of $\WW(\a)/\sqrt{n}$  is given by
\[
\var_{v_0}\biggl(\frac{\WW(\a)}{\sqrt{n}}\biggr) = \sum_{i=1}^\infty \frac{ni^{2+4\a+ 6p{+4\d}}v_{0,i}^2}{(i^{1+2\a+2p}+n)^4}.
\]
We show that uniformly in $\a \in [\underline{\a}_{n,0}, \overline{\a}_{n,0}]$, 
this variance is bounded above by a constant that depends only on $\|v_0\|_{G^{\b}}$
times $r_n^2(\log n)^2$, for $r_n:=n^{-(1+4{(\b-\d)})/(2+4\b+4p)}$.
On the set $P_n$ the upper bound $\overline{\a}_{n,0}\leq \log n/\log2-1/2-p$ is finite and hence
$h_n(\overline{a}_{n,0};v_0)\le L(\log n)^2$.

For the sum over  $i \leq n^{1/(1+2\a+2p)}$ we have
\begin{equation}\label{eq: piet}
\begin{split}
&\qquad\quad\sum_{i \le n^{1/(1+2\a+2p)}} \frac{ni^{2+4\a+6p{+4\d}}v_{0,i}^2}{(i^{1+2\a+2p}+n)^4}\\
& \leq
\frac{v_{0,1}^2}{n^3} +
\frac{1}{n^3}\sum_{2 \le i \le n^{1/(1+2\a+2p)}} \frac{i^{1+2\a+6p{+4\d}}(\log i)^{-1}n^2i^{1+2\a}v_{0,i}^2\log i}{(i^{1+2\a+2p}+n)^2}\\
&\leq
\frac{\|v_0\|^2_{G^{\b}}}{n^3} +
(1+2\a+2p)\frac{n^{(4p{+4\d)}/(1+2\a+2p)}}{(\log n)n^2}\sum_{i \le n^{1/(1+2\a+2p)}} \frac{n^2i^{1+2\a}v_{0,i}^2\log i}{(i^{1+2\a+2p}+n)^2}\\
&\leq
\frac{\|v_0\|^2_{G^{\b}}}{n^3} +
n^{-\frac{1+4{(\a-\d)}}{1+2\a+2p}}h_n(\a; v_0).
\end{split}
\end{equation}
We note that the second term on the right hand side of the preceding display disappears for $\a>\log n/(2\log2)-1/2-p$. We have used again the fact that on the range $i \leq n^{1/(1+2\a+2p)}$, the quantity $i^{1+2\a+6p{+4\d}}(\log i)^{-1}$
is maximal for the largest $i$.
Now the function $x \mapsto -(1+2x)/(x+c)$ is  decreasing on $(0, \infty)$
for any $c > 1/2$. Moreover $h_n(\a; v_0) \leq L(\log n)^2$ for any $\a \leq \overline{\a}_{n,0}$, 
thus the preceding display is bounded above by a multiple of $n^{-(1+4{ (\underline{\a}_{n,0}-\d)})/(1+2\underline{\a}_{n,0}+2p)}(\log n)^2$. 
By Lemma~\ref{lem: abar}.(i) this
is further bounded by a constant times $n^{-(1+4{(\b-\d)})/(1+2\b+2p)}(\log n)^2$.

Next we consider the sum over the range $i > n^{1/(1+2\a+2p)}$. We distinguish two cases according to the value of $\a$. 
First suppose that ${ 2\a \geq 2p+4\d}$. Then $i^{-1-2\a+2p{+4\d}}(\log i)^{-1}$ is (easily) decreasing in $i$, hence
\begin{align*}
& \sum_{i>n^{1/(1+2\a+2p)}}\frac{ni^{2+4\a+ 6p{+4\d}}v_{0,i}^2}{(i^{1+2\a+2p}+n)^4} \\
&\quad \leq \frac{1}{n}\sum_{i>n^{1/(1+2\a+2p)}} \frac{n^2i^{-1-2\a+2p{+4\d}}(\log i)^{-1}i^{1+2\a}v_{0,i}^2\log i}{(i^{1+2\a+2p}+n)^2}\\
&\quad \leq \frac{1}{n^{(2+4{(\a-\d)})/(1+2\a+2p)}\log n} \sum_{i>n^{1/(1+2\a+2p)}}
\frac{n^2i^{1+2\a}v_{0,i}^2\log i}{(i^{1+2\a+2p}+n)^2}\\
&\quad \leq  n^{-\frac{1+4{(\a-\d)}}{1+2\a+2p}}h_n(\a;v_0).
\end{align*}
As above, this  is further bounded by a constant times $n^{-(1+4{ (\b-\d)})/(1+2\b+2p)}(\log n)^2$.
Next consider ${ 2\a < 2p+4\d}$.  { For $\a\ge\underline\a_{n,0}\ge \b-c_0/\log n$
and $\d<\b$, we have $2+4\a+2p+2\b>4\d$,  easily}, and hence
\begin{align*}
\sum_{i>n^{1/(1+2\a+2p)}} \frac{ni^{2+4\a+6p{+4\d}}v_{0,i}^2}{(i^{1+2\a+2p}+n)^4}
&\leq n\sum_{i>n^{1/(1+2\a+2p)}} i^{-2-4\a-2p-2\b{+4\d}}i^{2\b}v_{0,i}^2\\
&\leq \|v_0\|_{G^{\b}}^2n^{{\frac{-1-2\a-2\b+4\d}{1+2\a+2p}}}\\
&{ \leq \|v_0\|_{G^{\b}}^2n^{{\frac{-1-2\a-2(\b-c_0/\log n)+4\d}{1+2\a+2p}}}}.
\end{align*}
Since { $2p+4\d > 2\a>2(\beta-c_0/\log n)$, the function $\a\mapsto (1+2\a+2(\b-c_0/\log n)+4\d)/(1+2\a+2p)$ is increasing in $\a$}, 
and hence the right hand side of the preceding display attains its maximum at 
$\a=\underline{\a}_{n,0}\geq \b - c_0/\log n$. Since $n^{c_0/\log n}=e^{c_0}$, it follows that 
the preceding display is bounded above by 
\[
\|v_0\|_{G^{\b}}^2  e^{4c_0}n^{-\frac{1+4{(\b-\d)}}{1+2\b+2p}}.
\]
This concludes the proof that the variance of $\WW(\a)/\sqrt{n}$  is bounded above by a multiple of $r_n^2(\log n)^2$.

{ 
For $\a_1<\a_2$, we also have the easy bound
$$\var \Bigl(\frac{\WW(\a_1)-\WW(\a_2)}{\sqrt n}\Bigr)\le n\sum_{i=2}^\infty i^{-2-4\a_1-4p+4\d}v_{0,i}^2
\le n R^2 2^{-2\a_1},$$
for $\d<\b\le \a_1$. The right side is bounded above by $\e$ for $\a_1\ge 2 \log (n R^2/\e)$. By Lemma~\ref{lem: VarVW}
the square distance $d_n^2(\a_1,\a_2)$ on the left side of the display satisfies
$d_n(\a_1,\a_2)\lesssim |\a_1-\a_2| r_n (\log n)^2$.  This allows to bound 
the covering number of the interval $[\underline\a_{n,0},2\log (nR^2/\e)]$ by a multiple of
$1\vee r_n(\log n)^2\log (nR^2/\e)/\e$. Next we obtain
\begin{align*}\E \sup_{\underline\a_{n,0}\le \a\le \overline\a_{n,0}}\Bigl|\frac{\WW(\a)}{\sqrt n}\Bigr|
&\lesssim r_n\log n+\int_0^{r_n\log n}\sqrt {1\vee \frac{r_n(\log n)^2\log (nR^2/\e)}{\e}}\,d\e\\
&\lesssim r_n(\log n)^{3/2}\int_0^1\sqrt{\frac{1}{\eta}\log \frac{nR^2}{r_n\log n \eta}}\,d\eta
\lesssim r_n(\log n)^{2}\lesssim \e_n^2.
\end{align*}
}

\subsubsection{Bound for the expected and centered posterior risk over $Q_n$}\label{sec: 63}
To complete the proof of  Theorem~\ref{thm: ConvergenceEB} we show that similar results to Sections \ref{sec: 61} and \ref{sec: brisk} hold over the set $Q_n$ as well:
\begin{equation}\label{eq: Qn1}
\sup_{v_0\in Q_n} \sup_{\a\in[\underline{\a}_{n,0}, \infty)} \E_{v_0} R_n(\a) = O(\e_n^2),
\end{equation}
\begin{equation}\label{eq: Qn2}
\sup_{v_0\in Q_n}\E_{v_0} \sup_{\a\in[\underline{\a}_{n,0},\infty)}\Bigl|\sum_{i=1}^\infty \bigl(\hat{v}_{\a,i}-v_{0,i}\bigr)^2{ i^{2\d}} - \E_{v_0} \sum_{i=1}^\infty \bigl(\hat{v}_{\a,i}-v_{0,i}\bigr)^2{ i^{2\d}}\Bigr| = O(\e_n^2).
\end{equation}
For the first statement $\eqref{eq: Qn1}$ we follow the same line of reasoning as in Section
\ref{sec: 61}. The second and third terms in \eqref{eq: PostRisk} are free of $v_0$,
and hence the same upper bounds as in Section \ref{sec: 61} apply.
The first term in \eqref{eq:  PostRisk} is also treated exactly as in Section \ref{sec: 61},
except that $n^{1/(1+2\overline{\a}_{n,0}+2p)}\le 2$ if $v_0\in Q_n$ and hence
the sum over the terms $i< n^{1/(1+2\overline{\a}_{n,0}+2p)}$ need not be treated,
and we can proceed by replacing $\overline{\a}_{n,0}$ by $\log n/(2\log 2)-1/2-p$ in the definitions
of $J$ and the sequence $b_j$.

To bound the centered posterior risk $\eqref{eq: Qn2}$, we follow the proof given in Section~\ref{sec: brisk}. 
There the process $\VV(\a)$ is already bounded uniformly over $[\underline{\a}_{n,0},\infty)$,
whence it remains to deal with the process $\WW(\a)$. The only essential difference
is the upper bound for the variance of the process $\WW(\a)/\sqrt n$. In Section \ref{sec:
  brisk} this was shown to be  bounded above by a
multiple of the desired rate $(\log n)^2n^{-(1+4{ (\beta-\d)})/(1+2\beta+2p)}$
for $\alpha\in[\underline{\a}_{n,0},\overline{\a}_{n,0}\wedge (\log n/\log2-1/2-p)]$,
which is $\alpha\in[\underline{\a}_{n,0},\log n/\log2-1/2-p]$ on the set $Q_n$.
Finally, for $\alpha\geq \log n/\log 2-1/2-p$, and $-1-2\a-2\b+4\d<0$, we have
\begin{align}
&\sum_{i=1}^{\infty}\frac{ni^{2+4\a+6p+{ 4\d}}v_{0,i}^2}{(i^{1+2\a+2p}+n)^4}
\leq \frac{v_{0,1}^2}{n^3}+\sum_{i=2}^{\infty}\frac{ni^{-1-2\a{ +4\d}}v_{0,i}^2}{i^{1+2\a+2p}+n}\label{eq: Var2}\\
&\qquad\leq \frac{\|v_{0}\|_{\beta}^2}{n^3}+\sum_{i=2}^{\infty}i^{-1-2\a-2\beta{ +4\d}}i^{2\beta}v_{0,i}^2\nonumber\\
&\qquad\leq \frac{\|v_{0}\|_{\beta}^2}{n^3}+2^{-1-2\a-2\b{ +4\d}}\|v_0\|_{\beta}^2
\leq \frac{\|v_{0}\|_{\beta}^2}{n^3}+2^{2p{+4\d}}\frac{\|v_{0}\|_{\beta}^2}{n^2}\nonumber
\lesssim \frac1{n^2}.
\end{align}

\subsubsection{Bounds for the semimetrics associated to  $\VV$ and $\WW$}

\begin{lemma}\label{lem: VarVW}
{ Assume $1+2\underline\a_{n,0}+6p+4\d>-1$.}
For any $\underline{\a}_{n,0} \leq \a_1 < \a_2\leq\overline{\a}_{n,0}$ the following inequalities hold:
\[
\var_{v_0}\bigl(\VV(\a_1)-\VV(\a_2)\bigr) \lle (\a_1-\a_2)^2n^{(1+4p+{4\d})/(1+2\underline{\a}_{n,0}+2p)}(\log n)^2,
\]
\[
\var_{v_0}\biggl(\frac{\WW(\a_1)}{\sqrt{n}}-\frac{\WW(\a_2)}{\sqrt{n}}\biggr)
\lesssim (\a_1-\a_2)^2 n^{-\frac{1+4{ (\overline{\a}_{n,0}-\d)}}{1+2\overline{\a}_{n,0}+2p}}
(\log n)^4,
\]
with a constant that does not depend on $\a$ and $v_0$.
\end{lemma}

\begin{proof}
The left-hand side of the first inequality is equal to, for $f_i(\a) = (i^{1+2\a+2p}+n)^{-2}$,
\[
n^4\sum_{i=1}^\infty (f_i(\a_1)-f_i(\a_2))^2i^{4p{+4\d}}\var Z_i^2,
\]
The derivative of $f_i$ is given by  $f_i'(\a) = -4i^{1+2\a+2p}(\log i)/(i^{1+2\a+2p}+n)^{3}$, 
whence the preceding display is bounded above by a multiple of
\begin{align*}
& (\a_1-\a_2)^2n^4\sup_{\a\in[\a_1, \a_2]} \sum_{i=1}^\infty \frac{i^{2+4\a+8p{+4\d}}(\log i)^2}{(i^{1+2\a+2p}+n)^6}\\
&\qquad\leq (\a_1-\a_2)^2n^3(\log n)^2\sup_{\a\in[\a_1, \a_2]} \frac{1}{(1+2\a+2p)^2}\sum_{i=1}^\infty \frac{i^{1+2\a+6p{+4\d}}}{(i^{1+2\a+2p}+n)^4}\\
&\qquad\lle (\a_1-\a_2)^2(\log n)^2 \sup_{\a\in[\a_1, \a_2]}
n^{(1+4p{+4\d})/(1+2\a+2p)},
\end{align*}
with the help of Lemma~\ref{lem: Botond'sTrick} (with $r=1+2\a+2p$, and $m=2$), and Lemma~\ref{lem: LogM} (with $m=0$, $l = 4$, $r=1+2\a+2p$, and $s={ 1+2\a+6p+4\d}$). Since $\a \geq \underline{\a}_{n,0}$, we get the first assertion of the lemma.

We next consider $\WW/\sqrt{n}$. The left-hand side of the second inequality in the statement of the lemma is equal to
\[
\sum_{i=1}^\infty (f_i(\a_1) - f_i(\a_2))^2nv_{0,i}^2{ i^{4\d}}\var Z_i,
\]
where now $f_i(\a)=i^{1+2\a+3p}/(i^{1+2\a+2p}+n)^2$.
The derivative of this $f_i$ satisfies
$|f_i'(\a)| \le 2(\log i)f_i(\a)$, hence we get the upper bound
\[
4(\a_2-\a_1)^2 \sup_{\a \in [\a_1, \a_2]} \sum_{i=1}^\infty
\frac{n i^{2+4\a+6p}v^2_{0,i}{ i^{4\d}}\log^2 i}{(i^{1+2\a+2p}+n)^4}.
\]
The proof is completed by arguing as in (\ref{eq: piet}) or (\ref{eq: Var2}), { where we may
use that the function $x\mapsto (\log x)/x$ is decreasing for $x>e$ to take care of the
extra $(\log i)^2$ factor.}
\end{proof}

\subsection{Proof of Theorem~\ref{thm: ConvergenceHB}}\label{sec: ProofHB}
Let $A_n$ be the event that $\hat\a_n \in [\underline{\a}_{n,0},\overline{\a}_{n,0}]$.
Then with $\a \mapsto \lambda_n(\a \given Y_n)$ denoting the posterior Lebesgue density of $\a$, we have
{\begin{align*}
& \sup_{\|v_0\|_{G^{\b}}\le R}\E_{v_0} \Pi(\|v-v_0\|_{G^\d}\geq M_n\e_n|Y_n)\\
&\qquad\leq \sup_{\|v_0\|_{G^{\b}}\le R} \Pr_0(A_n^c) + \sup_{\|v_0\|_{G^{\b}}\le R}\E_{v_0}\int_0^{\underline{\a}_{n,0}}\l_n(\a|Y_n)\, d\a\, 1_{A_n}\\
&\qquad\qquad +\sup_{v_0 \in Q_n} \E_{v_0} \sup_{\a\in[\underline{\a}_{n,0},\infty)} \Pi_\a(\|v-v_0\|_{G^\d}\geq M_n \e_{n}|Y_n)\\
&\qquad\qquad +\sup_{v_0 \in P_n}  \E_{v_0}\sup_{\a\in[\underline{\a}_{n,0},\overline{\a}_{n,0}]} \Pi_\a(\|v-v_0\|_{G^\d}\geq M_n
 \e_n|Y_n)\\
&\qquad\qquad\qquad+\sup_{v_0\in P_n}\E_{v_0}\int_{\overline{\a}_{n,0}}^\infty\l_n(\a|Y_n)\, d\a.
\end{align*}
The first, second and fifth terms on the right are shown to tend to zero in \cite{Knapik2016}. The third and fourth terms
are shown to tend to zero in the preceding.

\subsection{Auxiliary lemmas}\label{sec: Appendix}

\begin{lemma}[Lemma 8 of \cite{Knapik2016}]\label{lem: LogM}
For any $m>0$, $l \geq 1$, $r_0>0$, $r \in (0, r_0]$, $s \in (0, rl-2]$, and $n \geq e^{2mr_0}$
\[
\sum_{i=1}^\infty \frac{i^{s}(\log i)^m}{(i^{r}+n)^l} \leq 4n^{(1+s-lr)/r}\frac{(\log n)^m}{r^m}.
\]
The same upper bound holds for $m = 0$, $r \in (0, \infty)$, $s \in ({ -1}, rl-1)$, and $n \geq 1$.
{ For $m=0$ and $s=-1$, the rate is $(\log n)n^{-l}$ and for any $s<-1$ it becomes $n^{-l}$.}
\end{lemma}

\begin{lemma}[Lemma 10 of \cite{Knapik2016}]\label{lem: Botond'sTrick}
Let $m$, $i$, $r$, and $\xi$ be positive reals. Then for $n \geq e^m$
\[
\frac{ni^{r}\bigl(r\log i\bigr)^m}{(i^{r}+n)^2}\leq (\log n)^m, \qquad \text{and} \qquad \frac{n^\xi\bigl(r\log i\bigr)^{\xi m}}{(i^{r}+n)^\xi}\leq (\log n)^{\xi m}.
\]
\end{lemma}

\bibliographystyle{imsart-number} 
\bibliography{bibliography}       

\begin{thebibliography}{61}

\bibitem{Adams}
\begin{bbook}[author]
\bauthor{\bsnm{Adams},~\bfnm{Robert~A.}\binits{R.~A.}} \AND
  \bauthor{\bsnm{Fournier},~\bfnm{John J.~F.}\binits{J.~J.~F.}}
(\byear{2003}).
\btitle{Sobolev spaces},
\bedition{second} ed.
\bseries{Pure and Applied Mathematics (Amsterdam)}
\bvolume{140}.
\bpublisher{Elsevier/Academic Press, Amsterdam}.
\bmrnumber{2424078}
\end{bbook}
\endbibitem

\bibitem{agapiou:2013}
\begin{barticle}[author]
\bauthor{\bsnm{Agapiou},~\bfnm{Sergios}\binits{S.}},
  \bauthor{\bsnm{Larsson},~\bfnm{Stig}\binits{S.}} \AND
  \bauthor{\bsnm{Stuart},~\bfnm{Andrew~M}\binits{A.~M.}}
(\byear{2013}).
\btitle{Posterior contraction rates for the Bayesian approach to linear
  ill-posed inverse problems}.
\bjournal{Stochastic Processes and their Applications}
\bvolume{123}
\bpages{3828--3860}.
\end{barticle}
\endbibitem

\bibitem{agapiou:savva:2024}
\begin{barticle}[author]
\bauthor{\bsnm{Agapiou},~\bfnm{Sergios}\binits{S.}} \AND
  \bauthor{\bsnm{Savva},~\bfnm{Aimilia}\binits{A.}}
(\byear{2024}).
\btitle{{Adaptive inference over Besov spaces in the white noise model using
  p-exponential priors}}.
\bjournal{Bernoulli}
\bvolume{30}
\bpages{2275 -- 2300}.
\bdoi{10.3150/23-BEJ1673}
\end{barticle}
\endbibitem

\bibitem{AlbertiCapdeboscq}
\begin{bbook}[author]
\bauthor{\bsnm{Alberti},~\bfnm{Giovanni~S}\binits{G.~S.}} \AND
  \bauthor{\bsnm{Capdeboscq},~\bfnm{Yves}\binits{Y.}}
(\byear{2018}).
\btitle{Lectures on elliptic methods for hybrid inverse problems}
\bvolume{25}.
\bpublisher{Soci{\'e}t{\'e} Math{\'e}matique de France, Paris}.
\end{bbook}
\endbibitem

\bibitem{BerghLofstrom}
\begin{bbook}[author]
\bauthor{\bsnm{Bergh},~\bfnm{J\"{o}ran}\binits{J.}} \AND
  \bauthor{\bsnm{L\"{o}fstr\"{o}m},~\bfnm{J\"{o}rgen}\binits{J.}}
(\byear{1976}).
\btitle{Interpolation spaces. {A}n introduction}.
\bseries{Grundlehren der Mathematischen Wissenschaften}
\bvolume{No. 223}.
\bpublisher{Springer-Verlag, Berlin-New York}.
\bmrnumber{482275}
\end{bbook}
\endbibitem

\bibitem{cai2006adaptive}
\begin{barticle}[author]
\bauthor{\bsnm{Cai},~\bfnm{T.~Tony}\binits{T.~T.}} \AND
  \bauthor{\bsnm{Low},~\bfnm{Mark~G.}\binits{M.~G.}}
(\byear{2004}).
\btitle{An adaptation theory for nonparametric confidence intervals}.
\bjournal{Ann. Statist.}
\bvolume{32}
\bpages{1805--1840}.
\bdoi{10.1214/009053604000000049}
\bmrnumber{2102494}
\end{barticle}
\endbibitem

\bibitem{Castillo2012}
\begin{barticle}[author]
\bauthor{\bsnm{Castillo},~\bfnm{Isma\"{e}l}\binits{I.}}
(\byear{2012}).
\btitle{A semiparametric {B}ernstein--von {M}ises theorem for {G}aussian
  process priors}.
\bjournal{Probab. Theory Related Fields}
\bvolume{152}
\bpages{53--99}.
\bdoi{10.1007/s00440-010-0316-5}
\bmrnumber{2875753}
\end{barticle}
\endbibitem

\bibitem{CastilloNickl2014}
\begin{barticle}[author]
\bauthor{\bsnm{Castillo},~\bfnm{Isma\"{e}l}\binits{I.}} \AND
  \bauthor{\bsnm{Nickl},~\bfnm{Richard}\binits{R.}}
(\byear{2014}).
\btitle{On the {B}ernstein-von {M}ises phenomenon for nonparametric {B}ayes
  procedures}.
\bjournal{Ann. Statist.}
\bvolume{42}
\bpages{1941--1969}.
\bdoi{10.1214/14-AOS1246}
\bmrnumber{3262473}
\end{barticle}
\endbibitem

\bibitem{cavalier2008nonparametric}
\begin{bincollection}[author]
\bauthor{\bsnm{Cavalier},~\bfnm{Laurent}\binits{L.}}
(\byear{2011}).
\btitle{Inverse problems in statistics}.
In \bbooktitle{Inverse problems and high-dimensional estimation}.
\bseries{Lect. Notes Stat. Proc.}
\bvolume{203}
\bpages{3--96}.
\bpublisher{Springer, Heidelberg}.
\bdoi{10.1007/978-3-642-19989-9\_1}
\bmrnumber{2868199}
\end{bincollection}
\endbibitem

\bibitem{moumita}
\begin{barticle}[author]
\bauthor{\bsnm{Chakraborty},~\bfnm{Moumita}\binits{M.}} \AND
  \bauthor{\bsnm{Ghosal},~\bfnm{Subhashis}\binits{S.}}
(\byear{2022}).
\btitle{Rates and coverage for monotone densities using projection-posterior}.
\bjournal{Bernoulli}
\bvolume{28}
\bpages{1093--1119}.
\bdoi{10.3150/21-bej1379}
\bmrnumber{4388931}
\end{barticle}
\endbibitem

\bibitem{Hitchhiker}
\begin{barticle}[author]
\bauthor{\bsnm{Di~Nezza},~\bfnm{Eleonora}\binits{E.}},
  \bauthor{\bsnm{Palatucci},~\bfnm{Giampiero}\binits{G.}} \AND
  \bauthor{\bsnm{Valdinoci},~\bfnm{Enrico}\binits{E.}}
(\byear{2012}).
\btitle{Hitchhiker's guide to the fractional {S}obolev spaces}.
\bjournal{Bull. Sci. Math.}
\bvolume{136}
\bpages{521--573}.
\bdoi{10.1016/j.bulsci.2011.12.004}
\bmrnumber{2944369}
\end{barticle}
\endbibitem

\bibitem{Doob}
\begin{barticle}[author]
\bauthor{\bsnm{Doob},~\bfnm{J.~L.}\binits{J.~L.}}
(\byear{1955}).
\btitle{A probability approach to the heat equation}.
\bjournal{Trans. Amer. Math. Soc.}
\bvolume{80}
\bpages{216--280}.
\bdoi{10.2307/1993013}
\bmrnumber{79376}
\end{barticle}
\endbibitem

\bibitem{Engletal}
\begin{bbook}[author]
\bauthor{\bsnm{Engl},~\bfnm{Heinz~W.}\binits{H.~W.}},
  \bauthor{\bsnm{Hanke},~\bfnm{Martin}\binits{M.}} \AND
  \bauthor{\bsnm{Neubauer},~\bfnm{Andreas}\binits{A.}}
(\byear{1996}).
\btitle{Regularization of inverse problems}.
\bseries{Mathematics and its Applications}
\bvolume{375}.
\bpublisher{Kluwer Academic Publishers Group, Dordrecht}.
\bmrnumber{1408680}
\end{bbook}
\endbibitem

\bibitem{Evans10}
\begin{bbook}[author]
\bauthor{\bsnm{Evans},~\bfnm{Lawrence~C.}\binits{L.~C.}}
(\byear{2010}).
\btitle{Partial differential equations},
\bedition{second} ed.
\bseries{Graduate Studies in Mathematics}
\bvolume{19}.
\bpublisher{American Mathematical Society, Providence, RI}.
\bdoi{10.1090/gsm/019}
\bmrnumber{2597943}
\end{bbook}
\endbibitem

\bibitem{GGvdV2000}
\begin{barticle}[author]
\bauthor{\bsnm{Ghosal},~\bfnm{Subhashis}\binits{S.}},
  \bauthor{\bsnm{Ghosh},~\bfnm{Jayanta~K.}\binits{J.~K.}} \AND
  \bauthor{\bparticle{van~der} \bsnm{Vaart},~\bfnm{Aad~W.}\binits{A.~W.}}
(\byear{2000}).
\btitle{Convergence rates of posterior distributions}.
\bjournal{Ann. Statist.}
\bvolume{28}
\bpages{500--531}.
\bdoi{10.1214/aos/1016218228}
\bmrnumber{1790007}
\end{barticle}
\endbibitem

\bibitem{GhosalvdVbook2017}
\begin{bbook}[author]
\bauthor{\bsnm{Ghosal},~\bfnm{Subhashis}\binits{S.}} \AND
  \bauthor{\bparticle{van~der} \bsnm{Vaart},~\bfnm{A.~W.}\binits{A.~W.}}
(\byear{2017}).
\btitle{Fundamentals of Nonparametric Bayesian Inference}.
\bseries{Cambridge Series in Statistical and Probabilistic Mathematics}.
\bpublisher{Cambridge University Press, Cambridge}.
\end{bbook}
\endbibitem

\bibitem{Gilbarg2015}
\begin{bbook}[author]
\bauthor{\bsnm{Gilbarg},~\bfnm{David}\binits{D.}} \AND
  \bauthor{\bsnm{Trudinger},~\bfnm{Neil~S}\binits{N.~S.}}
(\byear{2015}).
\btitle{Elliptic Partial Differential Equations of Second Order}.
\bseries{Grundlehren der mathematischen Wissenschaften}
\bvolume{224}.
\bpublisher{Springer Berlin / Heidelberg}, \baddress{Berlin, Heidelberg}.
\end{bbook}
\endbibitem

\bibitem{GiordanoNickl2020}
\begin{barticle}[author]
\bauthor{\bsnm{Giordano},~\bfnm{Matteo}\binits{M.}} \AND
  \bauthor{\bsnm{Nickl},~\bfnm{Richard}\binits{R.}}
(\byear{2020}).
\btitle{Consistency of {B}ayesian inference with {G}aussian process priors in
  an elliptic inverse problem}.
\bjournal{Inverse Problems}
\bvolume{36}
\bpages{085001, 35}.
\bdoi{10.1088/1361-6420/ab7d2a}
\bmrnumber{4151406}
\end{barticle}
\endbibitem

\bibitem{Yan2020}
\begin{barticle}[author]
\bauthor{\bsnm{Gugushvili},~\bfnm{Shota}\binits{S.}},
  \bauthor{\bparticle{van~der} \bsnm{Vaart},~\bfnm{Aad}\binits{A.}} \AND
  \bauthor{\bsnm{Yan},~\bfnm{Dong}\binits{D.}}
(\byear{2020}).
\btitle{{Bayesian linear inverse problems in regularity scales}}.
\bjournal{Annales de l'Institut Henri Poincaré, Probabilités et Statistiques}
\bvolume{56}
\bpages{2081 -- 2107}.
\bdoi{10.1214/19-AIHP1029}
\end{barticle}
\endbibitem

\bibitem{Halmos}
\begin{bbook}[author]
\bauthor{\bsnm{Halmos},~\bfnm{Paul~R.}\binits{P.~R.}}
(\byear{1967}).
\btitle{A {H}ilbert space problem book}.
\bpublisher{D. Van Nostrand Co., Inc., Princeton, N.J.-Toronto, Ont.-London}.
\bmrnumber{0208368 (34 \#\#8178)}
\end{bbook}
\endbibitem

\bibitem{Ivrii2016}
\begin{barticle}[author]
\bauthor{\bsnm{Ivrii},~\bfnm{Victor}\binits{V.}}
(\byear{2016}).
\btitle{100 years of Weyl’s law}.
\bjournal{Bull. Math. Sci.}
\bvolume{6}
\bpages{379--452}.
\end{barticle}
\endbibitem

\bibitem{Lions1972}
\begin{bbook}[author]
\bauthor{\bsnm{J.~L.~Lions},~\bfnm{E.}\binits{E.} \bsuffix{Magenes}}
(\byear{1972}).
\btitle{Non-Homogeneous Boundary Value Problems and Applications : Volume II}.
\bseries{Die Grundlehren der mathematischen Wissenschaften, in
  Einzeldarstellungen mit besonderer Ber\"u cksichtigung der Anwendungsgebiete,
  182}.
\bpublisher{Springer Berlin Heidelberg}, \baddress{Berlin, Heidelberg}.
\end{bbook}
\endbibitem

\bibitem{kekkonen2022consistency}
\begin{barticle}[author]
\bauthor{\bsnm{Kekkonen},~\bfnm{Hanne}\binits{H.}}
(\byear{2022}).
\btitle{Consistency of Bayesian inference with Gaussian process priors for a
  parabolic inverse problem}.
\bjournal{Inverse Problems}
\bvolume{38}
\bpages{035002}.
\end{barticle}
\endbibitem

\bibitem{Kinderlehrer00}
\begin{bbook}[author]
\bauthor{\bsnm{Kinderlehrer},~\bfnm{David}\binits{D.}} \AND
  \bauthor{\bsnm{Stampacchia},~\bfnm{Guido}\binits{G.}}
(\byear{2000}).
\btitle{An Introduction to Variational Inequalities and Their Applications},
\bedition{first} ed.
\bseries{Classics in Applied Mathematics}
\bvolume{31}.
\bpublisher{Scoiety for Industrial and Applied Mathematics}.
\end{bbook}
\endbibitem

\bibitem{Knapik2016}
\begin{barticle}[author]
\bauthor{\bsnm{Knapik},~\bfnm{Bartek}\binits{B.}},
  \bauthor{\bsnm{Szabo},~\bfnm{Botond}\binits{B.}},
  \bauthor{\bparticle{van~der} \bsnm{Vaart},~\bfnm{Aad}\binits{A.}} \AND
  \bauthor{\bparticle{van} \bsnm{Zanten},~\bfnm{Harry}\binits{H.}}
(\byear{2016}).
\btitle{Bayes procedures for adaptive inference in inverse problems for the
  white noise model}.
\bjournal{Probability Theory and Related Fields}
\bvolume{164}
\bpages{771-813}.
\bdoi{10.1007/s00440-015-0619-7}
\end{barticle}
\endbibitem

\bibitem{Knapik2011}
\begin{barticle}[author]
\bauthor{\bsnm{Knapik},~\bfnm{Bartek}\binits{B.}}, \bauthor{\bparticle{van~der}
  \bsnm{Vaart},~\bfnm{Aad}\binits{A.}} \AND \bauthor{\bparticle{van}
  \bsnm{Zanten},~\bfnm{Harry}\binits{H.}}
(\byear{2011}).
\btitle{Bayesian Inverse Problems with Gaussian Priors}.
\bjournal{The Annals of Statistics}
\bvolume{39}
\bpages{2626-2657}.
\bdoi{10.1214/11-AOS920}
\end{barticle}
\endbibitem

\bibitem{KnapikHeat}
\begin{barticle}[author]
\bauthor{\bsnm{Knapik},~\bfnm{B.~T.}\binits{B.~T.}},
  \bauthor{\bparticle{van~der} \bsnm{Vaart},~\bfnm{A.~W.}\binits{A.~W.}} \AND
  \bauthor{\bparticle{van} \bsnm{Zanten},~\bfnm{J.~H.}\binits{J.~H.}}
(\byear{2013}).
\btitle{Bayesian recovery of the initial condition for the heat equation}.
\bjournal{Comm. Statist. Theory Methods}
\bvolume{42}
\bpages{1294--1313}.
\bdoi{10.1080/03610926.2012.681417}
\bmrnumber{3031282}
\end{barticle}
\endbibitem

\bibitem{koers2023misspecified}
\begin{bmisc}[author]
\bauthor{\bsnm{Koers},~\bfnm{Geerten}\binits{G.}},
  \bauthor{\bsnm{Szabó},~\bfnm{Botond}\binits{B.}} \AND
  \bauthor{\bparticle{van~der} \bsnm{Vaart},~\bfnm{Aad}\binits{A.}}
(\byear{2023}).
\btitle{Misspecified Bernstein-Von Mises theorem for hierarchical models}.
\end{bmisc}
\endbibitem

\bibitem{LedouxTalagrand}
\begin{bbook}[author]
\bauthor{\bsnm{Ledoux},~\bfnm{M.}\binits{M.}} \AND
  \bauthor{\bsnm{Talagrand},~\bfnm{M.}\binits{M.}}
(\byear{1991}).
\btitle{Probability in {B}anach spaces}
\bvolume{23}.
\bpublisher{Springer-Verlag}, \baddress{Berlin}.
\bmrnumber{MR1102015 (93c:60001)}
\end{bbook}
\endbibitem

\bibitem{Monardetal2021}
\begin{barticle}[author]
\bauthor{\bsnm{Monard},~\bfnm{Fran\c{c}ois}\binits{F.}},
  \bauthor{\bsnm{Nickl},~\bfnm{Richard}\binits{R.}} \AND
  \bauthor{\bsnm{Paternain},~\bfnm{Gabriel~P.}\binits{G.~P.}}
(\byear{2021}).
\btitle{Statistical guarantees for {B}ayesian uncertainty quantification in
  nonlinear inverse problems with {G}aussian process priors}.
\bjournal{Ann. Statist.}
\bvolume{49}
\bpages{3255--3298}.
\bdoi{10.1214/21-aos2082}
\bmrnumber{4352530}
\end{barticle}
\endbibitem

\bibitem{MR4230066}
\begin{barticle}[author]
\bauthor{\bsnm{Monard},~\bfnm{Fran\c{c}ois}\binits{F.}},
  \bauthor{\bsnm{Nickl},~\bfnm{Richard}\binits{R.}} \AND
  \bauthor{\bsnm{Paternain},~\bfnm{Gabriel~P.}\binits{G.~P.}}
(\byear{2021}).
\btitle{Consistent inversion of noisy non-{A}belian {X}-ray transforms}.
\bjournal{Comm. Pure Appl. Math.}
\bvolume{74}
\bpages{1045--1099}.
\bdoi{10.1002/cpa.21942}
\bmrnumber{4230066}
\end{barticle}
\endbibitem

\bibitem{Nickl18}
\begin{barticle}[author]
\bauthor{\bsnm{Nickl},~\bfnm{R.}\binits{R.}}
(\byear{2018}).
\btitle{Bernstein-von Mises Theorems for statistical inverse problems I:
  Schr\"odinger Equation}.
\bjournal{Journal of the European Mathematical Society}
\bvolume{22}
\bpages{2697-2750}.
\bdoi{https://doi.org/10.4171/JEMS/975}
\end{barticle}
\endbibitem

\bibitem{Nickl23}
\begin{bbook}[author]
\bauthor{\bsnm{Nickl},~\bfnm{R.}\binits{R.}}
(\byear{2023}).
\btitle{Bayesian Non-linear Inverse Problems}.
\bseries{Z\"urich Lectures in Advanced Mathematics}.
\bpublisher{ETH, Z\"urich}.
\end{bbook}
\endbibitem

\bibitem{nickl2020convergence}
\begin{barticle}[author]
\bauthor{\bsnm{Nickl},~\bfnm{Richard}\binits{R.}}, \bauthor{\bparticle{van~de}
  \bsnm{Geer},~\bfnm{Sara}\binits{S.}} \AND
  \bauthor{\bsnm{Wang},~\bfnm{Sven}\binits{S.}}
(\byear{2020}).
\btitle{Convergence rates for penalized least squares estimators in PDE
  constrained regression problems}.
\bjournal{SIAM/ASA Journal on Uncertainty Quantification}
\bvolume{8}
\bpages{374--413}.
\end{barticle}
\endbibitem

\bibitem{nickl:wang:2020}
\begin{bmisc}[author]
\bauthor{\bsnm{Nickl},~\bfnm{Richard}\binits{R.}} \AND
  \bauthor{\bsnm{Wang},~\bfnm{Sven}\binits{S.}}
(\byear{2020}).
\btitle{On polynomial-time computation of high-dimensional posterior measures
  by Langevin-type algorithms}.
\bdoi{10.48550/ARXIV.2009.05298}
\end{bmisc}
\endbibitem

\bibitem{nieman2023uncertainty}
\begin{barticle}[author]
\bauthor{\bsnm{Nieman},~\bfnm{Dennis}\binits{D.}},
  \bauthor{\bsnm{Szabo},~\bfnm{Botond}\binits{B.}} \AND
  \bauthor{\bparticle{van} \bsnm{Zanten},~\bfnm{Harry}\binits{H.}}
(\byear{2023}).
\btitle{Uncertainty quantification for sparse spectral variational
  approximations in Gaussian process regression}.
\bjournal{Electronic Journal of Statistics}
\bvolume{17}
\bpages{2250--2288}.
\end{barticle}
\endbibitem

\bibitem{randrianarisoa2023variational}
\begin{barticle}[author]
\bauthor{\bsnm{Randrianarisoa},~\bfnm{Thibault}\binits{T.}} \AND
  \bauthor{\bsnm{Szabo},~\bfnm{Botond}\binits{B.}}
(\byear{2023}).
\btitle{Variational Gaussian processes for linear inverse problems}.
\bjournal{Advances in Neural Information Processing Systems}
\bvolume{36}
\bpages{28960--28972}.
\end{barticle}
\endbibitem

\bibitem{Ray2013}
\begin{barticle}[author]
\bauthor{\bsnm{Ray},~\bfnm{Kolyan}\binits{K.}}
(\byear{2013}).
\btitle{Bayesian inverse problems with non-conjugate priors}.
\bjournal{Electron. J. Statist.}
\bvolume{7}
\bpages{2516--2549}.
\bdoi{10.1214/13-EJS851}
\end{barticle}
\endbibitem

\bibitem{Ray2017}
\begin{barticle}[author]
\bauthor{\bsnm{Ray},~\bfnm{Kolyan}\binits{K.}}
(\byear{2017}).
\btitle{Adaptive {B}ernstein--von {M}ises theorems in {G}aussian white noise}.
\bjournal{Ann. Statist.}
\bvolume{45}
\bpages{2511--2536}.
\bdoi{10.1214/16-AOS1533}
\bmrnumber{3737900}
\end{barticle}
\endbibitem

\bibitem{Reiss2008}
\begin{barticle}[author]
\bauthor{\bsnm{Reiss},~\bfnm{Markus}\binits{M.}}
(\byear{2008}).
\btitle{Asymptotic Equivalence for Nonparametric Regression with Multivariate
  and Random Design}.
\bjournal{The Annals of statistics}
\bvolume{36}
\bpages{1957--1982}.
\end{barticle}
\endbibitem

\bibitem{Richter}
\begin{barticle}[author]
\bauthor{\bsnm{Richter},~\bfnm{Gerard~R.}\binits{G.~R.}}
(\byear{1981}).
\btitle{An inverse problem for the steady state diffusion equation}.
\bjournal{SIAM J. Appl. Math.}
\bvolume{41}
\bpages{210--221}.
\bdoi{10.1137/0141016}
\bmrnumber{628945}
\end{barticle}
\endbibitem

\bibitem{RichterNumerical}
\begin{barticle}[author]
\bauthor{\bsnm{Richter},~\bfnm{Gerard~R.}\binits{G.~R.}}
(\byear{1981}).
\btitle{Numerical identification of a spatially varying diffusion coefficient}.
\bjournal{Math. Comp.}
\bvolume{36}
\bpages{375--386}.
\bdoi{10.2307/2007648}
\bmrnumber{606502}
\end{barticle}
\endbibitem

\bibitem{robins2006adaptive}
\begin{barticle}[author]
\bauthor{\bsnm{Robins},~\bfnm{James}\binits{J.}} \AND \bauthor{\bsnm{Van
  Der~Vaart},~\bfnm{Aad}\binits{A.}}
(\byear{2006}).
\btitle{Adaptive nonparametric confidence sets}.
\bjournal{The Annals of Statistics}
\bvolume{34}
\bpages{229--253}.
\end{barticle}
\endbibitem

\bibitem{rousseau2020asymptotic}
\begin{barticle}[author]
\bauthor{\bsnm{Rousseau},~\bfnm{Judith}\binits{J.}} \AND
  \bauthor{\bsnm{Szabo},~\bfnm{Botond}\binits{B.}}
(\byear{2020}).
\btitle{Asymptotic frequentist coverage properties of Bayesian credible sets
  for sieve priors}.
\bjournal{The Annals of Statistics}
\bvolume{48}
\bpages{2155--2179}.
\end{barticle}
\endbibitem

\bibitem{rousseau2017asymptotic}
\begin{barticle}[author]
\bauthor{\bsnm{Rousseau},~\bfnm{Judith}\binits{J.}},
  \bauthor{\bsnm{Szabo},~\bfnm{Botond}\binits{B.}} \betal{et~al.}
(\byear{2017}).
\btitle{Asymptotic behaviour of the empirical bayes posteriors associated to
  maximum marginal likelihood estimator}.
\bjournal{The Annals of Statistics}
\bvolume{45}
\bpages{833--865}.
\end{barticle}
\endbibitem

\bibitem{Seeley65}
\begin{barticle}[author]
\bauthor{\bsnm{Seeley},~\bfnm{R.~T.}\binits{R.~T.}}
(\byear{1965}).
\btitle{Integro-Differential Operators on Vector Bundles}.
\bjournal{Transactions of the American Mathematical Society}
\bvolume{117}
\bpages{167-204}.
\bdoi{10.1090/S0002-9947-1965-0173174-1}
\end{barticle}
\endbibitem

\bibitem{SniekersvdV2015}
\begin{barticle}[author]
\bauthor{\bsnm{Sniekers},~\bfnm{Suzanne}\binits{S.}} \AND
  \bauthor{\bparticle{van~der} \bsnm{Vaart},~\bfnm{Aad}\binits{A.}}
(\byear{2015}).
\btitle{Adaptive {B}ayesian credible sets in regression with a {G}aussian
  process prior}.
\bjournal{Electron. J. Stat.}
\bvolume{9}
\bpages{2475--2527}.
\bdoi{10.1214/15-EJS1078}
\bmrnumber{3425364}
\end{barticle}
\endbibitem

\bibitem{SniekersvdV2020}
\begin{barticle}[author]
\bauthor{\bsnm{Sniekers},~\bfnm{Suzanne}\binits{S.}} \AND
  \bauthor{\bparticle{van~der} \bsnm{Vaart},~\bfnm{Aad}\binits{A.}}
(\byear{2020}).
\btitle{Adaptive {B}ayesian credible bands in regression with a {G}aussian
  process prior}.
\bjournal{Sankhya A}
\bvolume{82}
\bpages{386--425}.
\bdoi{10.1007/s13171-019-00185-0}
\bmrnumber{4136240}
\end{barticle}
\endbibitem

\bibitem{szabo2023adaptation}
\begin{barticle}[author]
\bauthor{\bsnm{Szabo},~\bfnm{Botond}\binits{B.}},
  \bauthor{\bsnm{Hadji},~\bfnm{Amine}\binits{A.}} \AND
  \bauthor{\bparticle{van~der} \bsnm{Vaart},~\bfnm{Aad}\binits{A.}}
(\byear{2023}).
\btitle{Adaptation using spatially distributed Gaussian Processes}.
\bjournal{arXiv preprint arXiv:2312.14130}.
\end{barticle}
\endbibitem

\bibitem{SzabovdVvZ2015}
\begin{barticle}[author]
\bauthor{\bsnm{Szab\'{o}},~\bfnm{Botond}\binits{B.}},
  \bauthor{\bparticle{van~der} \bsnm{Vaart},~\bfnm{A.~W.}\binits{A.~W.}} \AND
  \bauthor{\bparticle{van} \bsnm{Zanten},~\bfnm{J.~H.}\binits{J.~H.}}
(\byear{2015}).
\btitle{Frequentist coverage of adaptive nonparametric {B}ayesian credible
  sets}.
\bjournal{Ann. Statist.}
\bvolume{43}
\bpages{1391--1428}.
\bdoi{10.1214/14-AOS1270}
\bmrnumber{3357861}
\end{barticle}
\endbibitem

\bibitem{Szabo2013}
\begin{barticle}[author]
\bauthor{\bsnm{Szab\'{o}},~\bfnm{B.~T.}\binits{B.~T.}},
  \bauthor{\bparticle{van~der} \bsnm{Vaart},~\bfnm{A.~W.}\binits{A.~W.}} \AND
  \bauthor{\bparticle{van} \bsnm{Zanten},~\bfnm{J.~H.}\binits{J.~H.}}
(\byear{2013}).
\btitle{Empirical {B}ayes scaling of {G}aussian priors in the white noise
  model}.
\bjournal{Electron. J. Stat.}
\bvolume{7}
\bpages{991--1018}.
\bdoi{10.1214/13-EJS798}
\bmrnumber{3044507}
\end{barticle}
\endbibitem

\bibitem{Taylor2011}
\begin{bbook}[author]
\bauthor{\bsnm{Taylor},~\bfnm{Michael~E.}\binits{M.~E.}}
(\byear{2011}).
\btitle{Partial Differential Equations I: Basic Theory}.
\bseries{Applied Mathematical Sciences}.
\bpublisher{Springer-Verlag, New York}.
\end{bbook}
\endbibitem

\bibitem{titsias2009variational}
\begin{binproceedings}[author]
\bauthor{\bsnm{Titsias},~\bfnm{Michalis}\binits{M.}}
(\byear{2009}).
\btitle{Variational learning of inducing variables in sparse Gaussian
  processes}.
In \bbooktitle{Artificial intelligence and statistics}
\bpages{567--574}.
\bpublisher{PMLR}.
\end{binproceedings}
\endbibitem

\bibitem{travis2024pointwise}
\begin{barticle}[author]
\bauthor{\bsnm{Travis},~\bfnm{Luke}\binits{L.}} \AND
  \bauthor{\bsnm{Ray},~\bfnm{Kolyan}\binits{K.}}
(\byear{2024}).
\btitle{Pointwise uncertainty quantification for sparse variational Gaussian
  process regression with a Brownian motion prior}.
\bjournal{Advances in Neural Information Processing Systems}
\bvolume{36}.
\end{barticle}
\endbibitem

\bibitem{Triebel2008EMS}
\begin{bbook}[author]
\bauthor{\bsnm{Triebel},~\bfnm{Hans}\binits{H.}}
(\byear{2008}).
\btitle{Function spaces and wavelets on domains}.
\bseries{EMS Tracts in Mathematics}
\bvolume{7}.
\bpublisher{European Mathematical Society (EMS), Z\"{u}rich}.
\bdoi{10.4171/019}
\bmrnumber{2455724}
\end{bbook}
\endbibitem

\bibitem{troltzsch2010optimal}
\begin{bbook}[author]
\bauthor{\bsnm{Tr\"{o}ltzsch},~\bfnm{Fredi}\binits{F.}}
(\byear{2010}).
\btitle{Optimal control of partial differential equations}.
\bseries{Graduate Studies in Mathematics}
\bvolume{112}.
\bpublisher{American Mathematical Society, Providence, RI}
\bnote{Theory, methods and applications, Translated from the 2005 German
  original by J\"{u}rgen Sprekels}.
\bdoi{10.1090/gsm/112}
\bmrnumber{2583281}
\end{bbook}
\endbibitem

\bibitem{vdVWellner}
\begin{bbook}[author]
\bauthor{\bparticle{van~der} \bsnm{Vaart},~\bfnm{Aad~W.}\binits{A.~W.}} \AND
  \bauthor{\bsnm{Wellner},~\bfnm{Jon~A.}\binits{J.~A.}}
(\byear{1996}).
\btitle{Weak convergence and empirical processes}.
\bseries{Springer Series in Statistics}.
\bpublisher{Springer-Verlag, New York}
\bnote{With applications to statistics}.
\bdoi{10.1007/978-1-4757-2545-2}
\bmrnumber{1385671}
\end{bbook}
\endbibitem

\bibitem{vdVWellner2nd}
\begin{bbook}[author]
\bauthor{\bparticle{van~der} \bsnm{Vaart},~\bfnm{Aad~W.}\binits{A.~W.}} \AND
  \bauthor{\bsnm{Wellner},~\bfnm{Jon~A.}\binits{J.~A.}}
(\byear{2023}).
\btitle{Weak convergence and empirical processes}.
\bseries{Springer Series in Statistics}.
\bpublisher{Springer-Verlag, New York}
\bnote{With applications to statistics, Second Edition}.
\end{bbook}
\endbibitem

\bibitem{Yan2020thesis}
\begin{bbook}[author]
\bauthor{\bsnm{Yan},~\bfnm{Dong}\binits{D.}}
(\byear{2020}).
\btitle{Bayesian Inference for Gaussian Models}.
\bseries{PhD thesis}.
\bpublisher{University of Leiden, https://hdl.handle.net/1887/86070}.
\end{bbook}
\endbibitem

\bibitem{Yan2024}
\begin{barticle}[author]
\bauthor{\bsnm{Yan},~\bfnm{Dong}\binits{D.}},
  \bauthor{\bsnm{Gugushvili},~\bfnm{Shota}\binits{S.}} \AND
  \bauthor{\bparticle{van~der} \bsnm{Vaart},~\bfnm{Aad}\binits{A.}}
(\byear{2024}).
\btitle{Bayesian Linear Inverse Problems in Regularity Scales with Discrete
  Observations}.
\bjournal{Sankhya A}.
\bdoi{10.1007/s13171-024-00342-0}
\end{barticle}
\endbibitem

\bibitem{Zhao2012}
\begin{bbook}[author]
\bauthor{\bsnm{Zhao},~\bfnm{Zhongxin}\binits{Z.}} \AND
  \bauthor{\bsnm{Chung},~\bfnm{Kai~L}\binits{K.~L.}}
(\byear{2012}).
\btitle{From Brownian Motion to Schrodinger's Equation}.
\bseries{Grundlehren der mathematischen Wissenschaften}
\bvolume{312}.
\bpublisher{Springer}.
\end{bbook}
\endbibitem

\end{thebibliography}

\end{document}